\documentclass[12pt]{article}
\usepackage{graphicx}
\usepackage{amsmath,amsthm,amssymb,enumerate}
\usepackage{euscript,mathrsfs}
\usepackage[left=2cm,right=2cm,top=3.5cm,bottom=3.5cm]{geometry}
\usepackage{color}
\catcode`\@=11 \@addtoreset{equation}{section}

\catcode`\@=12

\usepackage[most]{tcolorbox}
\definecolor{beg}{RGB}{220,220,220} 
\def\bcol{\begin{tcolorbox}[enhanced jigsaw,colback=beg,boxrule=0pt,arc=0pt]}
\def\ecol{\end{tcolorbox}}

\usepackage{accents}
\newcommand*{\tecka}[1]{%
  \accentset{\mbox{\large\bfseries .}}{#1}}

\definecolor{bg}{RGB}{255,249,227}

\allowdisplaybreaks

\newcommand\tr{\operatorname{tr}}

\newtheorem{Theorem}{Theorem}[section]
\newtheorem{Proposition}[Theorem]{Proposition}
\newtheorem{Lemma}[Theorem]{Lemma}
\newtheorem{Corollary}[Theorem]{Corollary}

\theoremstyle{definition}
\newtheorem{Definition}[Theorem]{Definition}

\newtheorem{Remark}[Theorem]{Remark}

\newcommand{\bTheorem}[1]{
\begin{Theorem} \label{T#1} }
\newcommand{\eT}{\end{Theorem}}

\newcommand{\bProposition}[1]{
\begin{Proposition} \label{P#1}}
\newcommand{\eP}{\end{Proposition}}

\newcommand{\bLemma}[1]{
\begin{Lemma} \label{L#1} }
\newcommand{\eL}{\end{Lemma}}

\newcommand{\bCorollary}[1]{
\begin{Corollary} \label{C#1} }
\newcommand{\eC}{\end{Corollary}}

\newcommand{\bRemark}[1]{
\begin{Remark} \label{R#1} }
\newcommand{\eR}{\end{Remark}}

\newcommand{\bDefinition}[1]{
\begin{Definition} \label{D#1} }
\newcommand{\eD}{\end{Definition}}

\newcommand{\bs}{\vc{s}}

\newcommand{\vje}{\mathbf{j}_{e}}
\newcommand{\vjtheta}{\mathbf{j}_{\theta}}
\newcommand{\vjb}{\mathbf{j}_{b}}
\newcommand{\vjeta}{\mathbf{j}_{\eta}}

\newcommand{\bfb}{\mathbf{b}}

\newcommand{\bn}{\mathbf{n}}

\def\diver{\mathop{\mathrm{div}}\nolimits}

\newcommand{\eps}{\varepsilon}
\newcommand{\dd}{\,\mathrm{d}}

\newcommand{\bFormula}[1]{
\begin{equation} \label{#1}}
\newcommand{\eF}{\end{equation}}

\newcommand{\vn}{\vc{n}}

\newcommand{\vc}[1]{{\bf #1}}

\newcommand{\Div}{\, {\rm div}_x}
\newcommand{\Grad}{\nabla_x}

\newcommand{\dx}{\,{\rm d} {x}}
\newcommand{\dS}{\,{\rm d} {S}}

\newcommand{\dt}{\,{\rm d} t }

\newcommand{\vv}{\vc{v}}
\newcommand{\vw}{\vc{w}}

\newcommand{\R}{\mathbb{R}}

\definecolor{Cgrey}{rgb}{0.85,0.85,0.85}
\definecolor{Cblue}{rgb}{0.50,0.85,0.85}
\definecolor{Cred}{rgb}{1,0,0}
\definecolor{fancy}{rgb}{0.10,0.85,0.10}

\newcommand\Cbox[2]{%
    \newbox\contentbox%
    \newbox\bkgdbox%
    \setbox\contentbox\hbox to \hsize{%
        \vtop{
            \kern\columnsep
            \hbox to \hsize{%
                \kern\columnsep%
                \advance\hsize by -2\columnsep%
                \setlength{\textwidth}{\hsize}%
                \vbox{
                    \parskip=\baselineskip
                    \parindent=0bp
                    #2
                }%
                \kern\columnsep%
            }%
            \kern\columnsep%
        }%
    }%
    \setbox\bkgdbox\vbox{
        \color{#1}
        \hrule width  \wd\contentbox %
               height \ht\contentbox %
               depth  \dp\contentbox
        \color{black}
    }%
    \wd\bkgdbox=0bp%
    \vbox{\hbox to \hsize{\box\bkgdbox\box\contentbox}}%
    \vskip\baselineskip%
}

\usepackage{todonotes}

\begin{document}


\title{\Large On incompressible heat-conducting viscoelastic rate-type fluids with stress-diffusion and purely spherical elastic response}

\author{Miroslav Bul\'{\i}\v{c}ek$^{*}$, Josef M\'{a}lek$^{*}$, V\'\i{}t Pr\accent23u\v{s}a$^{*}$, and Endre S\"uli$^{**}$}




\maketitle

\bigskip

\centerline{$^{*}$ Charles University, Faculty of Mathematics and Physics, Mathematical Institute}

\centerline{Sokolovsk\'a 83, 186 75 Prague 8, Czech Republic}

\centerline{}

\centerline{$^{**}$ University of Oxford, Mathematical Institute, }

\centerline{Radcliffe Observatory Quarter, Woodstock Road, Oxford OX2 6GG, United Kingdom}

\bigskip

\begin{abstract}

  We prove the existence of large-data global-in-time weak solutions to an evolutionary PDE system describing flows of incompressible \emph{heat-conducting} viscoelastic rate-type fluids with stress-diffusion, subject to a stick-slip boundary condition for the velocity and a homogeneous Neumann boundary condition for the extra stress tensor. In the introductory section we develop the thermodynamic foundations of the proposed model, and we document the role of thermodynamics in obtaining critical structural relations between the quantities of interest. These structural relations are then exploited in the mathematical analysis of the governing equations. In particular, the definition of weak solution is motivated by the thermodynamic basis of the model.
  The extra stress tensor describing the elastic response of the fluid is in our case purely spherical, which is a simplification from the physical point of view. The model nevertheless exhibits features that require novel mathematical ideas in order to deal with the technically complex structure of the associated internal energy and the more complicated forms of the corresponding entropy and energy fluxes. The paper provides the first rigorous proof of the existence of large-data global-in-time weak solutions to the governing equations for \emph{coupled thermo-mechanical processes} in viscoelastic rate-type fluids.
\end{abstract}

temperature-dependent material moduli, weak solution, global-in-time existence



\section{Introduction}
\label{s1}

We develop a rigorous mathematical theory of existence of large-data global weak solutions to the system of evolutionary nonlinear partial differential equations governing the motion of a class of incompressible \emph{heat-conducting} viscoelastic rate-type fluids with stress-diffusion. In particular, we deal with the \emph{temperature evolution equation coupled with the governing equations for a viscoelastic rate-type fluid}.

A simple system of evolution equations that is a special case of the general system investigated later is the following\footnote{The notation in~\eqref{eq:governing-equations-introduction} is the standard one, and it is fully explained in Section~\ref{s2}.}
,
\begin{subequations}
  \label{eq:governing-equations-introduction}
  \begin{align}
    \label{eq:1}
    \Div \vv &=0, \\
    \label{eq:2}
    \varrho \tecka{\vv} &= - \Grad p + \Div \mathbb{T}_{\delta} + \varrho \bfb, \\
    \label{eq:3}
    \mathbb{T}_{\delta} &= 2 \nu \mathbb{D}(\vv),
    \\
    \label{eq:4}
    \nu_1 \tecka{b} &= \Div(\tilde{\mu}  \Grad b) -  \mu \left(b - 1\right)
    \\
    \label{eq:5}
    \varrho c_V \tecka{\theta} &=
                                 \Div \left( \kappa \Grad \theta \right)
                                 +
                                 2 \nu \mathbb{D}(\vv):\mathbb{D}(\vv)
                                 +
                                 \frac{3\mu^2}{2\nu_1}
                                 \left(
                                 b
                                 +
                                 \frac{1}{b}
                                 -
                                 2
                                 \right)
                                 +
                                 \frac{3\mu \tilde{\mu}}{2\nu_1}
                                 \Grad b \cdot \Grad b
                                 .
  \end{align}
\end{subequations}
Later on in the paper we shall allow all material parameters to be, among other things, temperature-dependent, and the right-hand side of~\eqref{eq:4} and~\eqref{eq:5} will be allowed to contain various nonlinear terms. However, for the sake of ease of exposition in these introductory remarks we shall confine ourselves to~\eqref{eq:governing-equations-introduction}. The system~\eqref{eq:governing-equations-introduction} is a simplified version of the system\footnote{Here $\text{Sym}$ denotes the symmetric part of the corresponding term, that is
  $
  \text{Sym}
  \left(
    \mathbb{B}_{\kappa_{p(t)}}^{-\frac{1}{2}}
    \left(
      \Grad \mathbb{B}_{\kappa_{p(t)}}^{\frac{1}{2}}
    \right)
  \right)
  =
  \frac{1}{2}
  \left(
    \mathbb{B}_{\kappa_{p(t)}}^{-\frac{1}{2}}
    \left(
      \Grad \mathbb{B}_{\kappa_{p(t)}}^{\frac{1}{2}}
    \right)
    +
    \left(
      \Grad \mathbb{B}_{\kappa_{p(t)}}^{\frac{1}{2}}
    \right)
    \mathbb{B}_{\kappa_{p(t)}}^{-\frac{1}{2}}
  \right)
  $, where we use the notation $\left[ \left(\Grad \mathbb{U}\right) \mathbb{V} \right]_{ij} = \left(\frac{\partial U_{ik}}{\partial x}\right) V_{kj}$.}
\begin{subequations}
  \label{eq:governing-equations-Oldroyd-B-stress-diffusion-full}
  \begin{align}
    \label{eq:6}
    \Div \vv &=0,
    \\
    \label{eq:7}
    \varrho \tecka{\vv} &= - \Grad p + \Div \mathbb{T}_{\delta} + \varrho \bfb,
    \\
    \label{eq:8}
    \mathbb{T}_{\delta} &= 2 \nu \mathbb{D}(\vv) + \mu \left(\mathbb{B}_{\kappa_{p(t)}}\right)_{\delta},
    \\
    \label{eq:9}
    \nu_1
    \accentset{\triangledown}{\mathbb{B}}_{\kappa_{p(t)}}
             &=
               -
               \mu \left(\mathbb{B}_{\kappa_{p(t)}} - \mathbb{I}  \right)
               +
               \Div \left( \tilde{\mu} \Grad \mathbb{B}_{\kappa_{p(t)}} \right)
               ,
    \\
    \label{eq:10}
    \varrho
    c_V
    \tecka{\theta}
             &=
               \Div \left(\kappa \Grad \theta \right)
               +
               2 \nu \mathbb{D}(\vv):\mathbb{D}(\vv)
               +
               \frac{\mu^2}{2\nu_1}
               \left(
               \tr \mathbb{B}_{\kappa_{p(t)}}
               +
               \tr \mathbb{B}_{\kappa_{p(t)}}^{-1}
               -
               6
               \right)
    \\
    \notag
             &\qquad\qquad +
               2
               \frac{\mu \tilde{\mu}}{\nu_1}
               \tr
               \left[
               \text{Sym}
               \left(
               \mathbb{B}_{\kappa_{p(t)}}^{-\frac{1}{2}}
               \left(
               \Grad \mathbb{B}_{\kappa_{p(t)}}^{\frac{1}{2}}
               \right)
               \right)
               \cdot
               \text{Sym}
               \left(
               \mathbb{B}_{\kappa_{p(t)}}^{-\frac{1}{2}}
               \left(
               \nabla \mathbb{B}_{\kappa_{p(t)}}^{\frac{1}{2}}
               \right)
               \right)
               \right]
               .
  \end{align}
\end{subequations}
The first part of this system, that is equations~\eqref{eq:6}--\eqref{eq:9}, constitute the governing equations for the flow of an incompressible Oldroyd-B fluid with a so-called \emph{stress-diffusion} term. The nomenclature \emph{stress-diffusion} refers to the last term on the right-hand side of~\eqref{eq:9}. This term is not present in the original Oldroyd-B model, see~\cite{oldroyd.jg:on}, but it is frequently added to various viscoelastic models, see in particular~\cite{el-kareh.aw.leal.lg:existence}, and numerous other works in the field of rheology of complex fluid, see for example~\cite{cates.me.fielding.sm:rheology}, \cite{subbotin.av.malkin.ay.ea:self-organization} or \cite{fardin.m.radulescu.o.ea:stress}, or computational fluid dynamics, see for example~\cite{sureshkumar.r.beris.an:effect}, \cite{thomases.b:analysis}, \cite{chupin.l.ichim.a.ea:stationary} or~\cite{biello.ja.thomases.b:equilibrium}; further references can be found in~\cite{malek.j.prusa.v.ea:thermodynamics} and~\cite{dostalk.m.prusa.v.ea:on}.

The last equation in~\eqref{eq:governing-equations-Oldroyd-B-stress-diffusion-full} is the evolution equation for temperature, and it is derived by the method proposed in~\cite{malek.j.prusa.v.ea:thermodynamics}; see also~\cite{dostalk.m.prusa.v.ea:on}. (In fact \cite{malek.j.prusa.v.ea:thermodynamics} is the first paper that deals, from the thermodynamic point of view, with a temperature evolution equation in the context of viscoelastic rate-type fluids with stress-diffusion.) This equation provides the missing piece of information that complements the evolution equations~\eqref{eq:6}--\eqref{eq:9} for the mechanical quantities, and that in turn allows one to describe coupled thermo-mechanical response of the given fluid. \emph{No restriction to isothermal processes is necessary.} System~\eqref{eq:governing-equations-Oldroyd-B-stress-diffusion-full} can be shown to be consistent with the laws of thermodynamics, and the underlying thermodynamic concepts can be exploited in the mathematical theory.

System~\eqref{eq:governing-equations-introduction} can be seen as the special case of~\eqref{eq:governing-equations-Oldroyd-B-stress-diffusion-full}, where the tensorial quantity $\mathbb{B}_{\kappa_{p(t)}}$ has a simple spherical structure, that is $\mathbb{B}_{\kappa_{p(t)}} = b\mathbb{I}$, where $\mathbb{I}$ denotes the identity tensor. While the reduction of the tensor field $\mathbb{B}_{\kappa_{p(t)}}$ to a scalar field $b$ simplifies the mechanical part of the system, the temperature evolution equation is still part of the system, and if the material parameters (e.g. $\tilde{\mu}$ or $\nu$) depend on the temperature, then the temperature equation is fully coupled to the evolution equations for the mechanical quantities. The analysis of this coupled system therefore still constitutes a challenging mathematical problem.

The paper is structured as follows. A class of simplified viscoelastic rate-type models with stress-diffusion and the spherical tensorial quantity $\mathbb{B}_{\kappa_{p(t)}} = b\mathbb{I}$ is briefly derived in Section~\ref{s2}; the reader is referred to \cite{bulvcek.m.malek.j.ea:pde-analysis}, \cite{malek.j.prusa.v.ea:thermodynamics} and~\cite{dostalk.m.prusa.v.ea:on} for additional details. While our approach is based on the developments presented in~\cite{bulvcek.m.malek.j.ea:pde-analysis}, \cite{malek.j.prusa.v:derivation}, \cite{malek.j.rajagopal.kr.ea:on}, \cite{malek.j.prusa.v.ea:thermodynamics} and \cite{rajagopal.kr.srinivasa.ar:thermodynamic}, from the very beginning we consider only the spherical tensor~$\mathbb{B}_{\kappa_{p(t)}} = b \mathbb{I}$, which substantially simplifies the necessary algebraic manipulations.

Following the derivation of the class of models under consideration, we proceed with its mathematical analysis. Regarding mathematical analysis, a few comments are in order. Recently, a similar class of models was analysed in \cite{bulvcek.m.malek.j.ea:pde-analysis}, the key difference between the present work and \cite{bulvcek.m.malek.j.ea:pde-analysis} being that the presence of stress-diffusion here is \emph{due to irreversible (dissipative) effects, not to energy-storing mechanisms}, as considered in \cite{bulvcek.m.malek.j.ea:pde-analysis}. Consequently, in contrast with~\cite{bulvcek.m.malek.j.ea:pde-analysis}, the models under consideration here contain no Korteweg-type contribution to the stress tensor. Furthermore, the analysis pursued in \cite{bulvcek.m.malek.j.ea:pde-analysis} rests on the assumption that all processes of relevance are isothermal, whereas the \emph{present paper is devoted to the technically more challenging non-isothermal case}. The emphasis on the non-isothermal case also distinguishes our contribution from various studies concerning the solvability of the governing equations for viscoelastic rate-type fluids with a stress-diffusion term; see for example~\cite{el-kareh.aw.leal.lg:existence}, \cite{constantin.p.kliegl.m:note}, \cite{barrett.jw.boyaval.s:existence}, \cite{kreml.o.pokorny.m.ea:on}, \cite{chupin.l.martin.s:stationary}, \cite{lukacov-medvidova.m.notsu.h.ea:energy}, \cite{lukacova-medvidova.m.mizerova.h.ea:global},
\cite{barrett.jw.lu.y.suli.e:oldroyd-b}, \cite{barrett.jw.suli.e:existence*6} or~\cite{bathory.m.bulcek.m.ea:large}.

The first step in the mathematical analysis of the governing equations is the decision regarding the choice of a particular form of the evolution equation for the thermal variables. If one is concerned with strong solutions, then one can track the evolution of either the net total energy or the temperature field alone, and the two approaches are equivalent. On the other hand, when one works with a weak solution, then the two approaches are no longer equivalent, and one of them must be chosen. We use the approach based on the net total energy. (Note that the net total energy is also a primary concept from the physical point of view. The concept of internal energy is in this perspective a derived one.) This choice leads to a formulation where all nonlinearities appear in divergence form, which is then conveniently exploited in the definition of weak solution. We note in this respect that the concept of weak solution based on the notion of net total energy balance has been also gainfully exploited in the context of the Navier--Stokes--Fourier system, see \cite{feireisl.e.malek.j:on} and the subsequent publications~\cite{bulcek.m.feireisl.e.ea:navier-stokes-fourier}, \cite{bulcek.m.malek.j.ea:mathematical*1} and~\cite{bulcek.m.malek.j:internal}, as well as in the context of the governing equations of other complex liquids, see for example~\cite{feireisl.e.fremond.m.ea:new}.

Once the choice of the evolution equation based on the net total energy has been made, the proof of the existence of a large-data global weak solution proceeds by a three-level Galerkin approximation, described in Section~\ref{s3}, and careful passages to the limits in the sequences of Galerkin approximations to the velocity field, the internal energy, and the spherical part of the Cauchy stress tensor using various weak compactness techniques.

We close the paper with concluding remarks and a summary of the main results of the paper.

\section{Governing equations and thermodynamic underpinnings} 
\label{s2}

\subsection{Governing equations}

The basic system of governing partial differential equations in incompressible fluid dynamics stems from the balance equations for mass, linear and angular momenta and energy, and from the mathematical formulation of the second law of thermodynamics. When written in their local forms and assuming that the positive density $\varrho$ is uniform in space and time, these governing equations take the form
\begin{align}
\Div \vv &=0,\quad \varrho \tecka{\vv} = - \Grad p + \Div \mathbb{T}_{\delta} + \varrho \bfb, \qquad \mathbb{T}=\mathbb{T}^{\rm T}, \qquad \varrho \tecka{E} = \Div(\mathbb{T}\vv - \vje) + \varrho \bfb\cdot \vv, \label{A1} \\
\varrho\tecka{\eta} &= \zeta - \Div \vjeta \qquad \textrm{ with } \zeta \ge 0, \label{A2}
\end{align}
where $\vv$ is the velocity, $\mathbb{T}$ is the Cauchy stress tensor, $\mathbb{T}_{\delta} := \mathbb{T} - \frac{1}{3} \left( \tr \mathbb{T} \right) \mathbb{I}$ is its deviatoric (traceless) part, $-p:=\frac13 \tr \mathbb{T}$ is the mean normal stress, $E := \tfrac12 |\vv|^2 + e$ stands for the specific total energy (that is, the sum of the specific kinetic energy and the specific internal energy $e$), $\eta$ is the specific entropy, $\vje$ is the energy (heat) flux, $\vjeta$ is the entropy flux and $\zeta$ represents the entropy production that is supposed to be, in accordance with the second law of thermodynamics, nonnegative; $\bfb$ stands for the given external force. The superscript $\stackrel{\tecka{~}}{\vspace{-5cm}~}$ signifies the material derivative of a given physical quantity; to be more precise,  $\tecka{q} := \frac{\partial q}{\partial t} + \vv \cdot \Grad q$ denotes the material derivative of $q$. Note that it follows from the second and the last equation in \eqref{A1} that the evolution equation for the internal energy takes the form
\begin{equation}
  \label{A1a}
\varrho \tecka{e} = \mathbb{T}: \mathbb{D}(\vv) - \Div \vje,
\end{equation}
where $\mathbb{D}(\vv)$ denotes the symmetric part of the velocity gradient, $\mathbb{D}(\vv) := \frac{1}{2}\left( \Grad \vv + \left(\Grad \vv\right)^T\right)$.

The particular governing equations for the given substance then follow from the specification of energy storage mechanisms in the given substance and the specification of entropy producing mechanisms in the given substance. The specification of the energy storage mechanisms is tantamount to the specification of the \emph{Helmholtz free energy} $\psi$, while the specification of the entropy producing mechanisms is tantamount to the specification of \emph{entropy production} $\zeta$. (Whenever convenient other thermodynamic potentials such as the internal energy or Gibbs free energy may also be used instead of the Helmholtz free energy; see, for example, \cite{rajagopal.kr.srinivasa.ar:gibbs-potential-based}, \cite{srinivasa.ar:on} or \cite{prusa.v.rajagopal.kr.ea:gibbs}.) These quantities will be of interest in the mathematical theory as well, hence we show how to derive~\eqref{eq:governing-equations-introduction} and similar models from the knowledge of the general balance equations~\eqref{A1} and formulae for the Helmholtz free energy $\psi$ and the entropy production~$\zeta$.

Following the approach proposed in~\cite{rajagopal.kr.srinivasa.ar:thermodynamic}, \cite{malek.j.rajagopal.kr.ea:on}, \cite{hron.j.milos.v.ea:on} and~\cite{malek.j.prusa.v.ea:thermodynamics}, the visco-\emph{elastic} response of the given fluid is incorporated in the model via the dependence of the internal energy on the tensorial quantity~$\mathbb{B}_{\kappa_{p(t)}}$ that encodes the elastic response of the fluid\footnote{For alternative approaches to standard viscoelastic rate-type fluids \emph{without} stress-diffusion see especially~\cite{peters.gwm.baaijens.fpt:modelling}, \cite{wapperom.p.hulsen.ma:thermodynamics}, \cite{dressler.m.edwards.bj.ea:macroscopic} or \cite{guaily.a:temperature}. In particular~\cite{dressler.m.edwards.bj.ea:macroscopic} contains a rich bibliography on the subject matter. See also~\cite{pavelka.m.klika.v.ea:multiscale} for a recent discussion of viscoelastic rate-type models in the context of the GENERIC framework.}. (The quantity $\mathbb{B}_{\kappa_{p(t)}}$ is the left Cauchy--Green tensor associated with the elastic response from the so-called natural configuration to the current configuration; see in particular~\cite{malek.j.prusa.v:derivation} and~\cite{malek.j.prusa.v.ea:thermodynamics} for details.) In our case we keep the tensorial character of $\mathbb{B}_{\kappa_{p(t)}}$, but we assume that it is purely spherical, that is
\begin{equation}
\mathbb{B}_{\kappa_{p(t)}} = b \mathbb{I}. \label{aA7}
\end{equation}
In order to abbreviate and simplify the derivation of the model, we can assume that $b$ is just an ``internal variable'', which enters the constitutive equations for the internal energy $e$. (This allows us to abstain from in-depth discussion of the physical meaning of $\mathbb{B}_{\kappa_{p(t)}}$. On the other hand, the fact that $\mathbb{B}_{\kappa_{p(t)}}$ has a clear physical meaning is very useful in the development of new models for complex fluids.) This internal variable is assumed to be governed by the evolution equation
\begin{equation}
\tecka{b} = \Div(\alpha(\theta, b) \Grad b) - h(\theta, b), \label{eq_b}
\end{equation}
where $\alpha$ is a positive and continuous function of $\theta$ and $b$, and the function $h$ is a function to be specified later. Since only $b$ and not $\Grad b$ is considered in the constitutive equation for the internal energy, the stress-diffusion effect has, in the given class of models investigated below, no influence on the energy storage mechanisms. This means that stress-diffusion is interpreted as a purely dissipative mechanism; see~\cite{malek.j.prusa.v.ea:thermodynamics} and~\cite{bulvcek.m.malek.j.ea:pde-analysis} for a discussion. Moreover, the dissipative nature of the stress-diffusion term corresponds to the microscopic interpretation of the stress-diffusion term in~\cite{el-kareh.aw.leal.lg:existence}. Indeed, in~\cite{el-kareh.aw.leal.lg:existence} the stress-diffusion term is related to the ``hydrodynamic resistance of the dumbbell beads'', and the term ``originates in the Brownian motion of the center-of-mass of the dumbbell''.

Once we have identified the quantities that are relevant with respect to the energy storage mechanisms, we introduce the specific Helmholtz free energy $\psi$ through $\psi:= e- \theta \eta$, and we assume that~$\psi$ is a function of~$\theta$ and $b$ that has the following structure
\begin{equation}
\psi = \tilde{\psi}(\theta, b):=\tilde{\psi}_0(\theta) + \frac{1}{\varrho}\tilde{\psi}_1(\theta) \tilde{\psi}_2(b),\label{psi_ass}
\end{equation}
where $\tilde{\psi}_0$ and $\tilde{\psi}_1$ are defined and continuous on $[0,\infty)$, $\tilde{\psi}_2$ is defined and continuous on $(0,\infty)$, and they all belong to $\mathcal{C}^2((0,\infty))$ and satisfy suitable additional assumptions specified below. The decomposition is motivated by the special case we have in mind,
\begin{equation}
\psi = \tilde{\psi}(\theta, b) := - c_V\theta \left(\ln \left(\frac{\theta}{\theta_{\textrm{ref}}}\right) - 1 \right) + \frac{1}{\varrho}\tilde{\psi}_1(\theta) ( b-1 - \ln b),\label{A10}
\end{equation}
where $c_V>0$ denotes the specific heat at constant volume, $\theta_{\textrm{ref}} >0$ is a reference temperature, and $\tilde{\psi}_1(\theta) \ge C_0 >0$ for all $\theta \ge 0$ is a temperature-dependent shear modulus. 
The last term in the Helmholtz free energy formula~\eqref{A10} is motivated by the expression $\tr \mathbb{B}_{\kappa_{p(t)}} - 3 - \ln \det \mathbb{B}_{\kappa_{p(t)}}$, which is an expression that appears in the Helmholtz free energy for a class of viscoelastic fluids\footnote{Regarding this expression, see the early considerations in~\cite{sarti.gc.marrucci.g:thermomechanics}, 
  and a detailed discussion in~\cite{dressler.m.edwards.bj.ea:macroscopic}. For its mathematical underpinnings see~\cite{hu.d.lelievre.t:new} and \cite{boyaval.s.lelievre.t.ea:free-energy-dissipative}. A list of formulae for Helmholtz free energy behind various classical viscoelastic rate-type models is also given in~\cite{dostalk.m.prusa.v.ea:on}.}
.
The purely thermal part, that is the first term on the right-hand side of~\eqref{A10}, would lead, if considered separately, to the classical heat conduction equation with the constant specific heat.

We note that, as in \eqref{psi_ass} and \eqref{A10}, all particular functions of the temperature $\theta$ and $b$ will appear hereafter with the superscript tilde, while general physical quantities will be written without the superscript tilde.  The basic thermodynamic relations give that
\begin{equation}
\eta =\tilde{\eta}(\theta,b):= - \frac{\partial \tilde{\psi}(\theta,b)}{\partial \theta},\qquad e = \tilde{e}(\theta,b):=\theta \tilde{\eta}(\theta,b)+\tilde{\psi}(\theta,b). \label{eq_eta}
\end{equation}
Hence we see that, by taking the material derivative of $e-\theta \eta = \tilde{\psi}(\theta, b)$ and using~\eqref{psi_ass}, we arrive at
\begin{equation}
  \varrho \tecka{e} - \varrho \theta \tecka{\eta} = \varrho \frac{\partial \tilde{\psi}}{\partial b}\tecka{b} = \tilde{\psi}_1(\theta) \tilde{\psi}'_2(b) \tecka{b}. \label{eq_ent}
\end{equation}
(Recall that the density $\varrho$ has been assumed to be a positive constant.) Since we have a formula for the time derivative of $b$, see~\eqref{eq_b}, a formula for the time derivative of $e$, see~\eqref{A1a}, and a formula for the time derivative of $\eta$, see~\eqref{A2}, we obtain the following equation for the entropy production~$\zeta$ appearing in~\eqref{A2}:
\begin{equation}
  \begin{aligned}
    \label{eq:11}
\zeta&=\rho \tecka{\eta} + \Div \vjeta =\rho\frac{\tecka{e}}{\theta} - \rho\frac{\partial \tilde{\psi}}{\partial b}\frac{\tecka{b}}{\theta}+\Div \vjeta\\
 &=\frac{\mathbb{T}: \mathbb{D}(\vv) - \Div \vje}{\theta} - \rho \frac{\partial \tilde{\psi}}{\partial b}\frac{\Div(\alpha(\theta, b) \Grad b) - h(\theta, b)}{\theta}+\Div \vjeta\\
 &=\frac{\mathbb{T}: \mathbb{D}(\vv)}{\theta} + \rho\frac{\partial \tilde{\psi}}{\partial b}\frac{h(\theta, b)}{\theta} - \rho \frac{\partial \tilde{\psi}}{\partial b}\frac{\Div(\alpha(\theta, b) \Grad b)}{\theta}-\frac{\Div \vje}{\theta}+\Div \vjeta.
\end{aligned}
\end{equation}
Now the objective is to manipulate the right-hand side in such a way that all of the divergence terms are grouped together; that is, we want to obtain all divergence terms by the application of the divergence operator to a single quantity. This will help us to identify the fluxes. We see that
\begin{equation}
  \begin{aligned}
    \zeta
&=\frac{\mathbb{T}: \mathbb{D}(\vv)}{\theta} + \rho \frac{\partial \tilde{\psi}}{\partial b}\frac{h(\theta, b)}{\theta} -\Div\left(\rho \frac{\partial \tilde{\psi}}{\partial b}\frac{\alpha(\theta, b) \Grad b}{\theta}+\frac{\vje}{\theta}-\vjeta\right)\\
 &\quad+\rho \frac{\partial^2 \tilde{\psi}}{\partial b^2}\frac{\alpha(\theta, b)}{\theta} |\Grad b|^2  +\frac{1}{\theta^2}\left[ \rho \left( \theta\frac{\partial^2 \tilde{\psi}}{\partial \theta\, \partial b} -\frac{\partial \tilde{\psi}}{\partial b} \right)\alpha(\theta, b) \Grad b-\vje\right]\cdot \Grad \theta \\
 &=\frac{\mathbb{T}: \mathbb{D}(\vv)}{\theta} + \rho \frac{\partial \tilde{\psi}}{\partial b}\frac{h(\theta, b)}{\theta} -\Div\left(\rho\frac{\partial \tilde{\psi}}{\partial b}\frac{\alpha(\theta, b) \Grad b}{\theta}+\frac{\vje}{\theta}-\vjeta\right)\\
 &\quad+\rho \frac{\partial^2 \tilde{\psi}}{\partial b^2}\frac{\alpha(\theta, b)}{\theta} |\Grad b|^2  -\frac{1}{\theta^2}\left[\rho \frac{\partial \tilde{e}(\theta, b)}{\partial b}\alpha(\theta, b) \Grad b+\vje\right]\cdot \Grad \theta.
\end{aligned}\label{A13}
\end{equation}
We proceed by recalling the decomposition of the Helmholtz free energy $\psi$ postulated in \eqref{psi_ass}. In order to guarantee that the entropy production $\zeta$, that is the right-hand side of \eqref{A13}, is nonnegative we require that
\begin{equation}
  \label{eq:12}
\tilde{\psi}''_2(b) \ge 0 \quad \textrm{ for all } b>0, \qquad \tilde{\psi}_1(\theta)\ge 0 \quad \textrm{ for all } \theta\geq 0,
\end{equation}
where the prime symbol $~'$ denotes differentiation with respect to $b$. This guarantees the positivity of the term $\rho \frac{\partial^2 \tilde{\psi}}{\partial b^2}\frac{\alpha(\theta, b)}{\theta} |\Grad b|^2$. (Recall that $\alpha$ is assumed to be positive.) The remaining product terms such as $\frac{\mathbb{T}: \mathbb{D}(\vv)}{\theta}$ must have a structure that guarantees their non-negativity; this helps us to identify a relation between $\mathbb{T}$ and $\mathbb{D}(\vv)$ (cf. (\ref{A14}a) below). Finally, we need the divergence term in~\eqref{A13} to vanish, which yields a relation between the energy flux $\vje$ and the entropy flux $\vjeta$.

Consequently, if we assume that the Cauchy stress tensor~$\mathbb{T}$, the function $h$ in the evolution equation for $b$, the energy flux $\vje$ and the entropy flux $\vjeta$ are given by the formulae
\begin{subequations}
  \label{A14}
  \begin{align}
    \label{eq:15}
    \mathbb{T}_\delta &:= 2\nu(\theta, b) \mathbb{D}(\vv), \\
    \label{eq:16}
    h(\theta, b) &:= C(\theta, b) \tilde{\psi}'_2(b), \\
    \label{eq:14}
\vje &:= - \kappa(\theta,b) \Grad \theta -  \rho \frac{\partial \tilde{e}(\theta,b)}{\partial b}\alpha(\theta, b) \Grad b
       ,
    \\
    \label{eq:13}
\vjeta &:=-\frac{\kappa(\theta,b) \Grad \theta}{\theta}+\rho\frac{\partial^2 \tilde{\psi}(\theta,b)}{\partial \theta \, \partial b}\alpha(\theta, b) \Grad b
         .
\end{align}
\end{subequations}
where the shear viscosity $\nu$, the coefficient $C$ and the heat conductivity $\kappa$ are positive (or at least nonnegative) functions of their variables, then the right-hand side of~\eqref{A13} is indeed non-negative, and we see that the proposed constitutive relations $\eqref{A14}$ conform to the second law of thermodynamics. Note that thanks to~\eqref{psi_ass} we see that the formulae for the fluxes reduce to
\begin{subequations}
  \label{eq:17}
  \begin{align}
    \label{eq:26}
    \vje & = - \kappa(\theta, b) \Grad \theta -  (\tilde{\psi}_1(\theta) - \theta \tilde{\psi}'_1(\theta)) \tilde{\psi}'_2(b) \alpha(\theta, b)\Grad b, \\
    \label{eq:25}
    \vjeta &=- \frac{\kappa (\theta,b) \Grad \theta}{\theta} + \tilde{\psi}'_1(\theta) \tilde{\psi}'_2(b) \alpha(\theta, b) \Grad b.
  \end{align}
\end{subequations}

We then conclude from \eqref{A13} and \eqref{A14} that the formula for the entropy production reduces~to
\begin{equation}
\begin{split}
\theta \zeta &= \frac{\kappa(\theta, b)}{\theta} |\Grad \theta|^2 + 2\nu(\theta, b) |\mathbb{D}(\vv)|^2 + C(\theta, b) \tilde{\psi}_1(\theta) (\tilde{\psi}'_2(b))^2 + \tilde{\psi}_1(\theta) \tilde{\psi}''_2(b) \alpha(\theta, b) |\Grad b|^2.
\end{split}\label{A16}
\end{equation}
Given the requirement \eqref{eq:12} we see that the entropy production is non-negative, and that the second law of thermodynamics is satisfied. (We can also reverse the argument in the sense of the introductory discussion. If we assume entropy production in the form~\eqref{A16}, and if we compare it with the right-hand side of~\eqref{A13}, then we get \eqref{A14}. In this sense the Cauchy stress tensor and the fluxes are determined by the entropy production and the Helmholtz free energy.)

Furthermore, the constitutive equations \eqref{A14} together with the equations \eqref{A1}--\eqref{A1a} lead to the system of governing equations describing the evolution of $\varrho$, $\vv$, $e$, and $b$ that, because of the way in which the system is derived, fulfil the laws of thermodynamics. (In particular, the net total energy in a thermodynamically isolated system is constant, and the net entropy in a thermodynamically isolated system is a nondecreasing function.) If we use the constitutive equations \eqref{A14} in \eqref{A1}, \eqref{A1a} and \eqref{eq_b}, then we arrive at
\begin{subequations}
  \label{A15}
  \begin{align}
    \label{eq:18}
    \Div \vv &= 0 , \\
    \label{eq:19}
    \varrho \tecka{\vv} &=  - \Grad p  + \Div (2 \nu(\theta, b) \mathbb{D} (\vv)) + \varrho \bfb, \\
    \label{eq:20}
    \varrho\tecka{e} &= 2\nu(\theta, b) |\mathbb{D}(\vv)|^2 + \Div(\kappa(\theta, b) \Grad \theta + (\tilde{\psi}_1(\theta) - \theta \tilde{\psi}'_1(\theta)) \tilde{\psi}'_2(b)  \alpha(\theta, b) \Grad b),\\
    \label{eq:21}
\tecka{b} &= \Div(\alpha(\theta, b) \Grad b) - C(\theta, b) \tilde{\psi}'_2(b).
\end{align}
\end{subequations}
This system of equations is a closed set of equations provided that we can find a formula for the temperature $\theta$ in terms of $e$ and $b$.

Our aim is now therefore to explore the circumstances under which one can guarantee the existence of a bivariate function $\theta = \tilde\theta(e,b)$. Note first that as a consequence of \eqref{eq_eta} and the specific formula for the Helmholtz free energy~\eqref{psi_ass}, we have that
\begin{align}
\begin{aligned}\label{Aetae}
\eta = - \frac{\partial \tilde\psi}{\partial \theta} &= - \tilde{\psi}'_0 (\theta) - \frac{1}{\varrho} \tilde{\psi}'_1(\theta) \tilde{\psi}_2(b) =: \tilde\eta(\theta,b), \\
e = \psi + \theta \eta &= (\tilde{\psi}_0(\theta) - \theta \tilde{\psi}'_0(\theta)) + \frac{1}{\varrho} (\tilde{\psi}_1(\theta) - \theta \tilde{\psi}'_1(\theta)) \tilde{\psi}_2(b) =: \tilde{e}(\theta,b);
\end{aligned}
\end{align}
hence
\begin{align}
  \label{eq:22}
  \frac{\partial \tilde{e}(\theta,b)}{\partial \theta} = - \theta \tilde{\psi}_0''(\theta) - \frac{1}{\varrho} \theta \tilde{\psi}_1''(\theta) \tilde{\psi}_2(b) = - \theta \frac{\partial^2 \tilde \psi (\theta,b)}{\partial \theta^2}.
\end{align}
Consequently, if we assume that $\frac{\partial^2 \tilde \psi}{\partial \theta^2} < 0$ for all $\theta >0$ and $b>0$, then it follows that $\tilde{e}$ is strictly increasing, and we can therefore invert $\tilde{e}$ with respect to $\theta$ and consider $\theta$ as a function of $e$ and $b$. (We recall that the expression $- \theta \frac{\partial^2 \tilde \psi (\theta,b)}{\partial \theta^2}$, or equivalently $\frac{\partial \tilde{e}(\theta,b)}{\partial \theta}$, is a particular instance of the general formula $c_V (\theta, \cdot):= -\theta \frac{\partial^2 \psi}{\partial \theta^2}(\theta, \cdot)$ that gives one the specific heat at constant volume~$c_V$ in terms of the derivatives of the Helmholtz free energy/internal energy; see for example~\cite{callen.hb:thermodynamics}. The assumption on the sign of the derivative thus translates to a reasonable requirement on the positivity of the specific heat at constant volume.) Thus under the given assumptions we \emph{deduce}, with $\mathbb{R}_{>0}$ and $\mathbb{R}_{\geq 0}$ signifying the set of all positive and all nonnegative real numbers, respectively, \emph{the existence of} $\tilde{\theta}(e,b)$ with
\begin{equation}
\tilde{\theta} \colon \tilde{e}(\mathbb{R}_{\geq 0}, \mathbb{R}_{>0}) \times \mathbb{R}_{>0} \rightarrow \mathbb{R}_{\geq 0} \; \text{ such that } \; {\tilde \theta}(\tilde{e}(\theta,b), b) = \theta. \label{ex-temp}
\end{equation}
This function can be subsequently used in~\eqref{A15}. The treatment of \eqref{ex-temp} on the level of weak solution will be one of the main difficulties in the proof of existence of a weak solution to the given system of governing equations.

Inserting the formulae for the energy flux~\eqref{eq:14} and the Cauchy stress tensor~\eqref{eq:15} into the evolution equation for the total energy \eqref{A1}$_4$ and the evolution equation for the entropy \eqref{A2}, and recalling that $E:= e+ \tfrac12 |\vv|^2$, we obtain
\begin{equation}
\begin{split} \varrho \tecka{E} &= \Div((-p\mathbb{I} + 2 \nu(\theta, b) \mathbb{D}(\vv))\vv + \kappa (\theta) \Grad\theta) \\
& \qquad + \Div \left( (\tilde{\psi}_1(\theta) - \theta \tilde{\psi}'_1(\theta)) \tilde{\psi}'_2(b)  \alpha(\theta, b) \Grad b\right) + \varrho \bfb\cdot \vv,
\end{split} \label{A16new}
\end{equation}
and
\begin{equation}
\varrho\tecka{\eta} = \zeta + \Div \left(\kappa (\theta, b) \frac{\Grad \theta}{\theta}  - \tilde{\psi}'_1(\theta) \tilde{\psi}'_2(b) \alpha(\theta, b) \Grad b \right),
\label{A17}
\end{equation}
where $\zeta$ is given by \eqref{A16}. Note that equation \eqref{A16new} can be formally obtained by taking the scalar product of~\eqref{eq:19} and $\vv$ and adding the result to~\eqref{eq:20}. It is also worth observing that the equation \eqref{A17} can be obtained from the system \eqref{A15} by multiplying \eqref{eq:20} by $\frac{1}{\theta}$ and subtracting~\eqref{eq:21} multiplied by $\frac{\tilde{\psi}_1(\theta) \tilde{\psi}_2'(b)}{\theta}$, this calculation being motivated by the pair of equalities in~\eqref{eq_ent}.

To summarise, at the level of smooth functions (strong solutions) the system \eqref{A15} is equivalent to the system consisting of the first, the second, and the fourth equation in \eqref{A15}, and \eqref{A16new}. At the level of weak solutions, which we will introduce and study below, this equivalence \emph{holds provided that $\vv$ is an admissible test function in the weak formulation of the equation}~\eqref{eq:19}. This is however not the case in general. From the point of view of the mathematical analysis of the model, it is helpful to start with the equation~\eqref{eq:20} for~$e$ at the level of approximations as this then helps to show that $e$ is nonnegative. However, at a certain point in our analysis we need to invoke the equation \eqref{A16new} for $E$ as all terms that appear on the right-hand side of this equation are in divergence form, which aids passage to the limit in the weak formulation of this equation, while in general, when $|\mathbb{D}(\vv)|^2$ is only integrable, it is unclear how one could deduce that the weak formulation of the equation for $e$ holds. One could of course deduce an inequality instead of an equality in this way, but that would be an unsatisfactory end-result from the physical as well as mathematical point of view.

\subsection{Initial and boundary conditions}
\label{sec:init-bound-cond}

Finally, the system \eqref{A15} is supplemented with \emph{initial conditions and boundary conditions}. We assume that for a $T>0$ and $\Omega \subset \mathbb{R}^3$ being a bounded domain with $\mathcal{C}^{1,1}$-boundary, the following boundary conditions hold on $(0,T)\times \partial \Omega$:
\begin{equation}\label{bc1}
\begin{split}
\vjb:=\alpha(\theta, b) \Grad b\cdot \mathbf{n} &= 0, \\
\vjtheta:=\kappa(\theta,b) \Grad \theta \cdot \mathbf{n} &=0, \\
\vv\cdot\mathbf{n} &= 0,
\end{split}
\end{equation}
and
\begin{equation}\label{bc2}
\left\{\begin{aligned}
\vv_\tau&=\mathbf{0} \ &&\mbox{if and only if} \ |\mathbf{s}|\leq s_*  \ ,\\
 \mathbf{s}&=s_*\frac{\vv_\tau}{|\vv_\tau|}+\gamma_*\vv_\tau \ &&\mbox{if and only if} \ |\mathbf{s}|  > s_* \ ,
\end{aligned}\right.
\end{equation}
where $ \mathbf{s}:=-(\mathbb{T}\mathbf{n})_\tau=-(\mathbb{S}\mathbf{n})_\tau$, $\mathbf{z}_\tau := \mathbf{z} - (\mathbf{z}\cdot \mathbf{n}) \mathbf{n}$ and $\mathbf{n}$ is the unit outward normal vector to~$\partial\Omega$. It is more convenient to rewrite the condition \eqref{bc2} as
an implicitly constituted boundary condition of the form $\mathbf{g}(\mathbf{s}, \vv_{\tau}) = \mathbf{0}$; that is,
\begin{equation} \label{bc2new}
\gamma_* \vv_\tau=\frac{(|\mathbf{s}| - s_*)^{+}}{|\mathbf{s}|}\mathbf{s}.
\end{equation}
The reasons for this choice of boundary condition are twofold. First, it covers several interesting cases, such as slip, Navier's slip and stick-slip. (Slip type boundary conditions are popular in rheology, see for example~\cite{hatzikiriakos.sg:wall}, hence we are not compromising the physical relevance of the boundary condition.) In addition, the frequently used no-slip boundary condition can be seen as an approximation of the physically more realistic stick-slip boundary condition when $|\mathbf{s}|$ is small, in the sense that it is below the threshold value~$s_*$.
Secondly, our choice of boundary condition enables us to show that the mean normal stress $-p$ is an integrable function. This property plays an important role in the proof of our main result, because $p$ appears in the evolution equation for the net total energy \eqref{A16new}.

Finally, we impose the initial conditions
\begin{equation}
\vv(0, \cdot)=\vv_0, \qquad \theta(0,\cdot)={\theta}_0 \quad \textrm{ and } \quad b(0, \cdot) = b_0 \qquad \mbox{in} \ \Omega,\\
\label{ics}
\end{equation}
where the given functions $\vv_0$, ${\theta}_0$ and $b_0$ satisfy, for some $0<b_{\min}\le 1\le b_{\max}<\infty$,
%
\begin{alignat}{2}\label{v0-cond}
\begin{aligned}
&\!\Div \vv_0 = 0 \;\textrm{ in } \Omega, \quad \;\vv_0\cdot\mathbf{n} = 0 \;&&\textrm{ on } \partial \Omega, \\
&0<b_{\textrm{min}} \le b_0 \le b_{\textrm{max}}  &&\textrm{ in } \Omega, \\
&0<\theta_0 &&\textrm{ in }\Omega.
\end{aligned}
\end{alignat}

\subsection{Assumptions on the material functions}

Since the density is assumed to be uniform in time and space, without loss of generality we can suppose that $\varrho\equiv 1$. This will help us to ensure
that various long formulae are more transparent, and \emph{we follow this practice throughout the rest of the paper}.

We recall several observations made in the preceding section, in particular identities~\eqref{A13} and~\eqref{A16}, and we precisely state the assumptions on the material functions. Motivated by the example given in~\eqref{A10} and by the thermodynamic considerations stated above, see especially~\eqref{eq:12} and the discussion following~\eqref{eq:22}, we formulate our structural assumptions on the Helmholtz free energy~$\tilde{\psi}$.
\bcol
\begin{itemize}
\item[\textbf{(A1)}] Let $\psi = \tilde{\psi}(\theta,b)$ be of the form \eqref{psi_ass}, that is
 \begin{equation*}
\tilde{\psi}(\theta, b):=\tilde{\psi}_0(\theta) + \tilde{\psi}_1(\theta) \tilde{\psi}_2(b),
\end{equation*}
where 
$\tilde{\psi}_0$, $\tilde{\psi}_1$ and $\tilde{\psi}_2$ belong to $\mathcal{C}^2 ((0,\infty))$ and satisfy
\[
  \tilde{\psi}_0''(s) <0, \quad \tilde{\psi}_1''(s) \le 0, \quad \tilde{\psi}'_1(s)\ge 0,   \quad \textrm{ and } \quad  \tilde{\psi}_2''(s) \ge 0 \quad \textrm{ for all } s>0.
\]
\end{itemize}
\ecol
Next, we introduce the notation
\begin{equation}\label{e0e1}
\tilde{e}_0(\theta):= \tilde{\psi}_0(\theta) - \theta\tilde{\psi}'_0(\theta) \quad \textrm{ and }
\quad \tilde{e}_1(\theta):= \tilde{\psi}_1(\theta) - \theta\tilde{\psi}'_1(\theta),
\end{equation}
and with the help of \eqref{e0e1} we define
\begin{equation}\label{eeta:def}
\begin{aligned}
  \tilde{\eta}(\theta,b)&:=-\frac{\partial \tilde{\psi}(\theta,b)}{\partial \theta} = -\tilde{\psi}'_0(\theta) -\tilde{\psi}'_1(\theta) \tilde{\psi}_2(b),\\
  \tilde{e}(\theta,b)&:= \tilde{\psi}(\theta, b) + \theta \tilde{\eta}(\theta,b) =
  \tilde{e}_0(\theta) + \tilde{e}_1(\theta) \tilde{\psi}_2(b).
\end{aligned}
\end{equation}
The  assumption \textbf{(A1)} on $\tilde{\psi}_0$ and $\tilde{\psi}_1$ also implies that $\tilde{e}_0$ and $\tilde{e}_1$ are increasing and nondecreasing, respectively,  as
\begin{equation}\label{incr}
  \tilde{e}_0'(s) = - s \tilde{\psi}_0''(s) >0 \quad \textrm{ and } \quad \tilde{e}_1'(s) = - s \tilde{\psi}_1''(s) \ge 0 \quad \textrm{ for all } s>0.
\end{equation}

The next assumption is aimed at controlling the behaviour of $\tilde{\psi}_1(\theta)$ at $\theta = 0$ and $\theta = +\infty$; it also regulates the behaviour of $\tilde{\psi}_2(b)$ at $b=1$.
\bcol
\begin{itemize}
\item[\textbf{(A2)}] Let $\tilde{\psi}_1 \in \mathcal{C}^2([0,\infty))$  and $\tilde{\psi}_2\in \mathcal{C}^2((0, \infty))$ satisfy 
\begin{equation*}
\tilde{\psi}_1(0) \ge 0\qquad \mbox{and} \qquad \tilde{\psi}_2(1) = \tilde{\psi}_2'(1) = 0,
\end{equation*}
and suppose that there exists a positive constant $C$ such that
\begin{equation*}
-s^2 \tilde{\psi}_1''(s) \leq C \quad \mbox{for all $s \in [0,\infty)$}.
\end{equation*}
\end{itemize}
\ecol
The first condition in \textbf{(A2)} combined with \textbf{(A1)} implies that $\tilde{\psi}_1$ is nonnegative; recall that $\tilde{\psi}_1$ is nondecreasing. In addition, the first condition also leads to $\tilde{e}_1(0)\ge 0$ and consequently, as $\tilde{e}_1$ is increasing, see \eqref{incr}, the function $\tilde{e}_1$ is nonnegative. The second condition mimics the example~\eqref{A10} above  and asserts that $\tilde{\psi}_2$ has minimum at $b=1$. This is a consequence of the convexity of $\tilde{\psi}_2$ stated in {\bf (A1)}. Furthermore, since $\tilde{\psi}_2'(1) = 0$, we see that the function $h(\theta, b)$ on the right-hand side of the evolution equation for $b$, see~\eqref{eq_b}, \eqref{eq:16} and~\eqref{eq:21}, vanishes at $b=1$, which means that the ``equilibrium state'' $b \equiv 1$ is a trivial solution of~\eqref{eq:21}.

Moreover, since $\tilde\psi_1$ and $\tilde\psi_2$ are nonnegative, it directly follows from {\bf (A1)} that $\tilde\psi$ is strictly concave with respect to the variable $\theta$ and convex with respect to the variable $b$; that is
\[
  \frac{\partial^2 \tilde\psi}{\partial \theta^2} <0 \quad \textrm{ and } \quad   \frac{\partial^2 \tilde\psi}{\partial b^2} \ge 0.
\]
The first requirement is the classical strict concavity condition for the Helmholtz free energy with respect to temperature, and we again recall that it guarantees the positivity of the specific heat at constant volume. The second condition is well known in theory of elasticity: the stored energy/Helmholtz free energy must be, if we consider the scalar case, a convex function of the stretch; see for example~\cite{ericksen.jl:introduction}.

Note that $|\tilde\psi|$ may explode as $\theta \rightarrow + \infty$, and also as $b \rightarrow 0_+$ and $b \rightarrow + \infty$; similarly $|\frac{\partial \tilde{\psi}}{\partial \theta}|$ may explode as $\theta \rightarrow 0_+$. The next assumption is introduced in order
to impose suitable growth conditions on the Helmholtz free energy, and thereby, indirectly, also on the internal energy and the entropy.
\bcol
\begin{itemize}
\item[\textbf{(A3)}] Let
$$
\lim_{\theta\to 0+} \tilde{\psi}_0'(\theta) = +\infty \quad \textrm{ and } \quad \lim_{\theta\to 0+} \theta \tilde{\psi}_0'(\theta) = \lim_{\theta\to 0+} \tilde{\psi}_0(\theta)= 0.
$$
\end{itemize}
\ecol
It follows from the definition of $\tilde{e}_0$, see \eqref{e0e1},  and from \textbf{(A3)} that
\begin{equation}\label{e0to0}
\lim_{\theta\to 0+} \tilde{e}_0 (\theta) = 0,
\end{equation}
which implies, in accordance with the above properties, that $\tilde{e}_0$ is strictly increasing and positive on~$(0,\infty)$. The next assumption imposed on $\tilde{\psi}_0$ and $\tilde{\psi}_1$ helps to control the behaviour of $\tilde{\psi}$ as~$\theta \to +\infty$.
\bcol
\begin{itemize}
\item[\textbf{(A4)}] Let $C_1$ and $C_2$ be two positive constants such that
\begin{equation*}
  C_1\le - s \tilde{\psi}_0''(s) \le C_2 \quad \mbox{for all $s>0$,}
\end{equation*}
and
\begin{equation*}
\int_0^{\infty} - s \tilde{\psi}_1''(s)\; \dd s \le C_2.
\end{equation*}
\end{itemize}
\ecol
The first assumption in \textbf{(A4)} combined with \eqref{incr} leads to the two-sided bound
$$
C_1\le \tilde{e}_0'(s)\le C_2,
$$
which together with \eqref{e0to0} gives
\begin{equation}\label{ule0}
C_1 s \le \tilde{e}_0(s)\le C_2s.
\end{equation}
The second assumption in \textbf{(A4)} and \eqref{incr} lead to the two-sided bound
\begin{equation}\label{ule1}
\tilde{e}_1(0)\le \tilde{e}_1(s)\le \tilde{e}_1(0) +C_2.
\end{equation}
Finally, we need to impose suitable conditions on the other functions/parameters featuring in the model under consideration.
\bcol
\begin{itemize}
\item[\textbf{(A5)}] Suppose that the functions $\nu,\kappa,  \alpha, h\in \mathcal{C}(\R_{\geq 0} \times \R_{> 0})$, and that they satisfy, for all $\theta\in \R_{ \geq 0}$ and $b\in (b_{\min},b_{\max})$, where $0 < b_{\min} < 1 < b_{\max}$, the following bounds:
\begin{align}
C_1\le \nu(\theta,b),\kappa(\theta,b),  &~\!\alpha(\theta,b) \le C_2,\label{alpha1}\\
0\le h(\theta,b)(b-1) \quad \textrm{ and } & \quad |h(\theta,b)| \le C_2. \label{alpha2}
\end{align}
\end{itemize}
\ecol

The rationale for the assumptions on the function $h(\theta, b)$ is the following. If we are interested in the governing equations~\eqref{eq:governing-equations-introduction}, that embody a scalar version of the diffusive Oldroyd-B model, then we need to set the specific Helmholtz free energy in the form~\eqref{A10}, that is we set\footnote{For the sake of clarity, we focus only on the dependence of the given functions on $b$, and we set all unimportant---for the discussion hereafter---physical constants to one.}
\begin{equation}
  \label{eq:23}
  \tilde{\psi}_2(b) := \left( b-1 - \ln b \right),
\end{equation}
and we get $\tilde{\psi}_2'(b)=1 - \frac{1}{b}$ and $\tilde{\psi}_2''(b)=\frac{1}{b^2}$, and we see that $\tilde{\psi}_2(b)$ satisfies assumptions \textbf{(A1)} and~\textbf{(A2)}. Furthermore, the function $h(\theta, b)$ on the right-hand side of the evolution equation for $b$, see~\eqref{eq_b}, needs to be fixed as~$h(\theta, b) := (b-1)$. This means that the function $C$ in the constitutive relation~\eqref{eq:16} is fixed as $C(\theta, b):=b$. The condition~\eqref{alpha2} is clearly satisfied, and, moreover, the condition on the non-negativity of the entropy production, see~\eqref{A16}, is satisfied as well.

Similarly, if we keep the Helmholtz free energy the same as in the previous case, and if we choose ~$h(\theta, b) := \alpha b^2 + (1-2\alpha)b - (1-\alpha)$, where $\alpha \in [0,1]$, then we would be dealing with a scalar version of the diffusive Giesekus model, see~\cite{giesekus.h:simple} for the original (non-diffusive) model and \cite{dostalk.m.prusa.v.ea:on} and \cite{dostalk.m.prusa.v.ea:unconditional} for a thermodynamic background of this model. A straightforward computation again shows that~\eqref{alpha2} is again satisfied. (Note that if $\alpha=0$, then this model reduces to the previous one.) Consequently, the structural condition $\eqref{alpha2}$ covers important viscoelastic rate-type models.



We close this section by stating and proving an auxiliary property of the entropy $\tilde{\eta}$, which follows from the assumptions \textbf{(A1)}, \textbf{(A2)}
and \textbf{(A4)}.

\begin{Lemma}\label{L:psi}
With $C_1$ and $C_2$ as in \emph{\textbf{(A4)}} we have, for all $s>0$ and all $b>0$, that
\begin{align}
\tilde{\eta}(s,b)\le -C_1|\ln s| +  C(1+s). \label{etaup}
\end{align}
where $C:=\max(C_1 + C_2, |\tilde{\psi}_0'(1)|)$.
\end{Lemma}

\begin{proof}
Taylor expansion of $\tilde{\psi}_2$ about $b=1$ gives $\tilde{\psi}_2(b) = \tilde{\psi}_2(1) + (b-1)\tilde{\psi}_2'(1) + \frac{1}{2}(b-1)^2 \tilde{\psi}_2''(\sigma)$, where $\sigma \in (0,b)$; hence, by Assumptions \textbf{(A1)} and \textbf{(A2)}, $\tilde{\psi}_2(b) = \frac{1}{2}(b-1)^2 \tilde{\psi}_2''(\sigma)\geq 0$ for all $b > 0$. Thus, by \eqref{eeta:def} and assumptions \textbf{(A1)} and \textbf{(A4)},
$$
\begin{aligned}
\tilde{\eta}(s,b) &= -\tilde{\psi}_0'(s) - \tilde{\psi}_1'(s)\tilde{\psi}_2(b) \le -\tilde{\psi}_0'(s) \\
&=\left\{
\begin{aligned}
&-\int_1^s \tilde{\psi}_0''(t)\dt -\tilde{\psi}_0'(1)\le C_2 \ln s -\tilde{\psi}_0'(1), &&\mbox{for } s\ge 1,\\
& \int_{s}^1 \tilde{\psi}_0''(t)\dt  -\tilde{\psi}_0'(1)\le C_1 \ln s - \tilde{\psi}_0'(1), &&\mbox{for } s\le 1\\
\end{aligned}
\right.
\\
&=\left\{
\begin{aligned}
& - C_1 (\ln s) + (C_1 + C_2) \ln s -\tilde{\psi}_0'(1), &&\mbox{for }s\ge 1,\\
& - C_1 (-\ln s) - \tilde{\psi}_0'(1), &&\mbox{for } s\le 1\\
\end{aligned}
\right.\\
&\leq  \left\{
\begin{aligned}
& - C_1 |\ln s| + (C_1 + C_2) s  + |\tilde{\psi}_0'(1)|, &&\mbox{for }s\ge 1,\\
& - C_1 |\ln s| +  |\tilde{\psi}_0'(1)|, &&\mbox{for } s\le 1,\\
\end{aligned}
\right.
\end{aligned}
$$
which implies the stated result with $C:=\max(C_1 + C_2, |\tilde{\psi}_0'(1)|)$.
\end{proof}

\section{Function spaces and the statement of the main results}
In this section, we provide a precise statement of the main results of the paper, concerning the existence of large-data global weak solutions of the evolutionary system under consideration. In order to formulate these, we need to introduce suitable function spaces.

\subsection{Notation}
Throughout the paper we shall use the standard notations for Lebesgue, Sobolev, Bochner and Sobolev--Bochner spaces:  $(L^p(\Omega), \|\cdot\|_p), (W^{s,p}(\Omega), \|\cdot\|_{s,p})$,  $L^p(0,T; X)$ and $W^{1,p}(0,T;X)$, with $p$~satisfying $1\le p \le \infty$, $s\in \mathbb{R}_+$,  and~$X$ being a Banach space. In order to distinguish between scalar-, vector-, and tensor-valued functions, we use small letters for scalars, small bold letters for vectors and capital blackboard bold letters for tensors. Moreover, in order to simplify the notation, for any Banach space $X$ we shall use the  abbreviation $$X^k:= \underset{k \text{-times}}{\underbrace{X \times \cdots \times X}}.$$ We shall also use the notation $Q:=(0,T) \times \Omega$ to denote the space-time cylinder in which the problem in studied.

We require that $\Omega$ is a $\mathcal{C}^{1,1}$ domain and we define (here $r':=r/(r-1)$ denotes the conjugate exponent to $r$)
\begin{align*}
W^{1,r}_{\bn}&:=\big\{ \vv \in W^{1,r}(\Omega)^3: \; \vv \cdot \bn =0 \textrm{ on } \partial \Omega \big\},\\
W^{1,r}_{\bn,\diver}&:=\big\{ \vv \in W^{1,r}_{\bn}: \; \diver \vv =0 \textrm{ in }  \Omega \big\},\\
W^{-1,r'}_{\bn}&:=\left(W^{1,r}_{\bn} \right)^{*}, \quad W^{-1,r'}_{\bn,\diver}:=\left(W^{1,r}_{\bn,\diver} \right)^{*},\\
L^2_{\bn,\diver}&:= \overline{W^{1,2}_{\bn,\diver}}^{\|\cdot \|_{2}}.
\end{align*}
All of the above spaces are Banach spaces, which are for $r\in [1,\infty)$ separable and for $r\in (1,\infty)$ reflexive.
We shall also employ the following notation for functions having zero mean value:
$$
L^r_0(\Omega):=\{u\in L^r(\Omega): \, \int_{\Omega} u \dd x =0\}.
$$
Since we shall work with sequences that are precompact only in the space of measures (bounded sequences in $L^1$),  we denote the space of Radon measures on a set $V$ by $\mathcal{M}(V)$. 

In addition, for any measurable subset $V$ of $\Omega$ we shall write $(a,b)_V:=\int_V ab \dd x $ whenever $a\in L^r(V)$ and $b\in L^{r'}(V)$, and in particular if $V=\Omega$ we shall omit writing the associated subscript in what follows. We shall use the same notational convention for vector- and tensor-valued functions. In the case of the duality pairing between a Banach space~$X$ and its dual space~$X^*$ we shall use the abbreviated notation $\langle a, b \rangle:=\langle a, b\rangle_{X^*,X}$ whenever $a\in X^*$ and $b\in X$ and the meaning of the duality pairing is clear from the context.

Finally, we shall employ the standard notation $L^p(0,T; X)$ for Bochner spaces of $X$-valued $L^p$ functions defined on $(0,T)$, and the space of $X$-valued weakly continuous functions defined on $(0,T)$ is denoted by
$$
\mathcal{C}_{w}(0,T; X):=\{u\in L^{\infty}(0,T; X); \; \lim_{t\to t_0} \langle u(t), \varphi\rangle = \langle u(t_0), \varphi\rangle \quad \forall \, \varphi \in X^*,\; \forall \, t_0 \in [0,T]\},
$$
and for a Banach space $X$, $\mathcal{M}(0,T;X^*)$ will denote the dual space of $\mathcal{C}([0,T];X)$.
\subsection{Statement of the main theorem}

We are now ready to state the main result of the paper, concerning the existence of large-data global weak solutions to the system \eqref{A15} subject to the boundary and initial conditions \eqref{bc1}, \eqref{bc2} and the initial conditions \eqref{ics}. We emphasise that, as the density $\varrho$ has been assumed to be constant, without loss of generality it has been taken to be identically 1 in what follows.
\begin{Theorem}\label{main:T}
Suppose that $\Omega\subset \mathbb{R}^3$ is a bounded open $\mathcal{C}^{1,1}$ domain and $T>0$ is arbitrary. Assume that the given  Helmholtz free energy $\tilde{\psi}:\mathbb{R}_+ \times \mathbb{R}_+ \to \mathbb{R}$ satisfies assumptions \textbf{(A1)}--\textbf{(A4)} and that the corresponding internal energy $\tilde{e}(\theta,b)$ and the entropy $\tilde{\eta}(\theta,b)$ are defined in \eqref{eeta:def}. Assume that the functions $\nu$, $\kappa$, $\alpha$, $h:\mathbb{R}_+ \times \mathbb{R}_+ \to \mathbb{R}$ fulfil the assumption \textbf{(A5)} and $s_*>0$ and $\gamma_*>0$ are given numbers. Then, for arbitrary initial data $(\vv_0,\theta_0,b_0)$ satisfying $\theta_0>0$, $b_0\in [b_{\min},b_{\max}]$ almost everywhere in $\Omega$ with $0<b_{\min} \le 1\le b_{\max}< \infty$  and
\begin{align}\label{ics1}
\vv_0 \in L^2_{\bn,\diver},\qquad  \tilde{e}(\theta_0,b_0)\in L^1(\Omega),\qquad  \tilde{\eta}(\theta_0,b_0) \in L^{1}(\Omega),
\end{align}
and for an arbitrary body force $\bfb \in L^2(0,T; (W^{1,2}_{\bn})^*)$, there exists a weak solution $(\vv, \theta, b,p)$, with $\theta$ and $b$ nonnegative, to the system \eqref{A15}, \eqref{bc1}, \eqref{bc2}, \eqref{ics}, in the sense that
\begin{align}
\vv&\in L^2(0,T; W^{1,2}_{\bn,\diver}) \cap \mathcal{C}_{\textrm{weak}}(0,T; L^2_{\bn,\diver}),\label{FS1}\\
\partial_t \vv &\in L^{\frac43}(0,T; (W^{1,2}_{\bn})^*),\label{FS2}\\
E&\in L^{\infty}(0,T; L^1(\Omega))\cap W^{1,\frac{10+\delta}{9+\delta}}(0,T; (W^{1,10+\delta}(\Omega))^*)&&\textrm{for all } \delta>0,\label{FS3}\\
e,\theta & \in L^{\frac54-\delta}(0,T; W^{1,\frac54 -\delta}(\Omega)) \cap L^{\frac53-\delta}(0,T; L^{\frac53-\delta}(\Omega))&&\textrm{for all }\delta\in\big(0,\textstyle{\frac14}\big),\label{FS4}\\
\partial_t e &\in \mathcal{M}(0,T; (W^{1,10+\delta}(\Omega))^*) &&\textrm{for all } \delta>0,\label{FS5}\\
b& \in L^2(0,T; W^{1,2}(\Omega))\cap W^{1,2}(0,T; (W^{1,2}(\Omega))^*),\label{FS6}\\
b&\in [b_{\min},b_{\max}] &&\textrm{a.e. in }Q,\label{FS7}\\
\eta&\in L^{\infty}(0,T; L^1(\Omega))\cap L^2(0,T; W^{1,2}(\Omega)),\label{FS8}\\
\partial_t \eta &\in \mathcal{M}(0,T; (W^{1,10+\delta}(\Omega))^*)&&\textrm{for all } \delta>0,\label{FS9}\\
 p&\in L^{\frac{5}{3}}(0,T; L^{\frac53}_0(\Omega))\cap L^{\frac43}(0,T; L^2_0(\Omega)),\label{FS10}\\
\bs &\in L^{2}(0,T; L^{2}(\partial \Omega)^3),\label{FS11}
\end{align}
are linked through the relations
\begin{equation}\label{constitutive}
e=\tilde{e}(\theta,b), \quad \eta=\tilde{\eta}(\theta,b), \quad E=e+\frac{|\vv|^2}{2} \qquad \textrm{ a.e. in } Q,
\end{equation}
\begin{equation}\label{slip}
\gamma_*\vv_{\mathbf{\tau}} = \frac{(|\bs|-s_*)_+}{|\bs|}\bs \qquad \textrm{ a.e. in } (0,T)\times \partial \Omega,
\end{equation}
and satisfy the following:
\begin{equation}
\begin{split}\label{eq:vv}
&\langle \partial_t\vv, \vw\rangle + \int_{\Omega}(2\nu(\theta,b)\mathbb{D}(\vv)-\vv\otimes \vv ):\Grad \vw\dx+ \int_{\partial \Omega} \bs\cdot \vw \dS
=\langle \bfb, \vw\rangle + \int_{\Omega} p \Div \vw \dx
\end{split}
\end{equation}
for all $\vw\in W^{1,2}_{\bn}$ and almost all $t\in (0,T)$;
\begin{equation}\label{eq:b}
\begin{split}
 \langle \partial_t b, u \rangle  + \int_{\Omega}\left(\alpha(\theta,b) \Grad b - b\vv \right)\cdot \Grad u +h(\theta,b)u \dx=0
\end{split}
\end{equation}
for all $u\in W^{1,2}(\Omega)$ and almost all $t\in (0,T)$;
\begin{equation}\label{eq:E}
\begin{split}
 \langle \partial_t E, u\rangle  &+ \int_{\Omega}\left( \kappa(\theta,b)\Grad \theta
-(E+p)\vv\right) \cdot \Grad u\dx +\int_{\partial \Omega} \bs\cdot \vv u \dS\\
&+\int_{\Omega} \alpha(\theta,b) \frac{\partial \tilde{e}(\theta,b)}{\partial b}\Grad b  \cdot \Grad u \dx=\langle \bfb, \vv u\rangle
\end{split}
\end{equation}
for all $u\in W^{1,\infty}(\Omega)$ and almost all $t\in (0,T)$;
\begin{equation}\label{eq:entropy}
\begin{aligned}
\langle \partial_t \eta, u \rangle \dx & +\int_{\Omega} \left(\frac{\kappa(\theta,b) \Grad \theta}{\theta}-\frac{\partial^2 \tilde{\psi}(\theta,b)}{ \partial \theta \, \partial b }\alpha(\theta, b) \Grad b -\eta \vv \right) \cdot \Grad u \dx\\
&\ge \int_{\Omega} \frac{\kappa(\theta,b)|\Grad \theta|^2u}{\theta^2}+  \frac{2\nu(\theta, b)|\mathbb{D}(\vv)|^2u}{\theta} \dd x \\
&\qquad \quad +\int_{\Omega} \frac{\alpha(\theta,b)}{\theta}\frac{\partial^2 \tilde{\psi}(\theta,b)}{\partial b^2}|\Grad b|^2u  +\frac{\partial \tilde{\psi}(\theta,b)}{\partial b}\frac{h(\theta,b)u}{\theta} \dx
\end{aligned}
\end{equation}
for all nonnegative $u\in W^{1,\infty}(\Omega)$ and almost all $t\in (0,T)$. In addition, the triple $(\vv, \theta, b)$ fulfils the following inequality constraint, which is a consequence of the second law of thermodynamics:
\begin{equation}\label{eq:ienergy}
\begin{split}
\langle \partial_t e, u\rangle  &+ \int_{\Omega}\left( \kappa(\theta,b)\Grad \theta
-e\vv\right) \cdot \Grad u\dx +\int_{\Omega} \alpha(\theta,b) \frac{\partial \tilde{e}(\theta,b)}{\partial b}\Grad b  \cdot \Grad u \dx\\
&\ge\int_{\Omega} 2\nu(\theta,b)\,|\mathbb{D}(\vv)|^2 u\dx
\end{split}
\end{equation}
for all nonnegative $u\in W^{1,\infty}(\Omega)$ and almost all $t\in (0,T)$.

The initial data are attained in the following sense:
\begin{equation}\label{attint}
\lim_{t\to 0_+} \big(\|\vv(t)-\vv_0\|_2 + \|b(t)-b_0\|_2 + \|e(t)-\tilde{e}(\theta_0,b_0)\|_1 + \|\theta(t)-\theta_0\|_1\big)=0.
\end{equation}
Furthermore, the following bound holds:
\begin{equation}\label{folbound}
\int_0^T \int_{\Omega} \frac{|\Grad \theta|^2}{\theta^2}+\frac{|\Grad e|^2}{(1+e)^{1+\delta}} + \frac{|\Grad \theta|^2}{(1+\theta)^{1+\delta}}\dx \dt \le C(\delta)\qquad \forall\, \delta >0.
\end{equation}
\end{Theorem}

\section{Road-map of the proof of Theorem \ref{main:T}}\label{s3}

In this section we describe the approximation scheme that is used to establish the existence of a weak solution. We use a relatively complicated approximation scheme with five parameters. The reason for employing such a complicated procedure is the fact that in the equation for the internal energy `cross-terms' appear, whose presence complicates the analysis, and, in addition, to obtain the entropy estimate \eqref{eq:entropy} and the energy estimate \eqref{eq:ienergy}
asserted in Theorem \ref{main:T}, we need to employ very special test functions, whose use is not easy to justify at the level of weak solutions or their finite-dimensional approximations.

\subsection{Approximation of the free energy}
We want to use the internal energy as a primitive variable (at least at the beginning of the proof) because its time derivative features in the equation for internal energy. Therefore, we slightly perturb the form of the Helmholtz free energy postulated in the assumption \textbf{(A1)} and, for  $\eps>0$ (fixed for the moment), we define the regularised Helmholtz free energy
\[ \tilde{\psi}^\eps(\theta,b) : = \tilde{\psi}_0(\theta) + \tilde{\psi}_1^\eps(\theta) \tilde{\psi}_2(b),\]
with nondecreasing concave $\tilde{\psi}^{\eps}_1\in \mathcal{C}^2([0,\infty))$ satisfying   $\tilde{\psi}_1^\eps(0)=0$, $\tilde{\psi}_1^\eps$ linear on $[0,\frac{1}{2}\eps]$ and $\tilde{\psi}_1^\eps(\theta) = \tilde{\psi}_1(\theta)$ for $\theta \geq \eps$ and $(\tilde{\psi}^{\eps}_1)'(s)\ge \tilde{\psi}_1'(s)$ for all $s\ge 0$. Observe that thanks to \textbf{(A1)} and \textbf{(A2)} such a construction is possible. 
In addition, we require\footnote{In fact, for such a construction, one requires that $\tilde{\psi}_1(0)>0$. Then, one can define
$$
\tilde{\varphi}(s):=\left\{\begin{aligned}&\tilde{\psi}_1(s) &&\mbox{for } s\ge \frac{3\eps}{4},\\
&s\frac{4\tilde{\psi}_1(3\eps/4)}{3\eps} &&\mbox{for } s\in [0,\frac{3\eps}{4}).
\end{aligned}\right.
$$
The function $\tilde{\varphi}$ satisfies all requirements except $\eps^{2}|(\tilde{\varphi})''(s)|\le C$ since $\varphi$ is not a $\mathcal{C}^2$ function. Hence, an appropriate function $\tilde{\psi}_1^{\eps}$ can be found as a regularization of $\tilde{\varphi}$ near $s=3\eps/4$. In case, $\tilde{\psi}_1(0)=0$ one can use a similar procedure with one proviso: the function $\tilde{\psi}_1$ will satisfy only $(\tilde{\psi}^{\eps}_1)'(s)\ge \frac12\tilde{\psi}_1'(s)$. However, the presence of the factor $\frac12$ does not change the subsequent analysis.} that for some constant $C>0$ independent of $\eps$ and for all $s\in (0,\eps)$ the following bound holds:
\begin{equation}\label{boundapp}
\eps |(\tilde{\psi}_1^\eps)'(s)| + \eps^{2}|(\tilde{\psi}_1^\eps)''(s)|\le C.
\end{equation}

Note further that there exists a positive constant $C>0$ independent of $\eps$ such that for all $s \in [0,\infty)$ the following bound holds:
\begin{equation}\label{boundapp1}
-s^2 (\tilde{\psi}_1^\eps)''(s) \leq C.
\end{equation}
For $s \in [0,\frac{1}{2}\eps]$ this bound is trivially true, as for such $s$ we have $(\tilde{\psi}_1^\eps)''(s) =0$. For $s \in [\frac{1}{2}\eps,\eps]$ the bound \eqref{boundapp1} is a direct consequence \eqref{boundapp} and the concavity of $\tilde{\psi}_1^{\eps}$.  For $s \geq \eps$ the inequality \eqref{boundapp1} follows from the assumption \textbf{(A2)}.

Motivated by the expressions for the internal energy and the entropy stated in \eqref{e0e1} and \eqref{eeta:def}, we also introduce the corresponding regularised internal energy and entropy as follows:
\begin{equation}\label{e0e1ep}
\begin{split}
  \tilde{e}^{\eps}(\theta,b)&:=
  \tilde{e}_0(\theta) + \tilde{e}^{\eps}_1(\theta) \tilde{\psi}_2(b),\qquad \mbox{where $\qquad \tilde{e}^{\eps}_1(\theta):= \tilde{\psi}^{\eps}_1(\theta) - \theta (\tilde{\psi}^{\eps}_1)'(\theta)$},\\
  \tilde{\eta}^{\eps}(\theta,b)&:=-\frac{\partial \tilde{\psi}^{\eps}(\theta,b)}{\partial \theta} = -\tilde{\psi}'_0(\theta) -(\tilde{\psi}^{\eps}_1)'(\theta) \tilde{\psi}_2(b).
\end{split}
\end{equation}
Hence, in particular, thanks to the assumed linearity of $\psi_1^\eps(\theta)$ for $\theta \in [0,\frac{1}{2}\eps]$, we have that
\[ \tilde{e}^\eps(\theta,b) = \tilde{e}_0(\theta) \qquad \textrm{for all }\theta \in \big[0,\frac{1}{2}\eps\big].\]
Moreover, we deduce from~\eqref{boundapp1} the following uniform bound on $\tilde{e}^{\eps}_1(s)$ for $s\in(\eps/2,\eps)$:
\begin{equation}\label{e1epsest}
\tilde{e}^{\eps}_1(s)=\int_0^s -t(\tilde{\psi}_1^{\eps})''(t)\dt\le \int_{\eps/2}^{\eps} C\eps^{-1}\dt \le C.
\end{equation}
It further follows from \eqref{e0to0} that
\begin{align}
\lim_{\theta \to 0_+} \tilde{e}^\eps(\theta,b) &=0;\label{l00}\\
\intertext{furthermore, the assumption \textbf{(A4)} (and its consequence \eqref{ule0}) give}
\lim_{\theta \to \infty} \tilde{e}^\eps(\theta,b) &=+\infty.\label{l01}
\end{align}
In addition, it follows from the assumption \textbf{(A3)} that
\begin{equation}\label{entr.nula}
\lim_{\theta\to 0_+}\tilde{\eta}^\eps(\theta,b)=-\infty.
\end{equation}

Consequently, using \textbf{(A1)} and \textbf{(A2)}, we see that for a fixed $b$ the function $\tilde{e}^\eps(\theta,b)$ is strictly monotonically increasing and thanks to \eqref{l00} and \eqref{l01} it maps $[0,\infty)$ onto $[0,\infty)$. We can therefore introduce ${\tilde \theta}^{\eps}(e,b)$ as the inverse with respect to $\theta$ of $\tilde{e}^\eps(\theta,b)$ for $\theta \geq 0$, and because $\tilde{e}^\eps(0,b)=0$, whereby ${\tilde \theta}^{\eps}(0,b)=0$, we can continuously extend $\tilde{\theta}^{\eps}$ to all $e \in \mathbb{R}$ by letting  ${\tilde \theta}^{\eps}(e,b) := e$ for $e<0$. Thus we  have that
\begin{equation}\label{signif}
\tilde{\theta}^{\eps}(e,b) > 0 \qquad \mbox{if and only if} \qquad e>0.
\end{equation}%
In addition, thanks to \textbf{(A4)}, we see that $\tilde{\theta}^{\eps}$ is Lipschitz continuous on $\mathbb{R} \times (0,\infty)$.

\subsection{Cut-off functions}
In this section we define several cut-off functions, which will be inserted into the system of PDEs in order to facilitate the initial stages of the existence proof. First, since we expect that $b\in (b_{\min}, b_{\max})$, we define the following cut-off of the quantity $b$:
\begin{equation}\label{cutb}
 b_M := \max(b_{\rm min}, \min(b, b_{\rm max})).
\end{equation}
The subscript $M$ will be from now on reserved for the operation introduced in \eqref{cutb}. Similarly, the positive and negative part of any quantity $a$ will be denoted, respectively, by
$$
a_+:=\max\{0,a\}, \qquad a_{-}:= \min\{0,a\}.
$$
Next, let $G$ be a nonnegative smooth function such that
$$
G(s)=\left\{\begin{aligned}&1 &&\mbox{for } s\in [0,1],\\
&0 &&\mbox{for } s>2,
\end{aligned}
\right.
$$
and for an arbitrary $k\in \mathbb{N}$ we set
\begin{equation}
\label{cutG}
G_k(s):=G(s/k).
\end{equation}
This cut-off will be used in the momentum equation in the convective term to preserve the energy equality.

The last cut-off function we introduce will be used for the initial conditions. For $\eps\in (0,1)$ we define $e_0^{\eps} \in L^{\infty}(\Omega)$ as
\begin{equation}\label{e0eps}
e_0^{\varepsilon}(x):=\tilde{e}^{\eps}\left(\max\left\{\eps, \min\{\theta_0(x),\eps^{-1}\}\right\},b_0(x)\right).
\end{equation}
Note that it follows from \eqref{e0e1} and the fact that $\tilde{\psi}_2$ and $\tilde{e}_1^{\eps}$ are nonnegative and $\tilde{e}_0$ is nondecreasing  that
\begin{equation}
e_0^{\varepsilon}(x)\ge \tilde{e}_0(\eps)\ge \tilde{e}_0(\eps/2)\qquad \textrm{a.e. in }\Omega.\label{minprinceps}
\end{equation}

We shall require one further approximation, which concerns the stick-slip boundary condition. Since we want to use standard techniques from PDE theory in the initial stages of the proof, we approximate the relation \eqref{bc2} by the following mollification:
\begin{equation}
\label{stickeps}
\tilde{\bs}^{\eps}(\vv_{\tau}):= s_*\frac{\vv_{\tau}}{\eps + |\vv_{\tau}|} + \gamma_{*}\vv_{\tau}.
\end{equation}
We note that $\tilde{\bs}^{\eps}: \mathbb{R}^3 \to \mathbb{R}^3$ defined through \eqref{stickeps} is a continuous mapping.

\subsection{Elliptic term in the evolution equation for internal energy}
\label{sec:ellipt-term-evol}
In the initial stages of the proof we focus on the evolution equation for the internal energy, that is on the equation whose strong form reads
\begin{equation}
  \label{eq:27}
  \varrho\tecka{e} = 2\nu(\theta, b) |\mathbb{D}(\vv)|^2 + \Div(\kappa(\theta, b) \Grad \theta + (\tilde{\psi}_1(\theta) - \theta \tilde{\psi}'_1(\theta)) \tilde{\psi}'_2(b)  \alpha(\theta, b) \Grad b).
\end{equation}
In order to facilitate the passage to the limit in the Galerkin approximation of this equation, we introduce the additional elliptic term $\mu \Delta e$ with $\mu >0$ on the right-hand side of~\eqref{eq:27}, and we shall consider instead (a weak form) of the equation
\begin{equation}
  \label{eq:28}
  \varrho\tecka{e} = \mu \Delta e + 2\nu(\theta, b) |\mathbb{D}(\vv)|^2 + \Div(\kappa(\theta, b) \Grad \theta + (\tilde{\psi}_1(\theta) - \theta \tilde{\psi}'_1(\theta)) \tilde{\psi}'_2(b)  \alpha(\theta, b) \Grad b),
\end{equation}
see~\eqref{Galerkin:e} below. In later stages of the proof we shall remove the elliptic term by passing to the limit~$\mu \to 0_+$.

\subsection{Galerkin approximation}
Finally, we introduce a Galerkin approximation, which represents the starting point for the existence proof. Thanks to the separability of the underlying spaces we can find systems of functions $\{\vw_i\}_{i=1}^{\infty}\subset W^{1,2}_{\bn, \diver}\cap W^{3,2}(\Omega)^3$ and $\{u_i\}_{i=1}^{\infty}\subset W^{1,2}(\Omega)$ whose linear hulls are dense in $W^{1,2}_{\bn, \diver}$ and $W^{1,2}(\Omega)$, respectively. Furthermore, we can choose these systems so that they are orthogonal in $W^{3,2}(\Omega)^3$ and $W^{1,2}(\Omega)$ and orthonormal in $L^2_{\bn, \diver}$ and $L^2(\Omega)$ respectively. Then, for the parameters $m,n,k \in \mathbb{N}$ and $\eps,\mu>0$, we shall seek a triple $(\vv,b,e)$ (whose dependence on the parameters $m, n, k, \eps$ and $\mu$ will not be indicated in our notation hereafter, unless we explicitly state otherwise) given by the Galerkin expansions
\begin{equation}\label{Galerkin1}
\begin{split}
\vv (t,x) &= \sum_{i=1}^n c_i (t)\vw_i(x),\\
b (t,x) &= \sum_{i=1}^m d_i (t) u_i(x),\\
e(t,x) &= \sum_{i=1}^m e_i (t) u_i(x),
\end{split}
\end{equation}
which solve \eqref{A15}  approximately in the following sense:
\begin{equation}
\begin{split}\label{Galerkin:vv}
\int_{\Omega} \partial_t\vv \cdot \vw_i &+ (2\nu(\tilde{\theta}^{\eps}(e_+,b_M),b_M)\mathbb{D}(\vv)-G_k(|\vv|)\vv\otimes \vv ):\Grad \vw_i\dx+ \int_{\partial \Omega} \tilde{\bs}^{\eps}(\vv_{\tau})\cdot \vw_i \dS\\
&=\langle \bfb, \vw_i\rangle
\end{split}
\end{equation}
for all $i=1,\ldots, n$ and almost all $t\in (0,T)$;
\begin{equation}\label{Galerkin:b}
\begin{split}
\int_{\Omega} \partial_t b \, u_i + \left(\alpha({\tilde \theta}^{\eps}(e_+,b_M),b_M) \Grad b - b\vv \right)\cdot \Grad u_i +h({\tilde \theta}
^{\eps}(e_+,b_M),b_M)u_i \dx=0
\end{split}
\end{equation}
for all $i=1,\ldots, m$ and almost all $t\in (0,T)$;
\begin{equation}\label{Galerkin:e}
\begin{split}
\int_{\Omega} \partial_t e \, u_i &+ \left(\mu \Grad e + \kappa({\tilde \theta}^{\eps}(e_+,b_M),b_M)\Grad {\tilde \theta}^{\eps}(e,b_M)
-e\vv\right) \cdot \Grad u_i\dx \\
&+\int_{\Omega} \left( \alpha({\tilde \theta}^{\eps}(e_+,b_M),b_M) \tilde{e}^{\eps}_1({\tilde \theta}^{\eps}(e_+,b_M)) \tilde{\psi}_2'(b_M)\Grad b  \right)\cdot \Grad u_i \dx\\
&\hspace{-1.6cm}=\int_{\Omega} 2\nu({\tilde \theta}^{\eps}(e_+,b_M),b_M)\,|\mathbb{D}(\vv)|^2 u_i\dx
\end{split}
\end{equation}
for all $i=1,\ldots, m$ and almost all $t\in (0,T)$.

The system \eqref{Galerkin:vv}--\eqref{Galerkin:e} is supplemented by the following initial conditions:
\begin{equation}\label{Galerkininit}
\begin{aligned}
\vv (0) &=\vv_0^n= \sum_{i=1}^n c_i^n (0)\vw_i, &&c_i^n(0):=  \int_{\Omega} \vv_0 \cdot \vw_i \dx,\\
b (0) &= b_0^m = \sum_{i=1}^m d_i^m (0) u_i, &&d_i^m(0):= \int_{\Omega} b_0\,  u_i \dx,\\
e(0) &= e_0^{m,\eps}= \sum_{i=1}^m e_i^m (0) u_i, &&e_i^m(0):=  \int_{\Omega} e_0^{\eps} \,  u_i \dx.
\end{aligned}
\end{equation}
We note that, thanks to the $L^2$-orthonormality of the three bases used, these initial conditions are just the projections of the original initial conditions onto the finite-dimensional spaces spanned by the first $n$, respectively $m$,
corresponding Galerkin basis-functions.

We do not provide the detailed proof of the existence of a solution to \eqref{Galerkin1}--\eqref{Galerkininit}. In fact, the short-time existence of solutions to \eqref{Galerkin1}--\eqref{Galerkininit} directly follows from Carath\'{e}odory's existence theorem for systems of ordinary differential equations. The global-in-time existence of solutions to \eqref{Galerkin1}--\eqref{Galerkininit} is then a consequence of the uniform (independent of the time interval) bounds on the solution established in Section~\ref{mtoinfty}.

\subsection{Limiting procedures}
The next step in the proof is to pass to the appropriate limits with the parameters $m, n, k, \eps$ and~$\mu$ in the sequence of approximate solutions. We recall that the parameters $m, n$ are related to the finite-dimensional (Galerkin) approximations of the governing equations, while the parameters $k, \eps, \mu$ are related to the use of cut-off functions, the approximation of the formula for the Helmholtz free energy, and the elliptic regularization, respectively.

\subsubsection*{The limit $m\to \infty$} In this first step, we pass to the limits in the equations for $b$ and for $e$, while we keep the equation for the velocity field at the Galerkin level. Note that this is needed because we need to control the term on the right-hand side in the equation for $e$ in a `good' function space. After passing to the limit $m \rightarrow \infty$, we will deduce minimum and maximum principles for $b$ (under the hypothesis \eqref{alpha2}) and also a uniform bound on $\Grad b$. We can then simply replace $b_M$ by $b$ in all equations.
Using the equation for $e$, we will show a minimum principle for $e$ and thus we shall be able to replace $e_+$ by $e$ in all equations. We note that the essential point here is the use of the approximate free energy $\tilde{\psi}^{\eps}$ because of the presence of the ``cross"-term $\Grad b$ in the equation for $e$.

\subsubsection*{The limit $\mu \to 0_+$} At this level we can remove the additional elliptic term from the equation for $e$ since we are able to guarantee the correct relation between $\theta$ and $e$ through $\tilde{\psi}^{\eps}$, and the term $\Grad \theta$ in the equation for $e$ also provides information about $\Grad e$. Ultimately, we have to remove this term in order to be able to derive the entropy inequality, but for this we need to be able to test the equation for the internal energy with $\theta^{-1}$, and for the rigorous justification of this step we need to have a minimum principle in hand. Therefore passage to the limit $\mu \to 0_+$ must precede passage to the limit $\eps \to 0_+$.
\subsubsection*{The limit $\eps\to 0_+$} At this level, we can rigorously deduce the equation for the entropy and consequently obtain a uniform bound on the measure of the set where the temperature hypothetically vanishes. This is done by appealing to \textbf{(A3)}. This then allows us to let $\eps\to 0_+$ and to pass to the original free energy $\tilde{\psi}$. Thanks to this step, our primary variable once again becomes the temperature, and the internal energy thus becomes a function of $\theta$ and $b$.
\subsubsection*{The limit $n\to \infty$} In this step, we finally pass from the Galerkin approximation to the continuous problem in the momentum equation as well. Then, thanks to the presence of the cut-off function $G_k$ in the convective term, we can still deduce the associated energy equality. Moreover, thanks to the boundary conditions, we can introduce an integrable pressure and consequently deduce the equation for the global energy~$E$.
\subsubsection*{The limit $k\to \infty$} This is the last step, where we remove the cut-off function from the convective term. In fact, having introduced all necessary estimates in the previous step, the limiting procedure is relatively simple, since the highest order terms are linear. The only nonstandard point here is to deduce the strong attainment of the initial data. This is however done by using an appropriate combination of the entropy inequality and the global energy equality.


\section{Proof of Theorem~\ref{main:T}}
As indicated in the previous section, the proof will be split into several steps, which we discuss here in detail. Our starting point is that we have a triple $(\vv,b,e)$, which solves \eqref{Galerkin1}--\eqref{Galerkininit}. In what follows, we also clearly indicate the dependence of the estimates on the parameters featuring in the various approximations we have performed. If an estimate is uniform we shall simply use the letter $C$ to denote a generic constant, which may depend on the data but not on the approximation parameters. If the estimate is not uniform with respect to a certain parameter, then we shall clearly indicate this; e.g., we shall write $C(n)$ when the constant $C$ depends on $n$.

\subsection{The limit $m\to \infty$}\label{mtoinfty}
We start with the derivation of the first set of a~priori estimates. We denote the triple solving the problem \eqref{Galerkin1}--\eqref{Galerkininit} by $(\vv^m, b^m, e^m)$ and derive our $m$-independent estimates, which will then permit passage to the limit $m \rightarrow \infty$.
\subsubsection*{A priori estimates}
We start with bounds on $\vv^m$. Multiplying \eqref{Galerkin:vv} by $c_i^{m}$ (the superscript $m$ is used here to emphasise the dependence on $m$) and summing with respect to $i=1,\ldots, n$ we get the identity
\begin{equation}
\begin{split}\label{EE1}
\frac12 \frac{\dd}{\dd t} \|\vv^m(t)\|_2^2 + \int_{\Omega} 2\nu(\tilde{\theta}^{\eps}(e^m_+,b^m_M), b^m_M)|\mathbb{D}(\vv^m)|^2\dx + \int_{\partial \Omega} \tilde{\bs}^{\eps}(\vv^m_{\tau})\cdot \vv^m_{\tau} \dS = \langle \bfb, \vv^m\rangle,
\end{split}
\end{equation}
where we have used the fact that $\Div \vv^m=0$, integration by parts, the fact that $\vv^m\cdot \mathbf{n} =0$ on $(0,T) \times \partial \Omega$, and the following identity:
$$
\begin{aligned}
\int_{\Omega} G_k(|\vv^m|)(\vv^m \otimes \vv^m) : \Grad \vv^m \dx &= \int_{\Omega} G_k(|\vv^m|)|\vv^m| \Grad |\vv^m| \cdot  \vv^m \dx \\
&= \int_{\Omega} \Grad \left(\int_0^{|\vv^m|} G_k(s)s \; \dd s \right) \cdot  \vv^m \dx\\
&=  -\int_{\Omega} \left(\int_0^{|\vv^m|} G_k(s)s \; \dd s \right) \Div \vv^m \dx=0.
\end{aligned}
$$
Hence, using \textbf{(A5)}, Korn's inequality, the definition \eqref{stickeps}, and the assumptions on the data, we deduce that
\begin{equation} \label{E1}
\sup_{t\in (0,T)} \|\vv^m(t)\|_{2}^2 + \int_0^T \|\vv^m\|^2_{1,2}  \dt \le C\left( \|\vv^m_0\|_2^2 + \int_0^T \|\bfb\|_{(W^{1,2}_{\bn})^*}^2 \dt \right) \le C.
\end{equation}
Note that this estimate is independent of the approximation parameter $m$. We shall also require estimates in stronger norms in what follows, but they will be $n$-dependent. Since we are considering a basis consisting of $W^{3,2}(\Omega)^3$ functions and $W^{3,2}(\Omega)^3\hookrightarrow W^{1,\infty}(\Omega)^3$ in three space dimensions, the above bounds (by norm-equivalence in finite-dimensional spaces) and the identity \eqref{Galerkin:vv} imply that
\begin{align}\label{E1nn}
 \sup_{t\in (0,T)}\left(\|\vv^m(t)\|_{1,\infty} +\|\partial_t \vv^m(t)\|_{\infty}\right)   &\leq C(n).
\end{align}

Next, we derive bounds on $b^m$. Multiplying the $i$-th equation in \eqref{Galerkin:b} by $b_i^m$, summing the resulting equalities with respect to $i=1,\ldots,m$ and using integration by parts and the fact that $\Div \vv^m = 0$, we arrive at the identity
\[ \frac{\dd}{\dd t} \|b^m\|_2^2 + 2\int_{\Omega}\alpha({\tilde \theta}^{\eps}(e^m_+,b^m_M),b^m_M)|\Grad b^m|^2 \; \dd x = -2\int_{\Omega}\! h({\tilde \theta}
^{\eps}(e^m_+,b^m_M),b^m_M)b^m\le C(b_{\min},b_{\max},\Omega) \|b^m\|_2, \]
where we have used \eqref{alpha2} to estimate the right-hand side. Therefore, using Gronwall's inequality and the fact that $\|b_0^m\|_2\le \|b_0\|_2\le C$, we see that
\begin{equation}\label{E2}
\sup_{t\in (0,T)} \|b^m(t)\|_2^2 + \int_0^T \|b^m\|_{1,2}^2 \dt \le C.
\end{equation}
Hence, using the orthogonality of the basis, the estimates \eqref{E1nn} and \eqref{E2}, and the identity \eqref{Galerkin:b}, we get the bound
\begin{equation}\label{E2t}
\int_0^T  \|\partial_t b^m \|_{(W^{1,2}(\Omega))^*}^2 \dt \le C(n).
\end{equation}

Finally, we derive bounds on $e^m$. We multiply \eqref{Galerkin:e} by $e_i^m$ and sum with respect to $i=1,\ldots, m$ to get the identity (we again use $\Div \vv^m=0$)
\begin{equation}\label{ste}
\begin{aligned}
\frac{\dd}{\dd t} \|e^m\|^2_{2} &+ 2\mu \|\Grad e^m\|^2_{2}
+ 2\int_\Omega \kappa({\tilde \theta}_{\eps}(b^m_M,e^m_+), b^m_M) \Grad {\tilde \theta}_{\eps}(b^m_M,e^m)\cdot \Grad e^m \dx \\
&+ 2\int_\Omega  \alpha({\tilde \theta}^{\eps}(e^m_+,b^m_M), b^m_M) \tilde{e}^{\eps}_1({\tilde \theta}^{\eps}(e^m_+,b^m_M)) \psi_2'(b^m_M) \Grad b^m \cdot \Grad e^m \dx\\
&\hspace{-1.6cm}= 4\int_\Omega \nu({\tilde \theta}_{\eps}(b^m_M,e^m_+), b^m_M)\,|\mathbb{D}(\vv^m)|^2 e^m \dx \leq C(n)\|e^m\|_{2},
\end{aligned}
\end{equation}
where we have used the assumption \eqref{alpha1} on $\nu$ and the bound \eqref{E1nn}. We focus on the third and the fourth term on the left-hand side of this inequality, and bound them as well. First, the fourth term can be simply estimated with the help of Young's inequality as
$$
2\int_\Omega  \alpha({\tilde \theta}^{\eps}(e^m_+,b^m_M), b^m_M) \tilde{e}^{\eps}_1({\tilde \theta}^{\eps}(e^m_+,b^m_M)) \psi_2'(b^m_M) \Grad b^m \cdot \Grad e^m \dx\ge -\frac{\mu}{2}\|\Grad e^m\|^2_{2} - C(\eps,\mu)\|b^m\|_{1,2}^2.
$$
Next, we consider the term in~\eqref{ste} containing $\Grad {\tilde \theta}_{\eps}(b^m_M,e^m)$, and
note that for $e^m < 0$ by definition $\Grad {\tilde \theta}_{\eps}(b^m_M,e^m) = \Grad e^m$. For $e^m\geq 0$ on the other hand we use the fact that $\theta^{\eps}$ is the inverse to $e^{\eps}$ and thanks to \eqref{e0e1ep} we have
\begin{equation}\label{nabla:e}
\begin{aligned}
\Grad e^m &= \Grad \left[ \tilde{e}_0({\tilde \theta}^{\eps}(b^m_M,e^m)) + \tilde{\psi}_2(b^m_M) \tilde{e}_1^{\eps}(\tilde {\theta}^{\eps}(b^m_M,e^m)) \right]\\
&=\left[\tilde{e}_0'(({\tilde \theta}^{\eps}(b^m_M,e^m)))+ \tilde{\psi}_2(b^m_M) (\tilde{e}_1^{\eps})'(\tilde {\theta}^{\eps}(b^m_M,e^m))\right]\Grad{\tilde \theta}^{\eps}(b^m_M,e^m) \\
&\quad + \tilde{e}_1^{\eps}(\tilde {\theta}^{\eps}(b^m_M,e^m))\tilde{\psi}_2'(b^m_M)\Grad b^m_M.
\end{aligned}
\end{equation}
Next, using \eqref{incr} and \textbf{(A4)}, we have $\tilde{e}_0'\ge C_1$ and since $\tilde{\psi}_1^{\eps}$ is concave, we have $(\tilde{e}_1^{\eps})'\ge 0$ and thus thanks to the nonnegativity of $\tilde\psi_2$ (which follows from \textbf{(A1)} and \textbf{(A2)}), we obtain from \eqref{nabla:e} that
\begin{align*}
\Grad {\tilde \theta}^{\eps}(b^m_M,e^m) \cdot \Grad e^m &= \left[\tilde{e}_0'(({\tilde \theta}^{\eps}(b^m_M,e^m)))+ \tilde{\psi}_2(b^m_M) (\tilde{e}_1^{\eps})'(\tilde {\theta}^{\eps}(b^m_M,e^m))\right]|\Grad{\tilde \theta}^{\eps}(b^m_M,e^m)|^2  \\
&\quad + \tilde{e}_1^{\eps}(\tilde {\theta}^{\eps}(b^m_M,e^m))\tilde{\psi}_2'(b^m_M)\Grad b^m_M \cdot \Grad {\tilde \theta}^{\eps}(b^m_M,e^m)\\
&  \ge C_1|\Grad{\tilde \theta}^{\eps}(b^m_M,e^m)|^2  -C(\eps)
|\Grad {\tilde \theta}^{\eps}(b^m,e^m)||\Grad b^m_M|\\
&  \ge \frac{C_1}{2}|\Grad{\tilde \theta}^{\eps}(b^m_M,e^m)|^2  -C(\eps)
|\Grad b^m|^2,
\end{align*}
where we have also used the definition of $b^m_M$ and consequently the simple estimate $|\Grad b^m_M|\le |\Grad b^m|$. Hence, using this estimate in \eqref{ste} together with Young's inequality (and the assumption \textbf{(A5)}), we see that
\begin{equation}\label{ste2}
\begin{aligned}
\frac{\dd}{\dd t} \|e^m\|^2_{2} &+ \mu \|\Grad e^m\|^2_{2}
+ C_1\int_\Omega |\Grad {\tilde \theta}^{\eps}(b^m_M,e^m_+)|^2 \dx \le C(\eps,\mu,n)(\|e^m\|_{2}+ \|b^m\|_{1,2}^2).
\end{aligned}
\end{equation}
%
%
%
%
%
%
%
This, together with Gronwall's lemma and the bound \eqref{E2}, gives (recall \eqref{e0eps} for $e^{\eps}_0 \in L^{\infty}(\Omega)$) that
\begin{align}\label{E3}
\sup_{t\in (0,T)}\|e^m(t)\|_{2} + \int_0^T \|e^m\|_{1,2}^2 \dt \le C(\eps, \mu, n).
\end{align}
Hence we deduce from \eqref{Galerkin:e} and the estimates \eqref{E1} and \eqref{E2} that
\begin{align}\label{E3t}
\int_0^T\|\partial_t e^m\|^2_{(W^{1,2}(\Omega))^*}\dt  \leq C(\eps, \mu, n).
\end{align}

\subsubsection*{Passing to the limit $m\to \infty$}
Using the reflexivity  of the underlying spaces together with the fact that $n$ is fixed, we can use \eqref{E1}, \eqref{E1nn}, \eqref{E2}, \eqref{E2t}, \eqref{E3} and \eqref{E3t} to find a subsequence $(\vv^m, e^m, b^m)_{m=1}^{\infty}$ that we do not relabel, such that
\begin{equation}\label{weakm}
\begin{aligned}
\vv^m &\rightharpoonup^* \vv &&\textrm{weakly$^*$ in } L^{\infty}(0,T; W^{3,2}(\Omega)^3 \cap W^{1,2}_{\bn, \diver}),\\
\partial_t\vv^m &\rightharpoonup^* \partial_t\vv &&\textrm{weakly$^*$ in } L^{\infty}(0,T; W^{3,2}(\Omega)^3 \cap W^{1,2}_{\bn, \diver}),\\
b^m &\rightharpoonup b &&\textrm{weakly in } L^{2}(0,T; W^{1,2}(\Omega)) \cap W^{1,2}(0,T; (W^{1,2}(\Omega))^*),\\
e^m &\rightharpoonup e &&\textrm{weakly in } L^{2}(0,T; W^{1,2}(\Omega)) \cap W^{1,2}(0,T; (W^{1,2}(\Omega))^*).
\end{aligned}
\end{equation}
Consequently, using the Aubin--Lions lemma and the trace theorem, we deduce that (taking again a subsequence that we do not relabel)
\begin{equation}\label{strongm}
\begin{aligned}
\vv^m &\to\vv &&\textrm{strongly in } L^{\infty}(0,T; W^{3,2}(\Omega)^3 \cap W^{1,2}_{\bn, \diver}),\\
\vv^m &\to\vv &&\textrm{strongly in } L^{\infty}(0,T;L^{\infty}(\partial \Omega)^3),\\
b^m &\to b &&\textrm{strongly in } L^{2}(0,T; L^{2}(\Omega)),\\
e^m &\to e &&\textrm{strongly in } L^{2}(0,T; L^{2}(\Omega)),\\
b^m &\to b &&\textrm{a.e. in } Q,\\
e^m &\to e &&\textrm{a.e. in } Q.
\end{aligned}
\end{equation}
Hence, thanks to the continuity of $\theta^{\eps}$, $\nu$, $\alpha$ and $\tilde{\bs}^{\eps}$, it is straightforward to pass to the limit $m\to \infty$ in \eqref{Galerkin:vv}--\eqref{Galerkin:e} to deduce that
\begin{equation}
\begin{split}\label{Galerkin:vvm}
\int_{\Omega} \partial_t\vv \cdot \vw_i &+ (2\nu(\tilde{\theta}^{\eps}(e_+,b_M), b_M)\mathbb{D}(\vv)-G_k(|\vv|)\vv\otimes \vv ):\Grad \vw_i\dx+ \int_{\partial \Omega} \tilde{\bs}^{\eps}(\vv_{\tau})\cdot \vw_i \dS\\
&=\langle \bfb, \vw_i\rangle
\end{split}
\end{equation}
for all $i=1,\ldots, n$ and almost all $t\in (0,T)$;
\begin{equation}\label{Galerkin:bm}
\begin{split}
 \langle \partial_t b, u \rangle  + \int_{\Omega}\left(\alpha({\tilde \theta}^{\eps}(e_+,b_M),b_M) \Grad b - b\vv \right)\cdot \Grad u +h({\tilde \theta}
^{\eps}(e_+,b_M), b_M)u \dx=0
\end{split}
\end{equation}
for all $u\in W^{1,2}(\Omega)$ and almost all $t\in (0,T)$;
\begin{equation}\label{Galerkin:em}
\begin{split}
\langle \partial_t e, u\rangle  &+ \int_{\Omega}\left(\mu \Grad e + \kappa({\tilde \theta}^{\eps}(e_+,b_M), b_M)\Grad {\tilde \theta}^{\eps}(e,b_M)
-e\vv\right) \cdot \Grad u\dx \\
&+\int_{\Omega} \left( \alpha({\tilde \theta}^{\eps}(e_+,b_M),b_M) \tilde{e}^{\eps}_1({\tilde \theta}^{\eps}(e_+,b_M)) \psi_2'(b_M)\Grad b  \right)\cdot \Grad u \dx\\
&\hspace{-1.3cm}=\int_{\Omega} 2\nu({\tilde \theta}^{\eps}(e_+,b_M), b_M)\,|\mathbb{D}(\vv)|^2 u\dx
\end{split}
\end{equation}
for all $u\in W^{1,2}(\Omega)$ and almost all $t\in (0,T)$.

In addition, it follows from \eqref{weakm} and parabolic embedding that $b,e \in \mathcal{C}([0,T]; L^2(\Omega))$ and $\vv\in \mathcal{C}([0,T]; W^{1,2}_{\bn, \diver})$, and it is also straightforward to show that the system \eqref{Galerkin:vvm}--\eqref{Galerkin:em} is complemented by the following initial conditions:
\begin{equation}\label{Galerkininitm}
\begin{aligned}
\vv (0) &=\vv_0^n= \sum_{i=1}^n c_i^n (0)\vw_i, &&c_i^n(0):=  \int_{\Omega} \vv_0 \cdot \vw_i \dx,\\
b (0) &= b_0,\qquad e(0) = e_0^{\eps}.
\end{aligned}
\end{equation}

\subsection{Maximum and minimum principles for $e$ and $b$}
We start this section by deriving uniform bounds on $b$. We recall \eqref{v0-cond}, i.e., there exist $b_{\min}$  and $b_{\max}$ such that $0<b_{\min} < 1 < b_{\max}$ and
$$
b_{\min}\le b_0\le b_{\max}
$$
almost everywhere in $\Omega$. Then we set $u:=(b-b_{\max})_+$ in \eqref{Galerkin:bm} and integrate over the time interval $(0,\tau)$ to get (the convective term again vanishes)
\[
\begin{aligned}
&\|(b-b_{\rm max})_+(\tau)\|^2_{2}  + 2\!\int_0^{\tau}\!\!\int_{\Omega}\alpha({\tilde \theta}^{\eps}(e_+,b_M),b_M) |\Grad (b-b_{\rm max})|^2 +2h({\tilde \theta}
^{\eps}(e_+,b_M),b_M)(b-b_{\max})_+ \dx \dt\\
&\quad  =\|(b_0-b_{\rm max})_+\|^2_{2}=0.
\end{aligned}
\]
%
%
The first integrand on the left-hand side is nonnegative thanks to the nonnegativity of $\alpha$ (\eqref{alpha1} in the assumption \textbf{(A5)}). For the second integrand we use \eqref{alpha2} from  \textbf{(A5)} and deduce that it is also nonnegative. Indeed, we see that the second integrand can be nonzero only on the set where $b > b_{\max}\ge 1$. However, using \eqref{alpha2}, we have that $h$ must be nonnegative on that set and the same is true of $(b-b_{\max})_+$. Hence the second integrand is nonnegative. Consequently, $b\le b_{\max}$ almost everywhere in $(0,T)\times \Omega$. In a similar way, we may set $u:=(b-b_{\min})_{-}$ in \eqref{Galerkin:bm} to deduce that $b\ge b_{\min}$ almost everywhere. To summarise, we deduce the two-sided uniform bound
\begin{equation}\label{inftyb}
b_{\min}\le b(t,x)\le b_{\max} \quad \textrm{for almost all }(t,x)\in (0,T)\times \Omega.
\end{equation}
Therefore, we can simply replace $b_M$ by $b$ in \eqref{Galerkin:vvm}--\eqref{Galerkin:em}, which follows from the definition of $b_M$ (see \eqref{cutb}) and the bound \eqref{inftyb}.

Next, we show that $e$ is strictly positive almost everywhere and therefore we can again simply replace $e_+$ by $e$ in the equations mentioned above.
Using the definition \eqref{e0e1ep} of $\tilde{e}^{\eps}$, the fact that $\tilde{e}_0$ and $\tilde{e}_1^{\eps}$ are nondecreasing (see \eqref{incr}) and the fact that $\tilde{\psi}_2$ has minimum at one, which follows from the convexity of $\tilde{\psi}_2$ (cf. \textbf{(A1)}) and  \textbf{(A2)}, we deduce from \eqref{e0eps} that
\begin{equation}\label{e0epsmen}
e_0^{\varepsilon}(x):=\tilde{e}^{\eps}\left(\max\left\{\eps, \min\{\theta_0(x),\eps^{-1}\}\right\},b_0(x)\right)\ge \tilde{e}^{\eps}\left(\eps/2, 1\right)=:\delta.
\end{equation}
Then we set $u:=(e- \delta)_{-}$ in \eqref{Galerkin:em} and after integration over $t\in (0,\tau)$, we deduce after using the fact that $\nu$ is nonnegative that
\begin{equation}\label{e:min}
\begin{split}
\|(e(\tau)-\delta)_{-}\|_2^2 &+2\int_0^\tau \int_{\Omega}\mu |\Grad (e-\delta)_{-}|^2  + \kappa({\tilde \theta}^{\eps}(e_+,b), b)\Grad {\tilde \theta}^{\eps}(e,b)\cdot \Grad (e-\delta)_{-}\dx\dt  \\
&+2\int_0^{\tau}\int_{\Omega}  \alpha({\tilde \theta}^{\eps}(e_+,b),b) \tilde{e}^{\eps}_1({\tilde \theta}^{\eps}(e_+,b)) \psi_2'(b)\Grad b  \cdot \Grad (e-\delta)_{-} \dx \dt \le 0.
\end{split}
\end{equation}
Our goal is to show that both integrals on the left-hand side are nonnegative. Note that we always integrate over the set where $e<\delta$. First we consider the set where $e\le 0$. However, on this set we have (see \eqref{l00}) that
$$
\tilde{e}^{\eps}_1({\tilde \theta}^{\eps}(e_+,b))=\tilde{e}^{\eps}_1({\tilde \theta}^{\eps}(0,b))=\tilde{e}^{\eps}_1(0)=0
$$
and therefore the last integrand vanishes. For the second integrand, we use the definition of $\tilde{\theta}^{\eps}(e,b)$ for negative $e$ and we see that $\Grad \tilde{\theta}^{\eps}(e,b)=\Grad e$. In summary then, over the set where $e \leq 0$ the second integral is equal to $0$ while the first integral is nonnegative and
can be therefore dropped from the left-hand side to deduce from \eqref{e:min} that
\begin{equation}\label{e:min2}
\begin{split}
& \|(e(\tau)-\delta)_{-}\|_2^2 +2\int_0^\tau \int_{\Omega} \kappa({\tilde \theta}^{\eps}(e_+,b), b)\Grad {\tilde \theta}^{\eps}(e,b)\cdot \Grad e\chi_{\{0\le e\le \delta\}}\dx\dt  \\
&+2\int_0^{\tau}\int_{\Omega}  \alpha({\tilde \theta}^{\eps}(e_+,b),b) \tilde{e}^{\eps}_1({\tilde \theta}^{\eps}(e,b)) \psi_2'(b)\Grad b  \cdot \Grad e \chi_{\{0\le e\le \delta\}} \dx \dt \le 0.
\end{split}
\end{equation}
Next, we use the fact that $\tilde{e}^{\eps}(\theta,b)$ is, for fixed $b$, increasing with respect to $\theta$ and, for fixed $\theta$, it attains its minimum at $b=1$. Thus we deduce from \eqref{e0epsmen} that
$$
\delta \ge e \ge 0\quad \implies \quad \tilde{e}^\eps(\eps/2,1)\ge \tilde{e}^{\eps}(\tilde{\theta}^{\eps}(e,b),b)\ge \tilde{e}^{\eps}(\tilde{\theta}^{\eps}(e,b),1)\quad \implies \quad \eps/2 \ge \tilde{\theta}^{\eps}(e,b).
$$
Therefore, it follows from \eqref{e0e1ep} that
$$
\delta \ge e \ge 0\quad \implies e=\tilde{e}^{\eps}(\tilde{\theta}^{\eps}(e,b),b)=\tilde{e}_0(\tilde{\theta}^{\eps}(e,b)) \quad \textrm{ and }\quad \tilde{e}_1^{\eps}(\tilde{\theta}^{\eps}(e,b))=0.
$$
Consequently, the last term in \eqref{e:min2} vanishes and for the second term we can use the identity $\Grad  e = \tilde{e}'_0(\tilde{\theta}^{\eps}(e,b)) \Grad \tilde{\theta}^{\eps}(e,b)$. Thus, both integrals are nonnegative and we therefore have
\begin{equation*}
\begin{split}
& \|(e(\tau)-\delta)_{-}\|_2^2=0,
\end{split}
\end{equation*}
which implies that
\begin{equation} \label{e:min3}
e(t,x) \ge \delta \quad \textrm{ for almost all } (t,x) \in (0,T)\times \Omega.
\end{equation}

Having shown the positivity of $e$ a.e. on $(0,T) \times \Omega$ we next show the positivity of $\tilde{\theta}^\eps$. If there existed a nonempty set where $\tilde{\theta}^{\eps}(e,b)< \eps/2$, then on such a set we would have
$$
e=\tilde{e}^{\eps}(\tilde{\theta}^{\eps}(e,b),b)= \tilde{e}_0(\tilde{\theta}^{\eps}(e,b)) < \tilde{e}_0(\eps/2)=\tilde{e}^{\eps}(\eps/2,1)=\delta,
$$
which would be in contradiction with \eqref{e:min3}. Hence, it follows that
\begin{equation} \label{e:min3t}
\tilde{\theta}^{\eps}(e(t,x),b(t,x)) \ge \eps/2 \quad \textrm{ for almost all } (t,x) \in (0,T)\times \Omega.
\end{equation}

\subsection{The limit $\mu \to 0$}\label{mutozero}
In this section we let $\mu \to 0_+$ in \eqref{Galerkin:vvm}--\eqref{Galerkininitm}. We denote by $(\vv^\mu,b^\mu,e^\mu)$ the sequence of solutions constructed in Section~\ref{mtoinfty}. Furthermore, we define
\begin{equation}
\label{thetamu}
\theta^{\mu}(t,x):= \tilde{\theta}^{\eps}(e^{\mu}(t,x),b^{\mu}(t,x)).
\end{equation}
Then using weak lower semicontinuity and the estimates \eqref{E1}, \eqref{E2} and \eqref{inftyb}, we see that
\begin{equation} \label{E1mu}
\sup_{t\in (0,T)} \|\vv^{\mu}(t)\|_{2}^2 + \int_0^T \|\vv^{\mu}\|^2_{1,2}+ \|b^{\mu}\|_{1,2}^2   \dt \le C
\end{equation}
and
\begin{equation}\label{inftybmu}
b_{\min}\le b^{\mu}(t,x)\le b_{\max} \quad \textrm{for almost all }(t,x)\in (0,T)\times \Omega.
\end{equation}
Since, $b^{\mu}$ is bounded from below and above independently of any parameter, we can now strengthen the nonuniform estimate \eqref{E2t}: it follows from \eqref{E1mu}, \eqref{inftyb} and \eqref{Galerkin:bm} that
\begin{equation}\label{E2tmu}
\int_0^T  \|\partial_t b^{\mu} \|_{(W^{1,2}(\Omega))^*}^2 \dt \le C.
\end{equation}
Furthermore, we still have a nonuniform estimate in stronger norms coming from \eqref{E1nn}, i.e.,
\begin{align}\label{E1nmu}
 \sup_{t\in (0,T)}\left(\|\vv^\mu(t)\|_{1,\infty} +\|\partial_t \vv^{\mu}(t)\|_{\infty}\right)   &\leq C(n).
\end{align}

Finally, we derive bounds on $e^{\mu}$ and $\theta^{\mu}$. We simply set $u:=e^{\mu}$ in \eqref{Galerkin:em} to get the identity
\begin{equation}\label{testmu}
\begin{split}
& \frac{\dd}{\dt}\|e^{\mu}\|_2^2 + 2\mu \|\nabla_x e^{\mu}\|_2^2 + 2\int_{\Omega}\kappa(\theta^{\mu},b^{\mu})\Grad \theta^{\mu} \cdot \Grad e^{\mu}+ \alpha(\theta^{\mu},b^{\mu}) \tilde{e}^{\eps}_1(\theta^{\mu}) \psi_2'(b^{\mu})\Grad b^{\mu} \cdot \Grad e^{\mu}\dx\\
&=4\int_{\Omega} \nu(\theta^{\mu},b^{\mu})\,|\mathbb{D}(\vv^{\mu})|^2 e^{\mu}\dx.
\end{split}
\end{equation}
Next, using the same type of computation as in~\eqref{nabla:e} and \textbf{(A4)} and also the bound \eqref{inftybmu}, we deduce that
$$
\begin{aligned}
\Grad \theta^{\mu} \cdot \Grad e^{\mu}&=\left[\tilde{e}_0'(\theta^{\mu})+ \tilde{\psi}_2(b^{\mu}) (\tilde{e}_1^{\eps})'(\theta^{\mu})\right]|\Grad \theta^{\mu}|^2 + \tilde{e}_1^{\eps}(\theta^{\mu})\tilde{\psi}_2'(b^{\mu})\Grad b^{\mu}\cdot \Grad \theta^{\mu},\\
&\ge C_1|\Grad \theta^{\mu}|^2 - C|\Grad b^{\mu}|^2,\\
|\Grad e^{\mu}|&\le C(\eps) (|\Grad \theta^{\mu}| +|\Grad b^{\mu}|).
\end{aligned}
$$
Thus, using these estimates in \eqref{testmu} and combining them with \eqref{E1mu}, \eqref{E1nmu} and using Gronwall's lemma, we have that
\begin{align}\label{E3mu}
\sup_{t\in (0,T)}\|e^\mu(t)\|_{2} + \int_0^T \|e^{\mu}\|_{1,2}^2+ \|\theta^{\mu}\|_{1,2}^2 \dt \le C(\eps, n),
\end{align}
and hence we deduce from \eqref{Galerkin:em} and the estimates \eqref{E1mu} and \eqref{E3mu} that
\begin{align}\label{E3tmu}
\int_0^T\|\partial_t e^\mu\|^2_{(W^{1,2}(\Omega))^*}\dt  \leq C(\eps, n).
\end{align}

We are now in exactly the same position as when  we passed to the limit $m\to \infty$, and thus we can again extract a subsequence (not indicated) such that
\begin{equation}\label{weakmu}
\begin{aligned}
\vv^\mu &\rightharpoonup^* \vv &&\textrm{weakly$^*$ in } L^{\infty}(0,T; W^{3,2}(\Omega)^3 \cap W^{1,2}_{\bn, \diver}),\\
\partial_t\vv^\mu &\rightharpoonup^* \partial_t\vv &&\textrm{weakly$^*$ in } L^{\infty}(0,T; W^{3,2}(\Omega)^3 \cap W^{1,2}_{\bn, \diver}),\\
b^\mu &\rightharpoonup b &&\textrm{weakly in } L^{2}(0,T; W^{1,2}(\Omega)) \cap W^{1,2}(0,T; (W^{1,2}(\Omega))^*),\\
\theta^\mu &\rightharpoonup \theta &&\textrm{weakly in } L^{2}(0,T; W^{1,2}(\Omega)),\\
e^\mu &\rightharpoonup e &&\textrm{weakly in } L^{2}(0,T; W^{1,2}(\Omega)) \cap W^{1,2}(0,T; (W^{1,2}(\Omega))^*).
\end{aligned}
\end{equation}
Consequently, using the Aubin--Lions lemma and the trace theorem, we deduce that
\begin{equation}\label{strongmu}
\begin{aligned}
\vv^\mu &\to\vv &&\textrm{strongly in } L^{\infty}(0,T; W^{3,2}(\Omega)^3 \cap W^{1,2}_{\bn, \diver}),\\
\vv^\mu &\to\vv &&\textrm{strongly in } L^{\infty}(0,T;L^{\infty}(\partial \Omega)^3),\\
b^\mu &\to b &&\textrm{strongly in } L^{2}(0,T; L^{2}(\Omega)),\\
e^\mu &\to e &&\textrm{strongly in } L^{2}(0,T; L^{2}(\Omega)),\\
b^\mu &\to b &&\textrm{a.e. in } Q,\\
e^\mu &\to e &&\textrm{a.e. in } Q.
\end{aligned}
\end{equation}
Finally, since for fixed $\eps$ the function $\tilde{\theta}^{\eps}$ is continuous, we also have that
\begin{equation}\label{strongmu2}
\begin{aligned}
\theta^\mu &\to \theta &&\textrm{strongly in } L^{2}(0,T; L^{2}(\Omega)),\\
\theta^\mu &\to \theta &&\textrm{a.e. in  } Q,
\end{aligned}
\end{equation}
where
$$
\theta=\tilde{\theta}^{\eps}(e,b).
$$
Now we are ready to let $\mu\to 0_+$ in \eqref{Galerkin:vvm}--\eqref{Galerkininitm} to get
\begin{equation}
\begin{split}\label{Galerkin:vvmu}
\int_{\Omega} \partial_t\vv \cdot \vw_i + (2\nu(\theta,b)\mathbb{D}(\vv)-G_k(|\vv|)\vv\otimes \vv ):\Grad \vw_i\dx+ \int_{\partial \Omega} \tilde{\bs}^{\eps}(\vv_{\tau})\cdot \vw_i \dS=\langle \bfb, \vw_i\rangle
\end{split}
\end{equation}
for all $i=1,\ldots, n$ and almost all $t\in (0,T)$;
\begin{equation}\label{Galerkin:bmu}
\begin{split}
 \langle \partial_t b, u \rangle  + \int_{\Omega}\left(\alpha(\theta,b) \Grad b - b\vv \right)\cdot \Grad u +h(\theta,b)u \dx=0
\end{split}
\end{equation}
for all $u\in W^{1,2}(\Omega)$ and almost all $t\in (0,T)$;
\begin{equation}\label{Galerkin:emu}
\begin{split}
\langle \partial_t e, u\rangle  &+ \int_{\Omega}\left( \kappa(\theta,b)\Grad \theta
-e\vv\right) \cdot \Grad u\dx +\int_{\Omega} \alpha(\theta,b) \tilde{e}^{\eps}_1(\theta) \psi_2'(b)\Grad b  \cdot \Grad u \dx\\
&=\int_{\Omega} 2\nu(\theta,b)\,|\mathbb{D}(\vv)|^2 u\dx
\end{split}
\end{equation}
for all $u\in W^{1,2}(\Omega)$ and almost all $t\in (0,T)$, with associated initial conditions
\begin{equation}\label{Galerkininitmu}
\begin{aligned}
\vv (0) &=\vv_0^n= \sum_{i=1}^n c_i^n (0)\vw_i, &&c_i^n(0):=  \int_{\Omega} \vv_0 \cdot \vw_i \dx,\\
b (0) &= b_0,\qquad e(0) = e_0^{\eps}.
\end{aligned}
\end{equation}

\subsection{The limits $\eps \to 0_+$ and $n\to \infty$}
In this section we set $\eps^n:=1/n$ and let $n\to \infty$. This means that the perturbed free energy converges to the original free energy $\tilde{\psi}$ and we can also pass to the limit in the momentum equation. The only remaining parameter will then be $k$, which corresponds to the cut-off in the convective term. We denote by $(\vv^n, b^n, e^n)$ the solution constructed in the previous section and also set $\theta^n:=\tilde{\theta}^{\eps^n}(e^n,b^n)$; our goal is to let $n\to \infty$ in \eqref{Galerkin:vvmu}--\eqref{Galerkininitmu}, where we set $\eps:=\eps^n = 1/n$.

\subsubsection*{Uniform estimates}
We start again by stating the relevant uniform estimates. Using weak lower semicontinuity in \eqref{E1mu} and \eqref{E2tmu} and the strong convergence \eqref{strongmu} in \eqref{inftybmu}, we deduce that
\begin{equation} \label{E1n}
\sup_{t\in (0,T)} \|\vv^{n}(t)\|_{2}^2 + \int_0^T \|\vv^{n}\|^2_{1,2}+ \|b^{n}\|_{1,2}^2 + \|\partial_t b^{n} \|_{(W^{1,2}(\Omega))^*}^2  \dt \le C
\end{equation}
and
\begin{equation}\label{inftybn}
b_{\min}\le b^{n}(t,x)\le b_{\max} \quad \textrm{for almost all }(t,x)\in (0,T)\times \Omega.
\end{equation}
Further, by using the interpolation inequality
$$
\|\vv\|_{\frac{10}{3}}^{\frac{10}{3}}\le C(\Omega)\|\vv\|^{\frac{4}{3}}_2 \|\vv\|^2_{1,2},
$$
we also obtain with the help of \eqref{E1n} that
\begin{equation} \label{E2n}
\int_0^T \|\vv^{n}\|^{\frac{10}{3}}_{\frac{10}{3}} \dt \le C.
\end{equation}
Next, thanks to the presence of $G_k$ in \eqref{Galerkin:vvmu}, we deduce by using \eqref{E1n} that
\begin{equation}\label{E1tn}
\int_0^T \|\vv^n\|_{(W^{1,2}_{\bn,\diver})^*}^2\dt\le C(k).
\end{equation}

Finally, we focus on the bounds on $e^n$ and $\theta^n$, respectively. Here, the energy equality and the  entropy inequality will play an essential role. Note that all of the bounds on $e^n$ and $\theta^n$ that we shall derive are independent of the parameters $n$ and $k$. First, we set $u=1$ in \eqref{Galerkin:emu} and integrate over $t\in (0,\tau)$ to get (using the fact that $e^{n}\ge 0$)
$$
\|e^{n}(\tau)\|_1 = \|e^{\eps}_0\|_1+ \int_0^{\tau} \int_{\Omega}2\nu(\theta^n, b^{n})\,|\mathbb{D}(\vv^{n})|^2 \dx\dt\le C,
$$
where we have used the uniform bound \eqref{E1n}, the assumption \textbf{(A5)} on $\nu$, the definition of $e_0^{\eps}$, and the fact that $\tilde{e}(\theta_0,b_0) \in L^1(\Omega)$. Consequently,
\begin{equation}\label{E3n}
\sup_{t\in (0,T)} \|e^{n}(t)\|_{1}\le C.
\end{equation}
Hence, it follows from \eqref{ule0} and the fact that $e^n \ge \tilde{e}_0(\theta^n)$ that
\begin{equation}\label{E3nb}
\sup_{t\in (0,T)} \|\theta^{n}(t)\|_{1}\le C.
\end{equation}

Next, we note that since $\theta^n \ge \eps^n/2$ it follows from \eqref{weakmu} that we can set $u:=\frac{1}{\theta^n}$ in \eqref{Galerkin:emu} to deduce that
\begin{equation}\label{entropy1}
\begin{aligned}
&\langle \partial_t e^n, \frac{1}{\theta^n} \rangle +\int_{\Omega} \frac{\vv^n \cdot \Grad e^n}{\theta^n}\dx\\
&= \int_{\Omega} \frac{\kappa(\theta^n,b^n)|\Grad \theta^n|^2}{(\theta^n)^2}+ \frac{\alpha(\theta^n,b^n) \tilde{e}^{\eps}_1(\theta^n) \psi_2'(b^n)}{(\theta^n)^2}\Grad b  \cdot \Grad \theta^n +  \frac{2\nu(\theta^n, b^n)|\mathbb{D}(\vv^n)|^2}{\theta^n} \dd x.
\end{aligned}
\end{equation}
Similarly, we set $u:=\frac{\tilde{\psi}^{\eps}_1(\theta^n) \psi_2'(b^n)}{\theta^n}$ in \eqref{Galerkin:bmu} (note that thanks to the minimum principle this is an admissible test function) to get (recall that $\tilde{e}_1^{\eps}(\theta)=\tilde{\psi}_1^{\eps}(\theta) - \theta (\tilde{\psi}_1^{\eps})'(\theta)$)
\begin{equation}\label{entropy2}
\begin{split}
 &\left\langle \partial_t b^n, \frac{\tilde{\psi}^{\eps}_1(\theta^n) \psi_2'(b^n)}{\theta^n} \right\rangle  + \int_{\Omega} \frac{\tilde{\psi}^{\eps}_1(\theta^n) \psi_2'(b^n)}{\theta^n} \vv^n \cdot \Grad b^n \dd x\\
 &=-\int_{\Omega}\alpha(\theta^n,b^n) \Grad b^n \cdot \Grad \left(\frac{\tilde{\psi}^{\eps}_1(\theta^n) \psi_2'(b^n)}{\theta^n}\right) +\frac{\tilde{\psi}^{\eps}_1(\theta^n) \psi_2'(b^n)h(\theta^n,b^n)}{\theta^n} \dx\\
 &=-\int_{\Omega} \frac{\alpha(\theta^n,b^n)\tilde{\psi}^{\eps}_1(\theta^n) \psi_2''(b^n)}{\theta^n} |\Grad b^n|^2  +\frac{\tilde{\psi}^{\eps}_1(\theta^n) \psi_2'(b^n)h(\theta^n,b^n)}{\theta^n} \dx\\
&\qquad+\int_{\Omega} \frac{\alpha(\theta^n,b^n)\psi_2'(b^n)\tilde{e}_{1}^{\eps}(\theta^n)}{(\theta^n)^2} \Grad b^n \cdot \Grad \theta^n\dx
\end{split}
\end{equation}
Thus, by subtracting \eqref{entropy2} from \eqref{entropy1}, we obtain
\begin{equation}\label{entropy}
\begin{aligned}
&\langle \partial_t e^n, \frac{1}{\theta^n} \rangle -\left\langle \partial_t b^n, \frac{\tilde{\psi}^{\eps}_1(\theta^n) \psi_2'(b^n)}{\theta^n} \right\rangle +\int_{\Omega} \frac{\vv^n \cdot \Grad e^n}{\theta^n}-\frac{\tilde{\psi}^{\eps}_1(\theta^n) \psi_2'(b^n)}{\theta^n} \vv^n \cdot \Grad b^n \dx\\
&= \int_{\Omega} \frac{\kappa(\theta^n,b^n)|\Grad \theta^n|^2}{(\theta^n)^2}+  \frac{2\nu(\theta^n, b^n)|\mathbb{D}(\vv^n)|^2}{\theta^n} \dd x \\
&\quad +\int_{\Omega} \frac{\alpha(\theta^n,b^n)\tilde{\psi}^{\eps}_1(\theta^n) \psi_2''(b^n)}{\theta^n} |\Grad b^n|^2  +\frac{\tilde{\psi}^{\eps}_1(\theta^n) \psi_2'(b^n)h(\theta^n,b^n)}{\theta^n} \dx.
\end{aligned}
\end{equation}
We focus now on the left-hand side, where we use the definition of the entropy $\tilde{\eta}^{\eps}$. Indeed, let $\partial$ be a partial derivative with respect to the variable $t$ or $x_i$, $i \in \{1,2,3\}$. Then, by using \eqref{e0e1ep}, we deduce that
\begin{equation}\label{evale}
\begin{split}
\partial e^n& = \partial \left(\tilde{e}_0(\theta^n) + \tilde{e}_1^{\eps}(\theta^n) \tilde{\psi}_2(b^n) \right)= (\tilde{e}_0'(\theta^n) + (\tilde{e}_1^{\eps})'(\theta^n) \tilde{\psi}_2(b^n))\partial \theta^n + \tilde{e}_1^{\eps}(\theta^n) \tilde{\psi}_2'(b^n) \partial b^n\\
&= (-\theta^n \tilde{\psi}_0''(\theta^n) -\theta^n (\tilde{\psi}^{\eps}_1)''(\theta^n) \tilde{\psi}_2(b^n))\partial \theta^n + (\tilde{\psi}_1^{\eps}(\theta^n)-\theta^n (\tilde{\psi}_1^{\eps})'(\theta^n) ) \tilde{\psi}_2'(b^n) \partial b^n.
\end{split}
\end{equation}
Consequently,
\begin{equation}\label{evale2}
\begin{split}
\frac{\partial e^n}{\theta^n} - \frac{\tilde{\psi}^{\eps}_1(\theta^n) \psi_2'(b^n)\partial b^n}{\theta^n}
&= (-\tilde{\psi}_0''(\theta^n) -(\tilde{\psi}^{\eps}_1)''(\theta^n) \tilde{\psi}_2(b^n))\partial \theta^n -(\tilde{\psi}_1^{\eps})'(\theta^n)  \tilde{\psi}_2'(b^n) \partial b^n\\
&= \partial(-\tilde{\psi}_0'(\theta^n) -(\tilde{\psi}^{\eps}_1)'(\theta^n) \tilde{\psi}_2(b^n))\\
&=\partial \tilde{\eta}^{\eps}(\theta^n, b^n).
\end{split}
\end{equation}
Hence, using this relation in \eqref{entropy} and integrating the result over $t\in (0,\tau)$, we get
\begin{equation}\label{entropy3}
\begin{aligned}
&\int_{\Omega}(\tilde{\eta}^{\eps}(\theta^n(\tau),b^n(\tau))-\tilde{\eta}^{\eps}(\theta^n(0),b^n(0)))\dx \\
&= \int_0^{\tau}\int_{\Omega} \frac{\kappa(\theta^n,b^n)|\Grad \theta^n|^2}{(\theta^n)^2}+  \frac{2\nu(\theta^n, b^n)|\mathbb{D}(\vv^n)|^2}{\theta^n} \dd x \dt\\
&\quad +\int_0^{\tau}\int_{\Omega} \frac{\alpha(\theta^n,b^n)\tilde{\psi}^{\eps}_1(\theta^n) \tilde{\psi}_2''(b^n)}{\theta^n} |\Grad b^n|^2  +\frac{\tilde{\psi}^{\eps}_1(\theta^n) \psi_2'(b^n)h(\theta^n,b^n)}{\theta^n} \dx\dt.
\end{aligned}
\end{equation}
We first focus on the right-hand side. For the last term, we use the fact that $\tilde{\psi}_2$ is convex and then thanks to \textbf{(A2)}, we have $\tilde{\psi}_2'(b) (b-1) \ge 0$. Consequently, using also \eqref{alpha2}, we see that the second summand in the second integral on the right-hand side of \eqref{entropy3} is nonnegative. For the remaining terms on the right-hand side, we use \textbf{(A5)} to obtain
\begin{equation}\label{entropy4}
\begin{aligned}
&\int_{\Omega}(\tilde{\eta}^{\eps}(\theta^n(\tau),b^n(\tau))-\tilde{\eta}^{\eps}(\theta^n(0),b^n(0)))\dx \\
&\ge C_1\int_0^{\tau} \int_{\Omega} \frac{|\Grad \theta^n|^2}{(\theta^n)^2}+  \frac{|\mathbb{D}(\vv^n)|^2}{\theta^n}  +\frac{\tilde{\psi}^{\eps}_1(\theta^n) \tilde{\psi}_2''(b^n)|\Grad b^n|^2}{\theta^n}  \dx\dt.
\end{aligned}
\end{equation}
Next, we focus on the term on the left-hand side of \eqref{entropy4}. First, it follows from \eqref{e0epsmen} that
\begin{equation}\label{theta0}
\theta^n(0)=\max\left\{\eps^n, \min\{\theta_0(x),(\eps^n)^{-1}\}\right\}\ge \eps^n.
\end{equation}
Consequently,
$$
|\tilde{\eta}^{\eps}(\theta^n(0),b^n(0))|=|\tilde{\eta}(\theta^n(0),b^n(0))|=|\tilde{\eta}( \max\left\{\eps^n, \min\{\theta_0(x),(\eps^n)^{-1}\}\right\},b_0)|\le C(1+|\tilde{\eta}(\theta_0,b_0)|).
$$
To estimate the first term in the integrand on the left-hand side of \eqref{entropy4} from above, we use Lemma~\ref{L:psi}, the estimate \eqref{etaup} and the fact that $(\tilde{\psi}_1^{\eps})'\ge \tilde{\psi}_1'$ together with
and \eqref{etaup}, to deduce that
\begin{align*}
\tilde{\eta}^{\eps}(\theta^n,b^n)
&\le -C_1|\ln \theta^n| + C(1+ \theta^n).
\end{align*}
Using these two inequalities in \eqref{entropy4} we arrive at the bound
\begin{equation}\label{entropy5}
\begin{aligned}
&\sup_{t\in (0,T)} \|\ln \theta^n(t)\|_1 + \int_0^{T} \|\nabla_x \ln \theta^n\|_2^2 +  \left\|\frac{\mathbb{D}(\vv^n)}{\sqrt{\theta^n}}\right\|_2^2
+ \left\| \frac{\sqrt{\tilde{\psi}^{\eps}_1(\theta^n) \tilde{\psi}_2''(b^n)}\Grad b^n}{\sqrt{\theta^n}} \right\|_2^2 \dt\\
 &\quad \le C(1+\|\tilde{\eta}(\theta_0,b_0)\|_1 + \sup_{t\in (0,T)} \|\theta^n(t)\|_1)\le C,
\end{aligned}
\end{equation}
where for the last inequality we used the assumptions on the data and the uniform bound \eqref{E3nb}. The key information from the estimate \eqref{entropy5} is that $\theta^n$ remains positive almost everywhere uniformly (logarithmically uniformly). The estimate above however does not provide sufficient information about large values of $\theta^n$. Thus to improve it, we consider $\delta\in (0,1)$ and set $u:=(1+e^{n})^{-\delta}$ in \eqref{Galerkin:emu} to obtain (using again $\Div \vv^n=0$)
\begin{equation}\label{e:est1}
\begin{split}
&  \frac{1}{1-\delta}\frac{\dd}{\dt} \int_{\Omega} (1+e^n)^{1-\delta} \dx  -\delta \int_{\Omega}\frac{\kappa(\theta^n, b^n)}{(1+e^{n})^{1+\delta}}\Grad \theta^n \cdot \Grad e^{n}\dx \\
&\quad -\delta\int_{\Omega} \frac{\alpha(\theta^n, b^{n})\tilde{e}^{\eps}_1(\theta^{n}) \psi_2'(b^{n})}{(1+e^{n})^{1+\delta}}\Grad b^{n}  \cdot \Grad e^{n} \dx\\
&=\int_{\Omega} 2\frac{\nu(\theta^{n},b^{n})}{(1+e^{n})^{\delta}}\,|\mathbb{D}(\vv^n)|^2 \dx.
\end{split}
\end{equation}
Next, using the first line in \eqref{evale}, we find that
$$
\begin{aligned}
\Grad \theta^{n} \cdot \Grad e^{n}&=\left[\tilde{e}_0'(\theta^{n})+ \tilde{\psi}_2(b^{n}) (\tilde{e}_1^{\eps})'(\theta^{n})\right]|\Grad \theta^{n}|^2 + \tilde{e}_1^{\eps}(\theta^{n})\tilde{\psi}_2'(b^{n})\Grad b^{n}\cdot \Grad \theta^{n},\\
\Grad b^{n} \cdot \Grad e^{n}&=\left[\tilde{e}_0'(\theta^{n})+ \tilde{\psi}_2(b^{n}) (\tilde{e}_1^{\eps})'(\theta^{n})\right]\Grad \theta^{n} \cdot \Grad b^{n}  + \tilde{e}_1^{\eps}(\theta^{n})\tilde{\psi}_2'(b^{n})|\Grad b^{n}|^2.
\end{aligned}
$$
Thus, by Young's inequality, \eqref{boundapp}--\eqref{e1epsest},  \textbf{(A2)}, the consequences following \textbf{(A4)} and the uniform bound \eqref{inftybn}, we have
\begin{equation}\label{explicite}
\begin{aligned}
\Grad \theta^{n} \cdot \Grad e^{n}&\ge \tilde{e}_0'(\theta^{n})|\Grad \theta^{n}|^2 - \tilde{e}_1^{\eps}(\theta^{n})|\tilde{\psi}_2'(b^{n})||\Grad b^{n}||\Grad \theta^{n}|\\
&\ge \frac{C_1}{2}|\Grad \theta^{n}|^2 - C|\Grad b^{n}|^2,\\
|\Grad b^{n} \cdot \Grad e^{n}|&\le \left[\tilde{e}_0'(\theta^{n})+ \tilde{\psi}_2(b^{n}) (\tilde{e}_1^{\eps})'(\theta^{n})\right]|\Grad \theta^{n}| |\Grad b^{n}|  + C|\Grad b^{n}|^2\\
&\le C(|\Grad \ln \theta^n|^2 +|\Grad b^n|^2 +  |\Grad \theta^{n}| |\Grad b^{n}|).
\end{aligned}
\end{equation}
Therefore, substituting the above inequalities into \eqref{e:est1}, using Young's inequality, the assumption \textbf{(A5)}, integrating the result with respect to $t\in (0,T)$ and recalling the uniform bounds \eqref{E1n} and \eqref{entropy5}, we see that
\begin{equation}\label{entropy6}
\begin{split}
\int_0^T \int_{\Omega} \frac{|\Grad \theta^n|^2}{(1+e^n)^{1+\delta}} \dx \dt \le \frac{C}{\delta(1-\delta)}.
\end{split}
\end{equation}
Next, thanks to \eqref{ule0}, \eqref{ule1}, \eqref{e1epsest} and the uniform bound \eqref{entropy5}, we see that \eqref{entropy6} leads to
\begin{equation}\label{entropy7}
\begin{split}
\int_0^T \int_{\Omega} \frac{|\Grad \theta^n|^2}{(1+\theta^n)^{1+\delta}} \dx \dt\le \int_0^T \int_{\Omega} \frac{|\Grad \theta^n|^2}{(1+e^n)^{1+\delta}} + \frac{|\Grad \theta^n|^2}{(\theta^n)^2}\dx \dt \le \frac{C}{\delta(1-\delta)}.
\end{split}
\end{equation}
Finally, using \eqref{evale}, \eqref{E1n}, \eqref{entropy5}, \eqref{entropy7}, \eqref{boundapp} and \eqref{e1epsest} and the \eqref{ule0}, we have
\begin{equation}\label{entropy8}
\begin{split}
&\int_0^T \int_{\Omega} \frac{|\Grad e^n|^2}{(1+e^n)^{1+\delta}} \dx \dt\le C\int_0^T \int_{\Omega} \frac{(1+|(\tilde{e}_1^{\eps})'(\theta^n)|^2)|\Grad \theta^n|^2}{(1+e^n)^{1+\delta}} + |\Grad b^n|^2\dx \dt \\
&\quad \le C \int_0^T \int_{\Omega} \frac{|\Grad \theta^n|^2}{(1+e^n)^{1+\delta}} + \frac{|\Grad \theta^n|^2}{(\theta^n)^2} + |\Grad b^n|^2\dx \dt \le \frac{C}{\delta(1-\delta)}.
\end{split}
\end{equation}

Next, we shall derive  uniform bounds on $\theta^n$ and $e^n$ in Lebesgue and Sobolev spaces. To this end we recall the following interpolation inequality in three space-dimensions:
\begin{equation}\label{interp1}
\int_{0}^T \|u\|_{\frac{10}{3}}^{\frac{10}{3}} \le C \sup_{t\in (0,T)} \|u(t)\|^{\frac{4}{3}}_2 \int_0^T \|u\|_{1,2}^2 \dt.
\end{equation}
Then, for $\delta\in (0,1)$ we define $u:=(1+\theta^n)^{\frac{1-\delta}{2}}$ and appeal to \eqref{E3nb} and \eqref{entropy7} to deduce that
$$
\sup_{t\in (0,T)} \|u(t)\|_2 + \int_{0}^T \|u\|_{1,2}^2 \dt \le C(\delta).
$$
Consequently, using \eqref{interp1}, we have
\begin{equation}\label{blaob}
\int_{0}^T \|u\|_{\frac{10}{3}}^{\frac{10}{3}} \le C(\delta), \quad \mbox{and therefore} \quad  \int_0^T \int_{\Omega} (1+\theta^n)^{\frac{5(1-\delta)}{3}}\dx \dt\le C(\delta),
\end{equation}
and since $\delta \in (0,1)$ was arbitrary we deduce that
\begin{equation}\label{E4theta}
\int_0^T \|\theta^n\|_p^p \dt\le C(p) \quad \textrm{ for arbitrary } p\in [1,5/3).
\end{equation}
In the same way, we also have that
\begin{equation}\label{E4e}
\int_0^T \|e^n\|_p^p \dt\le C(p) \quad \textrm{ for arbitrary } p\in [1,5/3).
\end{equation}
Next, using H\"{o}lder's inequality and the bounds \eqref{entropy7} and \eqref{blaob}, we deduce that
$$
\begin{aligned}
\int_0^T\int_{\Omega} |\Grad \theta^n|^{\frac{5(1-\delta)}{4-\delta}} \dx \dt &= \int_0^T\int_{\Omega} \left(\frac{|\Grad \theta^n|^2}{(1+\theta^n)^{1+\delta}}\right)^{\frac{5(1-\delta)}{2(4-\delta)}}(1+\theta^n)^{\frac{5(1-\delta)(1+\delta)}{2(4-\delta)}}\dx \dt\\
&\le \left(\int_0^T\int_{\Omega} \frac{|\Grad \theta^n|^2}{(1+\theta^n)^{1+\delta}}\dx \dt \right)^{\frac{5(1-\delta)}{2(4-\delta)}}\left(\int_0^T \int_{\Omega} (1+\theta^n)^{\frac{5(1-\delta)}{3}}\dx \dt\right)^{\frac{3(1+\delta)}{2(4-\delta)}}\\
&\le C(\delta).
\end{aligned}
$$
Thus, since $\delta \in (0,1)$ was arbitrary,  we obtain that
\begin{equation}\label{E4thetagrad}
\int_0^T \|\Grad \theta^n\|_p^p \dt\le C(p) \quad \textrm{ for arbitrary } p\in [1,5/4).
\end{equation}
In the same way we deduce also that
\begin{equation}\label{E4egrad}
\int_0^T \|\Grad e^n\|_p^p \dt\le C(p) \quad \textrm{ for arbitrary } p\in [1,5/4).
\end{equation}

We end this part of the proof by stating a uniform bound on $\partial_t e^n$. Using H\"{o}lder's inequality, we appeal to \eqref{E4theta} and \eqref{E2n} to find that\footnote{Note that the sequence $\{e^n\}$ is bounded in $L^{5/3-}(Q)$ and $\{\mathbf{v}^n\}$ is bounded in $L^{10/3}(Q)$. Hence $\{e^n \mathbf{v}^n\}$ is bounded in $L^{10/9-}(Q)$. Here, for $p>1$, by \textit{bounded in $L^{p-}(Q)$} we mean \textit{bounded in $L^r(Q)$ for all $r \in [1,p)$}.}
$$
\int_0^T \|e^n \vv^n\|_{p}\dt \le C(p) \quad \textrm{ for arbitrary } p\in [1,10/9).
$$
Thus, using also \eqref{E4thetagrad}, \eqref{E1n}, \eqref{inftybn} and the boundedness of $\kappa$, $\alpha$, $\tilde{e}_1^{\eps}$ and $\tilde{\psi}_2'$, and recalling also the embedding $W^{1,10}(\Omega)\hookrightarrow L^{\infty}(\Omega)$,  we deduce from \eqref{Galerkin:emu} that
\begin{equation}\label{Etnn}
\int_0^T \|\partial_t e^n\|_{(W^{1,10+\delta}(\Omega))^*}\dt \le C(\delta)\quad \textrm{ for arbitrary }\delta>0.
\end{equation}

\subsubsection*{Limits  based on the a~priori estimates}

Thanks to the uniform $n$-independent bounds \eqref{E1n}, \eqref{inftybn}, \eqref{E1tn}, \eqref{E4theta}, \eqref{E4e}, \eqref{E4thetagrad}, \eqref{E4egrad} and \eqref{Etnn}, the definition \eqref{stickeps} and the trace theorem, there exist subsequences that we do not relabel such that, for all $p\in [1,5/4)$,  all $q\in [1,5/3)$ and all $z\in (10,\infty)$,
\begin{equation}\label{weakn}
\begin{aligned}
\vv^n &\rightharpoonup \vv &&\textrm{weakly in } L^{2}(0,T; W^{1,2}_{\bn, \diver}) \cap W^{1,2}(0,T; (W^{1,2}_{\bn, \diver})^*),\\
b^n &\rightharpoonup b &&\textrm{weakly in } L^{2}(0,T; W^{1,2}(\Omega)) \cap W^{1,2}(0,T; (W^{1,2}(\Omega))^*),\\
b^n &\rightharpoonup^* b &&\textrm{weakly$^*$ in } L^{\infty}(0,T; L^{\infty}(\Omega)),\\
\theta^n &\rightharpoonup \theta &&\textrm{weakly in } L^{p}(0,T; W^{1,p}(\Omega)) \cap L^q(0,T; L^q(\Omega)),\\
e^n &\rightharpoonup e &&\textrm{weakly in } L^{p}(0,T; W^{1,p}(\Omega)) \cap L^q(0,T; L^q(\Omega)),\\
\partial_t e^n &\rightharpoonup^* \partial_t e &&\textrm{weakly$^*$ in } \mathcal{M}(0,T; (W^{1,z}(\Omega))^*),\\
\tilde{\bs}^{\eps}(\vv^n_{\tau})&\rightharpoonup \bs &&\textrm{weakly in } L^{2}(0,T; L^2(\partial \Omega)^3),
\end{aligned}
\end{equation}
with $\eps = \eps^n :=\frac{1}{n}$. Consequently, using the generalised version of the Aubin--Lions lemma (cf. Corollary 7.9 in \cite{roubcek.t:nonlinear*1}) and the trace theorem, we deduce that
\begin{equation}\label{strongn}
\begin{aligned}
\vv^n &\to\vv &&\textrm{strongly in } L^{2}(0,T; L^2(\Omega)^3),\\
\vv^n &\to\vv &&\textrm{strongly in } L^{2}(0,T;L^{2}(\partial \Omega)^3),\\
b^n &\to b &&\textrm{strongly in } L^{2}(0,T; L^{2}(\Omega)),\\
e^n &\to e &&\textrm{strongly in } L^{1}(0,T; L^{1}(\Omega)),\\
b^n &\to b &&\textrm{a.e. in } Q,\\
e^n &\to e &&\textrm{a.e. in } Q.
\end{aligned}
\end{equation}
Hence, it follows from \eqref{E3n}, the above strong convergence result and Fatou's lemma that
\begin{equation}\label{linftylone}
\|e(t)\|_1 \le C \quad \textrm{for almost all } t\in (0,T).
\end{equation}
At this point it is worth noting that since we do not control the time derivative of $\theta^n$ we cannot so simply deduce the strong convergence of the temperature. Nevertheless, we show that it is a consequence of the strong convergence of $b^n$ and $e^n$. Indeed, using \eqref{strongn} and Egorov's theorem, we know that for any $\delta>0$ there exists a set $Q_{\delta} \subset Q=(0,T)\times \Omega$ such that $e^n \to e$ and $b^n\to b$ uniformly on $Q_{\delta}$ and such that $|Q\setminus Q_{\delta}|\le \delta$. In addition, since $|\theta^n|\le C(1+e^n)$, we also have that
\begin{equation}\label{wsth}
\theta^n \rightharpoonup^* \theta \quad \textrm{ weakly$^*$ in }L^{\infty}(Q_{\delta}).
\end{equation}
We use the assumption \textbf{(A4)}, the definition of $\tilde{e}^{\eps}$, the fact that $\tilde{e}_1^{\eps}$ is nondecreasing, the weak$^*$ convergence \eqref{wsth}, the uniform convergence of $e^n$ and $b^n$ in $Q_{\delta}$ and the estimate \eqref{entropy5} to deduce that (recall that $\eps=\eps^n:=1/n$ here)
$$
\begin{aligned}
\lim_{n\to \infty} &C_1\int_{Q_\delta}|\theta^n - \theta|^2\dx \dt  \le \lim_{n\to \infty}\int_{Q_{\delta}}(\tilde{e}_0(\theta^n) - \tilde{e}_0(\theta))(\theta^n -\theta)\dx \dt=\lim_{n\to \infty}\int_{Q_{\delta}} \tilde{e}_0(\theta^n) (\theta^n -\theta)\dx \dt\\
&=\lim_{n\to \infty}\int_{Q_{\delta}} (\tilde{e}_0(\theta^n)+\tilde{e}_1^{\eps}(\theta^n)\tilde{\psi}_2(b^n)-\tilde{e}_1^{\eps}(\theta^n)\tilde{\psi}_2(b^n)) (\theta^n -\theta)\dx \dt\\
&=\lim_{n\to \infty}\int_{Q_{\delta}} (e^n-\tilde{e}_1^{\eps}(\theta^n)\tilde{\psi}_2(b^n)) (\theta^n -\theta)\dx \dt\\
&=-\lim_{n\to \infty}\int_{Q_{\delta}}\tilde{e}_1^{\eps}(\theta^n)(\tilde{\psi}_2(b^n)-\tilde{\psi}_2(b) + \tilde{\psi}_2(b)) (\theta^n -\theta)\dx \dt\\
&=-\lim_{n\to \infty}\int_{Q_{\delta}}(\tilde{e}_1^{\eps}(\theta^n)-\tilde{e}_1^{\eps}(\theta)+\tilde{e}_1^{\eps}(\theta))\tilde{\psi}_2(b) (\theta^n -\theta)\dx \dt\\
&\le -\lim_{n\to \infty}\int_{Q_{\delta}} \tilde{e}_1^{\eps}(\theta) \tilde{\psi}_2(b) (\theta^n -\theta)\dx \dt.
\end{aligned}
$$
Next, we want to replace $\tilde{e}_1^{\eps}$ by $\tilde{e}_1$. To do so, we fix $\gamma\in (0,1)$ and decompose the integral into the sum of two integrals: the first integral corresponding to the set where $\theta^n < \gamma$, and the second integral corresponding to the complement with respect to $Q_\delta$ of that set, where $\theta^n \geq  \gamma$. We note that on the latter set the inequality $\theta \geq \gamma$ also holds, almost everywhere.\footnote{Indeed, letting $A_\gamma:=\{(t,x) \in Q_\delta \colon \theta^n \geq \gamma\}$, we deduce that, for any nonnegative test function $\varphi \in L^1(Q_\delta)$,  $\int_{Q_\delta} (\theta - \gamma)\chi_{A_\gamma} \varphi \dd x \dd t \geq \int_{Q_\delta} (\theta - \theta_n)\chi_{A_\gamma} \varphi \dd x \dd t \rightarrow 0$ as $n \rightarrow \infty$, thanks to \eqref{wsth} and because $\chi_{A_\gamma} \varphi \in L^1(Q_\delta)$. Therefore, $\int_{Q_\delta} (\theta - \gamma)\chi_{A_\gamma} \varphi \dd x \dd t \geq 0$ for all nonnegative $\varphi \in L^1(Q_\delta)$, whereby $(\theta - \gamma)\chi_{A_\gamma} \geq 0$ a.e. on $Q_\delta$; in other words, $\theta \geq \gamma$ a.e. on $A_\gamma$.}
The motivation for this decomposition of the integral is that  $\tilde{e}_1^{\eps}(\theta) = \tilde{e}_1(\theta)$ on the subset of $Q_\delta$ where $\eps < \gamma \leq \theta^n$, which follows from the fact that $\tilde{\psi}_1^{\eps}(s)=\tilde{\psi}_1(s)$ for all $s \in (\eps,\infty)$. For large enough $n$, $\eps = \eps^n := \frac{1}{n} < \gamma$; thus, we can assume without loss of generality that $0 < \eps < \gamma$; hence,
$$
\begin{aligned}
\lim_{n\to \infty} &C_1\int_{Q_\delta}|\theta^n - \theta|^2\dx \dt\le
-\lim_{n\to \infty}\int_{Q_{\delta}}\tilde{e}_1^{\eps}(\theta)\tilde{\psi}_2(b) (\theta^n -\theta)\chi_{\theta^n < \gamma}\dx \dt
\\
&\qquad
-\lim_{n\to \infty}\int_{Q_{\delta}}\tilde{e}_1^{\eps}(\theta)\tilde{\psi}_2(b) (\theta^n -\theta) \chi_{\theta^n  \geq \gamma}\dx \dt\\
&\le C(\delta) \lim_{n\to \infty}|\{\theta^n <\gamma\}|-\lim_{n\to \infty}\int_{Q_{\delta}} \tilde{e}_1^{\eps}(\theta)\tilde{\psi}_2(b) (\theta^n -\theta) \chi_{\theta^n\geq\gamma}\dx \dt\\
&= C(\delta) \lim_{n\to \infty}|\{\theta^n < \gamma\}|-\lim_{n\to \infty}\int_{Q_{\delta}}\tilde{e}_1(\theta)\tilde{\psi}_2(b) (\theta^n -\theta) \chi_{\theta^n\geq\gamma}\dx \dt\\
&= C(\delta) \lim_{n\to \infty}|\{\theta^n <\gamma\}|-\lim_{n\to \infty}\int_{Q_{\delta}}\tilde{e}_1(\theta)\tilde{\psi}_2(b) (\theta^n -\theta) \dx \dt\\
&\qquad +\lim_{n\to \infty}\int_{Q_{\delta}}\tilde{e}_1(\theta)\tilde{\psi}_2(b) (\theta^n -\theta) \chi_{\theta^n < \gamma}\dx \dt\\
&\le C(\delta)\lim_{n\to \infty}|\{\theta^n < \gamma\}|,
\end{aligned}
$$
where we have used the weak$^*$ convergence of $\theta^n$ asserted in \eqref{wsth}, in tandem with \eqref{ule1} and \eqref{inftyb}, which together imply that $\tilde{e}_1 \tilde{\psi}_2 \in L^1(Q_\delta)$. Finally, using the bound \eqref{entropy5} for $\ln \theta^n$,  upon recalling that $\gamma \in (0,1)$ we
find that
$$
\begin{aligned}
\lim_{n\to \infty} &C_1\int_{Q_\delta}|\theta^n - \theta|^2\dx \dt\le C(\delta)\lim_{n\to \infty} |\{\theta^n < \gamma\}| = C(\delta)\lim_{n\to \infty} |\{|\ln \theta^n| > |\ln \gamma|\}| \\
&\le C(\delta)\lim_{n\to \infty} \int_{Q_\delta}\frac{|\ln \theta^n|}{|\ln \gamma|}\dx \dt \le C(\delta)\lim_{n\to \infty} \int_{Q}\frac{|\ln \theta^n|}{|\ln \gamma|}\dx \dt\le \frac{C(\delta)}{|\ln \gamma|}.
\end{aligned}
$$
Finally, since the left-hand side does not depend on $\gamma$, we can let $\gamma\to 0_+$ to obtain
$$
\theta^n \to \theta \quad\textrm{ strongly in } L^2(Q_\delta).
$$
However, using the fact that $|Q\setminus Q_{\delta}|\le \delta$ for all $\delta>0$, we then deduce by a diagonal argument the almost everywhere convergence of a subsequence of $\theta^n$ (not indicated) to $\theta$ on the whole of $Q$, and hence, by the weak convergence of $\theta^n$ to $\theta$ in $L^q(Q)$ for $q \in [1,\frac{5}{3})$ (see \eqref{weakn}$_4$) and Vitali's theorem we obtain that
\begin{equation}\label{thetanstrong}
\begin{aligned}
\theta^n &\to \theta \quad &&\textrm{ strongly in } L^1(0,T; L^1(\Omega)),\\
\theta^n &\to \theta \quad &&\textrm{ a.e. in } Q.
\end{aligned}
\end{equation}
Similarly, with the help of \eqref{thetanstrong} and thanks to the control on $\ln \theta^n$, we see (recall, again, our choice $\eps=\eps^n:=1/n$) that
\begin{equation}\label{pointe}
\begin{aligned}
&\lim_{n\to \infty}\int_{Q}|\tilde{e}_1^{\eps}(\theta^n)-\tilde{e}_1(\theta)|\dx \dt\\
&\le \lim_{n\to \infty}\int_{Q}|\tilde{e}_1^{\eps}(\theta^n)-\tilde{e}_1(\theta^n)|\dx \dt +\lim_{n\to \infty}\int_{Q}|\tilde{e}_1(\theta^n)-\tilde{e}_1(\theta)|\dx \dt\\
&=\lim_{n\to \infty}\int_{Q}|\tilde{e}_1^{\eps}(\theta^n)-\tilde{e}_1(\theta^n)|\chi_{\theta^n < \eps}\dx \dt \le \lim_{n\to \infty}C|\{\theta^n < \eps\}|\le \lim_{n\to \infty}\frac{C}{|\ln n|} =0,
\end{aligned}
\end{equation}
where we have also used that $\tilde{e}_1^{\eps}(\theta^n) = \tilde{e}_1(\theta^n)$ whenever $\eps\leq \gamma < \theta^n$, which follows from the fact that $\tilde{\psi}_1^{\eps}(s)=\tilde{\psi}_1(s)$ for all $s \in (\eps,\infty)$.

Consequently, since also $b^n\to b$ almost everywhere, we have
$$
e^n= \tilde{e}^{\eps}(\theta^n,b^n) \to \tilde{e}(\theta,b) \quad \textrm{ almost everywhere on } Q,
$$
and therefore we have the identification
\begin{equation}
e=\tilde{e}(\theta,b)\qquad \textrm{almost everywhere on } Q.\label{thetavse}
\end{equation}

Next, we shall identify $\bs$, i.e., the boundary term. For this purpose, we recall from \eqref{strongn}$_2$ the strong convergence of $\vv^n$ to $\vv$ in $L^2((0,T)\times \partial \Omega)^3 = L^2(0,T;L^2(\partial\Omega)^3)$, with $\vv^n \cdot \vn =0$, $\vv \cdot \vn = 0$ on $(0,T) \times \partial \Omega$
 whereby $\vv^n_\tau = \vv^n$ and $\vv_\tau = \vv$ on $(0,T) \times \partial \Omega$; we then deduce using the definition \eqref{stickeps}, for any test function $\vw$ contained in $L^2(0,T; L^2(\partial \Omega)^3)$, that
\begin{equation}\label{minty}
\begin{aligned}
0&\le \lim_{n\to \infty}\int_0^T\int_{\partial \Omega}\left(\tilde{\bs}^{\eps}(\vv^n)-\vw\right)\cdot \left(\frac{(|\tilde{\bs}^{\eps}(\vv^n)|-s_*)_+}{|\tilde{\bs}^{\eps}(\vv^n)|}\tilde{\bs}^{\eps}(\vv^n)
-\frac{(|\vw|-s_*)_+}{|\vw|}\vw \right)\dS \dt\\
&=\lim_{n\to \infty}\int_0^T\int_{\partial \Omega}\left(\tilde{\bs}^{\eps}(\vv^n)-\vw\right)\cdot \left(\gamma_* \vv^n
-\frac{(|\vw|-s_*)_+}{|\vw|}\vw \right)\dS \dt\\
&\qquad +\lim_{n\to \infty}\int_0^T\int_{\partial \Omega}\left(\tilde{\bs}^{\eps}(\vv^n)-\vw\right)\cdot \left(\frac{(|\tilde{\bs}^{\eps}(\vv^n)|-s_*)_+}{|\tilde{\bs}^{\eps}(\vv^n)|}\tilde{\bs}^{\eps}(\vv^n)
-\gamma_* \vv^n\right)\dS \dt\\
&=\int_0^T\int_{\partial \Omega}\left(\bs-\vw\right)\cdot \left(\gamma_* \vv
-\frac{(|\vw|-s_*)_+}{|\vw|}\vw \right)\dS \dt\\
&\qquad +\lim_{n\to \infty}\int_0^T\int_{\partial \Omega}\left(\tilde{\bs}^{\eps}(\vv^n)-\vw\right)\cdot \left(\frac{(|\tilde{\bs}^{\eps}(\vv^n)|-s_*)_+}{|\tilde{\bs}^{\eps}(\vv^n)|}\tilde{\bs}^{\eps}(\vv^n)
-\gamma_* \vv^n\right)\dS \dt.
\end{aligned}
\end{equation}
We shall now estimate the last integral. Using \eqref{stickeps} and noting that $\tilde{\bs}^{\eps}(\vv^n)/|\tilde{\bs}^{\eps}(\vv^n)| = \vv^n/|\vv^n|$, we have
$$
\begin{aligned}
\left|\frac{(|\tilde{\bs}^{\eps}(\vv^n)|-s_*)_+}{|\tilde{\bs}^{\eps}(\vv^n)|}\tilde{\bs}^{\eps}(\vv^n)
-\gamma_* \vv^n\right|^2= |\vv^n|^2 \left(\gamma_* -  \frac{(|\tilde{\bs}^{\eps}(\vv^n)|-s_*)_+}{|\vv^n|}
\right)^2\\
=\left[\gamma_*|\vv^n|-\left(\gamma_* |\vv^n|  -\frac{\eps s^* }{\eps + |\vv^n|}\right)_+\right]^2.
\end{aligned}
$$
Hence, for an arbitrary $\delta>0$, we deduce the bound (recall that $\eps = \eps^n:= 1/n$)
$$
\begin{aligned}
\lim_{n\to \infty}& \int_0^T \int_{\partial \Omega}\left|\frac{(|\tilde{\bs}^{\eps}(\vv^n)|-s_*)_+}{|\tilde{\bs}^{\eps}(\vv^n)|}\tilde{\bs}^{\eps}(\vv^n)-\gamma_* \vv^n\right|^2\dS \dt \\
&=\lim_{n\to \infty}\int_0^T \int_{\partial \Omega}\left[\gamma_*|\vv^n|-\left(\gamma_* |\vv^n|  -\frac{\eps s^* }{\eps + |\vv^n|}\right)_+\right]^2\chi_{|\vv^n|>\delta}\dS \dt\\
&\qquad +\lim_{n\to \infty}\int_0^T \int_{\partial \Omega}\left[\gamma_*|\vv^n|-\left(\gamma_* |\vv^n|  -\frac{\eps s^* }{\eps + |\vv^n|}\right)_+\right]^2\chi_{|\vv^n|\le \delta}\dS \dt\\
&\le\lim_{n\to \infty}\int_0^T \int_{\partial \Omega}\left[\frac{s^* }{1 + n\delta}\right]^2\dS \dt +\lim_{n\to \infty}\int_0^T \int_{\partial \Omega}\left[\gamma_*\delta\right]^2\dS \dt\\
&\le C\delta^2.
\end{aligned}
$$
Since $\gamma>0$ was arbitrary, we can let $\gamma\to 0_+$ and find that
$$
\begin{aligned}
\lim_{n\to \infty}& \int_0^T \int_{\partial \Omega}\left|\frac{(|\tilde{\bs}^{\eps}(\vv^n)|-s_*)_+}{|\tilde{\bs}^{\eps}(\vv^n)|}\tilde{\bs}^{\eps}(\vv^n)-\gamma_* \vv^n\right|^2\dS \dt =0.
\end{aligned}
$$
Using this relation in \eqref{minty}, we get
\begin{equation}\label{minty2}
\begin{aligned}
0&\le \int_0^T\int_{\partial \Omega}\left(\bs-\vw\right)\cdot \left(\gamma_* \vv
-\frac{(|\vw|-s_*)_+}{|\vw|}\vw \right) \dS \dt.
\end{aligned}
\end{equation}
Finally, since $\vw \in L^2(0,T; L^2(\partial \Omega)^3)$ is arbitrary, we can use Minty's method to deduce that
\begin{equation}\label{bs}
\gamma_* \vv = \frac{(|\bs|-s_*)_+}{|\bs|}\bs\quad \textrm{ almost everywhere on } (0,T)\times \partial \Omega.
\end{equation}
Indeed, by taking $\vw = \bs - \beta \boldsymbol{z}$ in \eqref{minty2} with $\beta>0$ for any $\boldsymbol{z} \in L^2(0,T; L^2(\partial \Omega)^3)$, dividing by $\beta$, passing to the limit $\beta \rightarrow 0_+$, and substituting $\boldsymbol{z}$ by $-\boldsymbol{z}$ in order to replace the inequality sign in \eqref{minty2}
with equality, the assertion \eqref{bs} follows.

We can now let $n\to \infty$ in \eqref{Galerkin:vvmu} and  \eqref{Galerkin:bmu} to deduce that
\begin{equation}
\begin{split}\label{Galerkin:vveps}
\langle \partial_t\vv, \vw\rangle + \int_{\Omega}(2\nu(\theta,b)\mathbb{D}(\vv)-G_k(|\vv|)\vv\otimes \vv ):\Grad \vw\dx+ \int_{\partial \Omega} \bs\cdot \vw \dS=\langle \bfb, \vw\rangle
\end{split}
\end{equation}
for all $\vw\in W^{1,2}_{\bn, \diver}$ and almost all $t\in (0,T)$;
\begin{equation}\label{Galerkin:beps}
\begin{split}
 \langle \partial_t b, u \rangle  + \int_{\Omega}\left(\alpha(\theta,b) \Grad b - b\vv \right)\cdot \Grad u +h(\theta,b)u \dx=0
\end{split}
\end{equation}
for all $u\in W^{1,2}(\Omega)$ and almost all $t\in (0,T)$ and for the initial conditions we have that
\begin{equation}\label{Galerkininiteps}
\begin{aligned}
\vv (0) &=\vv_0 \qquad \textrm{and} \qquad
b (0) &= b_0.
\end{aligned}
\end{equation}
It remains to pass to the limit in \eqref{Galerkin:emu}, but here we are facing a difficulty with the term on the right-hand side, which is a~priori only integrable and therefore we cannot identify the limit so easily; in addition we also need to recover the initial condition for $e$, and consequently for $\theta$.

\subsubsection*{Limit passage: internal energy equation}
To pass to the limit in the internal energy equation, we first show the strong convergence of the velocity gradient. First of all, multiplying
\eqref{Galerkin:vvmu} by $c_i^n$ and summing the result over $i=1,\ldots, n$ we get after integration over $(0,T)$ the following equality:
\begin{equation}
\int_0^T\int_{\Omega} 2\nu(\theta^n,b^n) |\mathbb{D}(\vv^n)|^2 \dx \dt = \int_0^T \langle \bfb, \vv^n \rangle \dt + \frac12 (\|\vv^n_0\|_2^2 - \|\vv^n(T)\|_2^2).\label{dod1}
\end{equation}
In the same way, we set $\vw:=\vv$ in \eqref{Galerkin:vveps} to get
\begin{equation}
\int_0^T\int_{\Omega} 2\nu(\theta,b) |\mathbb{D}(\vv)|^2 \dx \dt = \int_0^T \langle \bfb, \vv \rangle \dt + \frac12 (\|\vv_0\|_2^2 - \|\vv(T)\|_2^2).\label{dod2}
\end{equation}
Hence, using the weak lower semicontinuity and the weak convergence of $\vv^n$, we get from \eqref{dod1} and \eqref{dod2}
\begin{equation}\label{Minty5}
\limsup_{n\to \infty}\int_0^T\int_{\Omega} \nu(\theta^n,b^n) |\mathbb{D}(\vv^n)|^2 \dx \dt \le \int_0^T\int_{\Omega} \nu(\theta,b) |\mathbb{D}(\vv)|^2 \dx \dt.
\end{equation}
Recalling the assumption on the viscosity coefficient stated in \textbf{(A5)} gives
$$
\begin{aligned}
\lim_{n\to \infty}&C_1 \int_0^T \int_{\Omega} |\mathbb{D}(\vv^n)-\mathbb{D}(\vv)|^2\dx \dt\le \lim_{n\to \infty} \int_0^T \int_{\Omega}\nu(\theta^n,b^n) |\mathbb{D}(\vv^n)-\mathbb{D}(\vv)|^2\dx\dt\\
&\hspace{2.5mm}\leq \limsup_{n\to \infty} \int_0^T \int_{\Omega}\nu(\theta^n,b^n) (|\mathbb{D}(\vv^n)|^2+ |\mathbb{D}(\vv)|^2-2\mathbb{D}(\vv^n):\mathbb{D}(\vv))\dx\dt\\
&\overset{\eqref{Minty5}}\le \int_0^T \int_{\Omega}\nu(\theta,b) (|\mathbb{D}(\vv)|^2+ |\mathbb{D}(\vv)|^2-2\mathbb{D}(\vv):\mathbb{D}(\vv))\dx\dt=0.
\end{aligned}
$$
In the transition from the last term in the second line to the last term in the third line we used that $\mathbb{D}(\vv^n) \rightharpoonup \mathbb{D}(\vv)$ weakly in $L^2(0,T;L^2(\Omega)^{3 \times 3})$ by \eqref{weakn}$_1$, whereby  $\mathbb{D}(\vv^n):\mathbb{D}(\vv) \rightharpoonup \mathbb{D}(\vv) : \mathbb{D}(\vv)$ weakly in $L^1(0,T;L^1(\Omega)^{3 \times 3})$; that $\theta^n \rightarrow \theta$ strongly in $L^1(0,T;L^1(\Omega))$ by \eqref{thetanstrong} and therefore for a subsequence, not indicated, $\theta^n \rightarrow \theta$ a.e. on $Q$; that $b^n \rightarrow b$ strongly in $L^2(0,T;L^2(\Omega))$ by \eqref{strongn}$_3$ and therefore for a subsequence, not indicated, $b^n \rightarrow b$ a.e. on $Q$, which by the assumed continuity and boundedness of the viscosity coefficient stated in \textbf{(A5)} implies that $\nu(\theta^n,b^n) \rightarrow \nu(\theta,b)$ a.e. on $Q$ and $C_1 \leq \nu(\theta^n,b^n) \leq C_2$.
(Here we have made use of Proposition 2.61 on p.183 in \cite{fonseca.i.leoni.g:modern.methods}, concerning the weak converge in $L^1(Q)$ of a sequence, which is the product of a weakly convergent sequence in $L^1(Q)$ and a uniformly bounded sequence which converges almost everywhere on $Q$.)
Therefore, we also have
\begin{equation}\label{strong}
\begin{aligned}
\vv^n &\to \vv &&\textrm{strongly in } L^2(0,T; W^{1,2}_{\bn, \diver}),\\
\nu(\theta^n,b^n)|\mathbb{D}(\vv^n)|^2 & \to \nu(\theta,b)|\mathbb{D}(\vv)|^2 &&\textrm{strongly in } L^1(0,T; L^1(\Omega)).
\end{aligned}
\end{equation}
Here \eqref{strong}$_2$ follows from \eqref{strong}$_1$ by writing
$$\nu(\theta^n,b^n)|\mathbb{D}(\vv^n)|^2 - \nu(\theta,b)|\mathbb{D}(\vv)|^2 = \nu(\theta^n,b^n)(|\mathbb{D}(\vv^n)|^2 - |\mathbb{D}(\vv)|^2) +  (\nu(\theta^n,b^n)|\mathbb{D}(\vv)|^2  - \nu(\theta,b)|\mathbb{D}(\vv)|^2),$$
noting that the first summand on the right-hand side converges to $0$ thanks to the boundedness of $\nu$ implied by the assumption \textbf{(A5)} and the strong convergence of $|\mathbb{D}(\vv^n)|^2$ to $|\mathbb{D}(\vv)|^2$ in $L^1(Q)$ implied by \eqref{strong}$_1$; and the second summand converges to $0$  by Lebesgue's dominated convergence theorem thanks to the a.e. convergence of $\nu(\theta^n,b^n)|\mathbb{D}(\vv)|^2$  to  $\nu(\theta,b)|\mathbb{D}(\vv)|^2$ in $Q$ and the boundedness of $\nu$.

Note that in a very similar way, we can also deduce that
\begin{equation}\label{strongbn}
\begin{aligned}
b^n &\to b &&\textrm{strongly in } L^2(0,T; W^{1,2}(\Omega)).
\end{aligned}
\end{equation}
In addition, having shown the strong convergence of $\mathbb{D}(\vv^n)$, we deduce from \eqref{Galerkin:emu} that, for any $\delta>0$,
$$
\|\partial_t e^n(t)\|_{(W^{1,10+\delta}(\Omega))^*}\le g^n(t),
$$
where the sequence $(g^n)_{n \geq 1}$ is weakly convergent in $L^1(0,T)$ and thus bounded and equi-integrable in $L^1(0,T)$. Hence the sequence $(\partial_t e^n)_{n \geq 1} \subset
L^1(0,T; (W^{1,10+\delta}(\Omega))^*)$ is bounded and equi-integrable; also, as $W^{1,10+\delta}(\Omega)$ is reflexive the same is true of $(W^{1,10+\delta}(\Omega))^*$. Thus, $\partial_t e^n$ is weakly compact in $L^1(0,T; (W^{1,10+\delta}(\Omega))^*)$ (cf. \cite{Di75}).
Consequently, we have for a subsequence that
\begin{equation}\label{et:n}
\begin{aligned}
\partial_t e^n &\rightharpoonup \partial e  &&\textrm{weakly in } L^1(0,T; (W^{1,10+\delta}(\Omega))^*), \qquad \delta>0.
\end{aligned}
\end{equation}
Using this we can let $n\to \infty$ also in \eqref{Galerkin:emu} and deduce that
\begin{equation}\label{Galerkin:en}
\begin{split}
& \langle \partial_t e, u\rangle  + \int_{\Omega}\left( \kappa(\theta,b)\Grad \theta
-e\vv\right) \cdot \Grad u\dx +\int_{\Omega} \alpha(\theta,b) \tilde{e}(\theta) \psi_2'(b)\Grad b  \cdot \Grad u \dx\\
&\quad =\int_{\Omega} 2\nu(\theta,b)\,|\mathbb{D}(\vv)|^2 u\dx
\end{split}
\end{equation}
for all $u\in W^{1,10+\delta}(\Omega)$ and almost all $t\in (0,T)$. Concerning the initial condition, it is straightforward to verify that
\begin{equation}\label{Galerkininitn2}
\begin{aligned}
e(0) = \tilde{e}(\theta_0,b_0) \textrm{ in } (W^{1,10+\delta}(\Omega))^*.
\end{aligned}
\end{equation}

\subsubsection*{Initial condition the for energy and the temperature}
We end this section by strengthening \eqref{Galerkininitn2}: we will show that
\begin{equation}\label{energyinit}
\lim_{t\to 0_+} \|e(t)-\tilde{e}(\theta_0,b_0) \|_1=0.
\end{equation}
We provide the detailed proof of \eqref{energyinit} since it will be also needed in the next section. First of all, it follows from \eqref{Galerkin:en} with $u \equiv 1$, integrating in time from $0$ to $t \in (0,T)$, the equality \eqref{Galerkininitn2}, and letting $t \rightarrow 0_+$ that
\begin{equation}\label{energyupper}
\lim_{t\to 0_+} \int_{\Omega} e(t) \dx = \int_{\Omega} \tilde{e}(\theta_0,b_0) \dx.
\end{equation}
Next, we set $u:=(1+e^n)^{-\frac12}\varphi$ in \eqref{Galerkin:emu} where $\varphi \in W^{1,\infty}(\Omega)$, noting that $u$ is an admissible choice of test function since $u\in W^{1,2}(\Omega)$; hence,
\begin{equation}\label{e:renorm:n}
\begin{split}
& 2\langle \partial_t \sqrt{1+e^n}, \varphi\rangle  + \int_{\Omega}\left( \frac{\kappa(\theta^n,b^n)}{\sqrt{1+e^n}}\Grad \theta^n
-2\sqrt{1+e^n}\vv^n\right) \cdot \Grad \varphi\dx \\
&\quad +\int_{\Omega} \frac{\alpha(\theta^n,b^n) \tilde{e}^{\eps}_1(\theta^n) \psi_2'(b^n)}{\sqrt{1+e^n}} \Grad b^n  \cdot \Grad \varphi \dx\\
&\quad -\frac12\int_{\Omega} \frac{\kappa(\theta^n,b^n)\varphi}{(1+e^n)^{\frac32}}\Grad \theta^n
\cdot \Grad e^n\dx -\frac12 \int_{\Omega} \frac{\alpha(\theta^n,b^n) \tilde{e}^{\eps}_1(\theta^n) \psi_2'(b^n)}{(1+e^n)^{\frac32}}\Grad b^n  \cdot \Grad e^n \varphi\dx\\
&=\int_{\Omega} \frac{2\nu(\theta^n,b^n)\,|\mathbb{D}(\vv^n)|^2 \varphi}{\sqrt{1+e^n}}\dx.
\end{split}
\end{equation}
Consequently, using \eqref{entropy7}, \eqref{entropy8} and \eqref{E1n}, we get the uniform bound (in fact, here we use $\delta=1$ from the corresponding estimates and replace $W^{1,10+\delta}$ by $W^{1,11}$)
\begin{equation}\label{renorm}
\int_0^T \|\partial_t \sqrt{1+e^n}\|_{(W^{1,11}(\Omega))^*}\dt \le C.
\end{equation}
Thus, and thanks to \eqref{strongn}$_4$ which implies that $\sqrt{1+e^n}$ converges to $\sqrt{1+e}$ strongly in $L^2(Q)$,
\begin{align}\label{e:measure}
\partial_t \sqrt{1+e^n}\rightharpoonup^* \partial_t \sqrt{1+e} \qquad \textrm{weakly$^*$ in } \mathcal{M}(0,T; (W^{1,11}(\Omega))^*).
\end{align}
Since the time derivative is a measure, we know that for every $\tau \in (0,T)$ there exists a limit from the left and the right, where the limit is taken in the topology of the space $(W^{1,11}(\Omega))^*$, i.e., we can define
\begin{equation}\label{limleft}
\sqrt{1+e(\tau_+)}:=\lim_{t\to \tau_+} \sqrt{1+e(t)} \quad \textrm{ and } \quad \sqrt{1+e^n(\tau_-)}:=\lim_{t\to \tau_-} \sqrt{1+e(t)},
\end{equation}
where both limits are considered in the space $(W^{1,11}(\Omega))^*$. In addition, using \eqref{linftylone}, we see that $\sqrt{1+e} \in L^{\infty}(0,T; L^2(\Omega))$. Therefore, we can use the density of $W^{1,11}(\Omega)$ in $L^2(\Omega)$ and it follows from \eqref{limleft} that
\begin{equation} \label{limleft2}
\begin{aligned}
\sqrt{1+e(t)} &\rightharpoonup \sqrt{1+e(\tau_+)} &&\textrm{weakly in } L^2(\Omega) \textrm{ as } t\to \tau_+,\\
\sqrt{1+e(t)} &\rightharpoonup \sqrt{1+e(\tau_-)} &&\textrm{weakly in } L^2(\Omega) \textrm{ as } t\to \tau_-.
\end{aligned}
\end{equation}
Having defined the quantity $\sqrt{1+e}$ pointwise from the left and the right, we now show that the initial condition is attained with an inequality sign. Using \eqref{explicite} and considering $\varphi \in W^{1,\infty}(\Omega)$, with $\varphi \ge 0$, in  \eqref{e:renorm:n}, we see that
\begin{equation*}
\begin{split}
& 2\langle \partial_t \sqrt{1+e^n}, \varphi\rangle  + \int_{\Omega}\left( \frac{\kappa(\theta^n,b^n)}{\sqrt{1+e^n}}\Grad \theta^n
-2\sqrt{1+e^n}\vv^n\right) \cdot \Grad \varphi\dx \\
&\quad +\int_{\Omega} \frac{\alpha(\theta^n,b^n) \tilde{e}^{\eps}_1(\theta^n) \psi_2'(b^n)}{\sqrt{1+e^n}} \Grad b^n  \cdot \Grad \varphi \dx + C\int_{\Omega}
|\nabla_x b^n|^2 \varphi \dx \ge 0,
\end{split}
\end{equation*}
where $C$ is a positive constant, independent of $n$. After integration over $(0,\tau)$ this leads to
\begin{equation*}
\begin{split}
& 2\int_{\Omega} \sqrt{1+e^n(\tau)} \varphi \dx  + \int_0^{\tau}\int_{\Omega}\left( \frac{\kappa(\theta^n,b^n)}{\sqrt{1+e^n}}\Grad \theta^n
-2\sqrt{1+e^n}\vv^n\right) \cdot \Grad \varphi\dx \dt \\
&\quad +\int_0^{\tau}\int_{\Omega} \frac{\alpha(\theta^n,b^n) \tilde{e}^{\eps}_1(\theta^n) \psi_2'(b^n)}{\sqrt{1+e^n}} \Grad b^n  \cdot \Grad \varphi \dx +C \int_{\Omega} |\nabla_x b^n|^2 \varphi \dx \dt  \\
&\quad \ge 2\int_{\Omega} \sqrt{1+e_0^{\eps}} \varphi \dx.
\end{split}
\end{equation*}
Integration over $\tau\in (t_0, t_0+h)$ then gives
\begin{equation*}
\begin{split}
& 2\int_{t_0}^{t_0+h}\int_{\Omega} \sqrt{1+e^n(\tau)} \varphi \dx \dd \tau + \int_{t_0}^{t_0+h}\int_0^{\tau}\int_{\Omega}\left( \frac{\kappa(\theta^n,b^n)}{\sqrt{1+e^n}}\Grad \theta^n
-2\sqrt{1+e^n}\vv^n\right) \cdot \Grad \varphi\dx \dt \dd \tau\\
&\quad +\int_{t_0}^{t_0+h}\int_0^{\tau}\int_{\Omega} \frac{\alpha(\theta^n,b^n) \tilde{e}^{\eps}_1(\theta^n) \psi_2'(b^n)}{\sqrt{1+e^n}} \Grad b^n  \cdot \Grad \varphi \dx +\int_{\Omega} |\nabla_x b^n|^2 \varphi \dx \dt  \dd \tau\\
&\quad \ge 2\int_{t_0}^{t_0+h}\int_{\Omega} \sqrt{1+e_0^{\eps}} \varphi \dx\dd \tau,
\end{split}
\end{equation*}
and using the already established convergence results we can pass to the limit $n \rightarrow \infty$ with, again, $\eps=\eps^n :=\frac{1}{n}$, to see that
\begin{equation*}
\begin{split}
& 2\int_{t_0}^{t_0+h}\int_{\Omega} \sqrt{1+e(\tau)} \varphi \dx \dd \tau + \int_{t_0}^{t_0+h}\int_0^{\tau}\int_{\Omega}\left( \frac{\kappa(\theta,b)}{\sqrt{1+e}}\Grad \theta
-2\sqrt{1+e}\vv\right) \cdot \Grad \varphi\dx \dt \dd \tau\\
&\quad +\int_{t_0}^{t_0+h}\int_0^{\tau}\int_{\Omega} \frac{\alpha(\theta,b) \tilde{e}_1(\theta) \psi_2'(b)}{\sqrt{1+e}} \Grad b  \cdot \Grad \varphi \dx +\int_{\Omega} |\nabla_x b|^2 \varphi \dx \dt  \dd \tau\\
&\quad \ge 2\int_{t_0}^{t_0+h}\int_{\Omega} \sqrt{1+\tilde{e}(\theta_0,b_0)} \varphi \dx\dd \tau.
\end{split}
\end{equation*}
Finally, dividing by $h$ and letting $h\to 0_+$ and using \eqref{limleft2}, we see that for all $\tau\in (0,T)$ and arbitrary nonnegative $\varphi\in W^{1,11}(\Omega)$ the following inequality holds:
\begin{equation}\label{e:ineq:n}
\begin{split}
& 2\int_{\Omega} \sqrt{1+e(\tau_+)} \varphi \dx  + \int_0^{\tau}\int_{\Omega}\left( \frac{\kappa(\theta,b)}{\sqrt{1+e}}\Grad \theta
-2\sqrt{1+e}\vv\right) \cdot \Grad \varphi\dx \dt \\
&\quad +\int_0^{\tau}\int_{\Omega} \frac{\alpha(\theta,b) \tilde{e}_1(\theta) \psi_2'(b)}{\sqrt{1+e}} \Grad b  \cdot \Grad \varphi \dx +\int_{\Omega} |\nabla_x b|^2 \varphi \dx \dt  \\
&\quad \ge 2\int_{\Omega} \sqrt{1+\tilde{e}(\theta_0,b_0)} \varphi \dx.
\end{split}
\end{equation}
Consequently, letting $\tau \to 0_+$, using also \eqref{limleft2}, we see that, for all nonnegative $\varphi\in L^2(\Omega)$,
\begin{equation}
\lim_{t\to 0_+} \int_\Omega  \sqrt{1+e(t)} \varphi \dx \ge \int_\Omega \sqrt{1+\tilde{e}(\theta_0,b_0)} \varphi \dx.
\end{equation}
Finally, by combining this inequality (with $\varphi:= \sqrt{1+\tilde{e}(\theta_0,b_0)}$~) with \eqref{energyupper}, we get
$$
\begin{aligned}
\lim_{t\to 0_+}&\int_{\Omega} |\sqrt{1+e(t)}-\sqrt{1+\tilde{e}(\theta_0,b_0)}|^2\dx \\
&=\lim_{t\to 0_+}\int_{\Omega} 1+e(t) +1+\tilde{e}(\theta_0,b_0) - 2\sqrt{1+e(t)}\sqrt{1+\tilde{e}(\theta_0,b_0)}\dx \\
&\le \int_{\Omega} 1+\tilde{e}(\theta_0,b_0) +1+\tilde{e}(\theta_0,b_0) - 2\sqrt{1+\tilde{e}(\theta_0,b_0)}\sqrt{1+\tilde{e}(\theta_0,b_0)}\dx =0.
\end{aligned}
$$
Thus, 
we have
$$
\begin{aligned}
\lim_{t\to 0_+}&\int_{\Omega}|e(t)-\tilde{e}(\theta_0,b_0)| \dx \\
&=\lim_{t\to 0_+}\int_{\Omega}|(\sqrt{1+e(t)}-\sqrt{1+\tilde{e}(\theta_0,b_0)})(\sqrt{1+e(t)}+\sqrt{1+\tilde{e}(\theta_0,b_0)})|\dx\\
&\le C\lim_{t\to 0_+}\left(\int_{\Omega}|\sqrt{1+e(t)}-\sqrt{1+\tilde{e}(\theta_0,b_0)}|^2\dx \right)^{\frac12}=0,
\end{aligned}
$$
which is \eqref{energyinit}.

\subsection{The limit $k\to \infty$}
In this section we will denote the solution constructed in the previous section by $(\vv^k, \theta^k, b^k)$.
The final part of the proof is to remove the cut-off from the convective term in the momentum equation by passing to the limit $k \rightarrow \infty$. The key difficulty is then the fact that in the limit the velocity field will not be an admissible test function anymore and consequently we will not be able to justify the limit procedure in the term on the right-hand side of the temperature equation. To this end, we shall therefore first identify the evolution equation for the global energy. Thanks to the fact that we have assumed a slip boundary condition \eqref{bc2new}, following~\cite{bulcek.m.malek.j.ea:mathematical*1} we can introduce the pressure $p^k \in L^2(0,T; L^2_0(\Omega))$, so that
\begin{equation}
\begin{split}\label{Galerkin:vvepspres}
&\langle \partial_t\vv^k, \vw\rangle + \int_{\Omega}(2\nu(\theta^k,b^k)\mathbb{D}(\vv^k)-G_k(|\vv^k|)\vv^k\otimes \vv^k ):\Grad \vw\dx+ \int_{\partial \Omega} \bs^k\cdot \vw \dS\\
&\qquad =\langle \bfb, \vw\rangle + \int_{\Omega} p^k \Div \vw \dx
\end{split}
\end{equation}
holds for all $\vw\in W^{1,2}_{\bn}$ and almost all $t\in (0,T)$.
Next, setting $\vw:=\vv^k u$, where $u\in W^{1,\infty}(\Omega)$ and adding the result to \eqref{Galerkin:en}, we obtain
\begin{equation}\label{Galerkin:Ek}
\begin{split}
& \langle \partial_t E^k, u\rangle  + \int_{\Omega}\left( \kappa(\theta^k,b^k)\Grad \theta^k
-(E^k+p^k)\vv\right) \cdot \Grad u\dx +\int_{\partial \Omega} \bs^k\cdot \vv^k u \dS\\
&\quad +\int_{\Omega} \alpha(\theta^k,b^k) \tilde{e}_1(\theta^k) \psi_2'(b^k)\Grad b^k  \cdot \Grad u \dx=\langle \bfb, \vv^k u\rangle.
\end{split}
\end{equation}
where $E^k:=\frac12 |\vv^k|^2 + e^k$. In addition, we recall the equation for $b^k$, see \eqref{Galerkin:beps},
\begin{equation}\label{Galerkin:bk}
\begin{split}
 \langle \partial_t b^k, u \rangle  + \int_{\Omega}\left(\alpha(\theta^k,b^k) \Grad b - b\vv \right)\cdot \Grad u +h(\theta^k,b^k)u \dx=0
\end{split}
\end{equation}
for all $u\in W^{1,2}(\Omega)$ and almost all $t\in (0,T)$.
 Thus, this equation is now prepared for the limiting procedure.

As is usual in compactness arguments, we begin by collecting the estimates that are uniform in $k$. It follows from weak lower semicontinuity, Fatou's lemma and \eqref{E1n}, \eqref{inftybn}, \eqref{E2n}, \eqref{E3n}, \eqref{E3nb}, \eqref{entropy5}, \eqref{entropy7}, \eqref{entropy8}, \eqref{E4theta}, \eqref{E4e}, \eqref{E4egrad}, \eqref{E4thetagrad} and \eqref{Etnn} that, for all $\delta\in (0,1)$,
\begin{equation} \label{E1k}
\begin{split}
&\sup_{t\in (0,T)} \left(\|\vv^{k}(t)\|_{2}^2+ \|e^k(t)\|_1 + \|\theta^k(t)\|_1 + \|\ln \theta^k(t)\|_1 \right) \\
&\quad  +\int_0^T \|\vv^{k}\|^2_{1,2}+ \|b^{k}\|_{1,2}^2 + \|\partial_t b^{k} \|_{(W^{1,2}(\Omega))^*}^2 +\|\partial_t e^k\|_{(W^{1,10+\delta}(\Omega))^*}+\|\vv^k\|_{\frac{10}{3}}^{\frac{10}{3}} + \|\ln \theta^k\|_{1,2}^2 \dt\\
&\quad + \int_0^T \|(1+e^k)^{\frac{1-\delta}{2}}\|_{1,2}^2+\|(1+\theta^k)^{\frac{1-\delta}{2}}\|_{1,2}^2 + \|\bs^k\|^2_{L^2(\partial \Omega)^3} \dt\\
&\quad + \int_0^T\|\theta^k\|^{\frac{5}{4}-\delta}_{1,\frac{5}{4}-\delta}+\|e^k\|^{\frac{5}{4}-\delta}_{1,\frac{5}{4}-\delta} +\|\theta^k\|^{\frac{5}{3}-\delta}_{\frac{5}{3}-\delta}+\|e^k\|^{\frac{5}{3}-\delta}_{\frac{5}{3}-\delta}\dt \le C(\delta),
\end{split}
\end{equation}
and
\begin{equation}\label{inftybk}
b_{\min}\le b^{k}(t,x)\le b_{\max} \quad \textrm{for almost all }(t,x)\in (0,T)\times \Omega.
\end{equation}
In addition, using the interpolation inequality
$$
\|u\|_4\le C\|u\|^{\frac{1}{4}}_2 \|u\|^{\frac{3}{4}}_{1,2},
$$
we have that
\begin{equation}\label{E2k}
\int_0^T \|\vv^k\|_{4}^{\frac83}\dt\le C.
\end{equation}
Then, by proceeding as in \cite{bulcek.m.malek.j.ea:mathematical*1}, we can also deduce\footnote{This is done by inserting $\vw:=\nabla_x u$ in \eqref{Galerkin:vvepspres}, where $u$ solves $\Delta u = p^k$ subject to a homogeneous Neumann boundary condition.} that
\begin{equation}\label{E2pk}
\int_0^T \|p^k\|_{2}^{\frac43}\dt\le C.
\end{equation}
It then follows from \eqref{Galerkin:vvepspres}, \eqref{Galerkin:Ek}, \eqref{E1k}, \eqref{E2k} and \eqref{E2pk} that, for all $\delta \in (0,1)$,
\begin{equation}\label{E2tk}
\int_0^T \|\partial_t \vv^k\|_{(W^{1,2}_{\bn})^*}^{\frac43}+ \|\partial_t E^k\|^{\frac{10+\delta}{9+\delta}}_{(W^{1,10+\delta}(\Omega))^*}\dt\le C(\delta).
\end{equation}

Now, we can repeat step by step the procedure from the previous section, i.e., we can find weakly and strongly converging subsequences, such that the limiting objects $(\vv,b,\theta,e,\bs)$ satisfy \eqref{FS1}--\eqref{FS11}, \eqref{constitutive} and \eqref{folbound} and we can also obtain \eqref{eq:vv}--\eqref{eq:E} directly by letting $k\to \infty$ in \eqref{Galerkin:vvepspres}--\eqref{Galerkin:bk}. The inequalities \eqref{eq:entropy}--\eqref{eq:ienergy}, can be obtained from \eqref{Galerkin:en} and using the same scheme as suggested in \eqref{entropy1}--\eqref{entropy4}. The identification \eqref{slip} follows by Minty's method and the following inequality, which is based on the strong convergence of $\vv^k$ in $L^2(0,T; L^2(\partial \Omega)^3)$:
$$
\begin{aligned}
0&\le \lim_{k\to \infty} \int_0^T \int_{\partial \Omega} (\frac{(|\bs^k|-s_*)_+}{|\bs^k|}\bs^k - \frac{(|\vw|-s_*)_+}{|\vw|}\vw) \cdot (\bs^k - \vw)\dS \dt\\
&=\lim_{k\to \infty} \int_0^T \int_{\partial \Omega} (\gamma_* \vv^k - \frac{(|\vw|-s_*)_+}{|\vw|}\vw) \cdot (\bs^k - \vw)\dS \dt\\
&=\int_0^T \int_{\partial \Omega} (\gamma_* \vv - \frac{(|\vw|-s_*)_+}{|\vw|}\vw) \cdot (\bs - \vw)\dS \dt
\end{aligned}
$$
for an arbitrary $\vw\in L^2(0,T; L^2(\partial \Omega)^3)$.

It remains to show \eqref{attint}. In the case of $b$ the proof of the asserted attainment of the initial datum is quite standard and we shall therefore omit its proof here. For $\vv$, the argument is exactly the same as in the case of the incompressible Navier--Stokes equations. First one can show that
$$
\vv(t)\rightharpoonup\vv_0 \quad \textrm{weakly in }L^2(\Omega)^3.
$$
Next, by weak lower semicontinuity, one can also conclude that
$$
\|\vv(\tau)\|_2^2 \le \|\vv_0\|_2^2 + \int_0^{\tau} \langle \bfb, \vv\rangle \dt \quad \implies \quad \limsup_{t\to 0_+}\|\vv(t)\|_2^2 \le \|\vv_0\|_2^2.
$$
These two pieces of information then lead to \eqref{attint} for $\vv$.

Concerning the attainment of the initial datum by $e$, we recall that since $\partial_t e$ is a measure, it is meaningful to consider pointwise values $e(t_+)$ and $e(t_-)$ and we can therefore also meaningfully consider $\lim_{t\to 0_+}e(t)$. Similarly, as before, we can conclude that for every nonnegative $u\in L^2(\Omega)$ we have
\begin{equation}
\label{ling}
\liminf_{t\to 0_+}\int_{\Omega} u \sqrt{1+e(t)}\dx \ge \int_{\Omega} u \sqrt{1+\tilde{e}(\theta_0, b_0)}\dx.
\end{equation}
Although this does not yet yield the attainment of the initial datum by $e$ stated in \eqref{eq:ienergy}, we can use \eqref{eq:E} to deduce that
$$
\limsup_{t\to 0_+} \int_{\Omega} \frac12 |\vv(t)|^2 + e(t) \dx \le \int_{\Omega} \frac12|\vv_0|^2 + \tilde{e}(\theta_0,b_0)\dx.
$$
Since we already know that $\vv_0$ is attained strongly, it follows that
\begin{equation}\label{lang}
\limsup_{t\to 0_+} \|\sqrt{1 + e(t)}\|_2^2  \le \|\sqrt{1 + \tilde{e}(\theta_0,b_0)}\|_2^2.
\end{equation}
By taking $u \equiv 1$ in \eqref{ling} and combining the resulting inequality with \eqref{lang} gives that $\sqrt{1 + e(t)}\to \sqrt{1 + \tilde{e}(\theta_0,b_0)}$ strongly in $L^2(\Omega)$ and the claim \eqref{attint} for $e$ then directly follows.

Finally, the attainment of the initial datum by $\theta$ in $L^1(\Omega)$ asserted in \eqref{attint} is a consequence
of the attainment of the initial datum by $e$ in $L^1(\Omega)$ and by $b$ in $L^2(\Omega)$ (and therefore also in $L^1(\Omega)$), and the fact that
the mapping $(e,b) \in \mathbb{R}_{\geq 0} \times [b_{\min},b_{\max}] \mapsto \theta = \tilde\theta(e,b) \in \mathbb{R}_{\geq 0}$ is globally Lipschitz with respect to $e\in \mathbb{R}_{\geq 0}$ uniformly in $b \in [b_{\min},b_{\max}]$, and with respect to $b\in [b_{\min},b_{\max}]$ uniformly in $e
\in \mathbb{R}_{\geq 0}$. To prove the asserted Lipschitz continuity of $\theta = \tilde\theta(e,b)$, note that by (implicit) differentiation of \eqref{Aetae}$_2$ with respect to $e$ and $b$ we have that
\[ \frac{\partial\theta}{\partial e}(e,b) = \frac{1}{-\theta \tilde\psi_0''(\theta) -  \theta\tilde\psi_1''(\theta)\tilde\psi_2(b)}\quad \mbox{and}\quad \frac{\partial\theta}{\partial b}(e,b) = \frac{(\theta \tilde\psi_1'(\theta) - \tilde\psi_1(\theta))\tilde\psi_2'(b)}{-\theta \tilde\psi_0''(\theta) -  \theta\tilde\psi_1''(\theta)\tilde\psi_2(b)} \quad \mbox{with $\theta = \tilde\theta(e,b)$}.
\]
As $\tilde\psi_2(b) \geq 0$ for all $b>0$, thanks to \textbf{(A1)} and \textbf{(A4)} the denominators are nonnegative and are, in fact, bounded below by $C_1>0$ (with $C_1$ as in \textbf{(A4)}) for all $\theta \geq 0$ and all $b > 0$. Concerning the numerator of the second fraction, the first factor is $-\tilde{e}_1(\theta)$, which by \eqref{ule1} is bounded in absolute value by $\tilde{\psi}_1(0) + C_2$ (with $C_2$ as in \textbf{(A4)}). The numerator in the second fraction is therefore bounded in absolute value by a positive constant, uniformly in $\theta \in \mathbb{R}_{\geq 0}$ and $b \in [b_{\min},b_{\max}]$,
and therefore also uniformly in $e \in \mathbb{R}_{\geq 0}$ and $b \in [b_{\min},b_{\max}]$.
That completes the proof of Theorem \ref{main:T} asserting the existence of global-in-time weak solutions to the problem.

\section{Conclusions}
\label{sec:conclusion}

Viscoelastic rate-type models are ubiquitous in the modelling of flows of complex fluids. As such they have been of interest from the mathematical point of view as well; see for example~\cite{guillope.c.saut.jc:existence} for a seminal contribution in this direction. However, the overwhelming majority of mathematical research on the subject has been concerned with isothermal flows. This is unsatisfactory from the physical point of view, since material parameters for viscoelastic rate-type fluids are known to be sensitive to temperature-variations. In fact, accounting for temperature-variations is recognised as a major modelling problem in rheology; see for example~\cite{tanner.ri:changing}. Consequently, the mathematical analysis of models of viscoelastic fluids should also focus on the qualitative properties of equations describing the corresponding coupled thermo-mechanical processes. This has been the main objective herein.

While the model we have considered is a very simple from the physical point of view, the proof of solvability of the corresponding governing equations has been a challenging mathematical problem. A key aspect of the mathematical analysis presented in the paper is that it exploits the thermodynamic foundations of the model. We have also introduced a novel approximation scheme, which enabled us to prove the existence of~\emph{large-data global-in-time weak solutions} to the corresponding governing equations for~\emph{coupled thermo-mechanical processes}. This is, to the best of our knowledge, the first result of this type regarding viscoelastic rate-type fluids.

%

\bigskip

\centerline{\bf Acknowledgements}

\medskip

The authors acknowledge support of the project 18-12719S financed
by the Czech Science Foundation. Miroslav Bul\'{\i}\v{c}ek, Josef M\'{a}lek and V\'{i}t Pr\r{u}\v{s}a are members of the Ne\v{c}as Center for Mathematical Modeling.
Endre S\"uli is a member of the Oxford Centre for Nonlinear Partial Differential Equations (OxPDE).

\bigskip

\def\cprime{$'$} \def\ocirc#1{\ifmmode\setbox0=\hbox{$#1$}\dimen0=\ht0
  \advance\dimen0 by1pt\rlap{\hbox to\wd0{\hss\raise\dimen0
  \hbox{\hskip.2em$\scriptscriptstyle\circ$}\hss}}#1\else {\accent"17 #1}\fi}


\providecommand{\bysame}{\leavevmode\hbox to3em{\hrulefill}\thinspace}
\providecommand{\MR}{\relax\ifhmode\unskip\space\fi MR }
\providecommand{\MRhref}[2]{%
  \href{http://www.ams.org/mathscinet-getitem?mr=#1}{#2}
}
\providecommand{\href}[2]{#2}

\end{document}